\documentclass{article}

\usepackage{PRIMEarxiv}

\usepackage[utf8]{inputenc} 
\usepackage[T1]{fontenc}    
\usepackage{url}            
\usepackage{booktabs}       
\usepackage{amsfonts}       
\usepackage{nicefrac}       
\usepackage{microtype}      
\usepackage{lipsum}
\usepackage{fancyhdr}       
\usepackage{graphicx}       
\graphicspath{{figs/}}
\usepackage{epsfig}
\usepackage{graphicx}
\usepackage{wrapfig}
\usepackage{psfrag}
\usepackage{substr}
\usepackage[tight]{minitoc}
\usepackage{subfigure}
\usepackage{longtable}
\usepackage[all]{xy}
\usepackage{xcite}
\usepackage{float}
\usepackage{xr}
\usepackage{framed}

\usepackage{amsmath,amssymb,mathtools}
\usepackage{graphicx}
\usepackage{booktabs}
\usepackage[ruled,vlined,linesnumbered]{algorithm2e}
\usepackage{tikz}
\usetikzlibrary{arrows.meta,shapes.geometric,positioning,calc}

\providecommand{\R}{}\renewcommand{\R}{\mathbb{R}}
\providecommand{\E}{}\renewcommand{\E}{\mathbb{E}}
\providecommand{\Prob}{}\renewcommand{\Prob}{\mathbb{P}}
\providecommand{\ones}{}\renewcommand{\ones}{\mathbf{1}}
\providecommand{\norm}[1]{}\renewcommand{\norm}[1]{\left\lVert #1\right\rVert}
\providecommand{\inner}[2]{}\renewcommand{\inner}[2]{\left\langle #1,\,#2\right\rangle}
\providecommand{\spec}{}\renewcommand{\spec}{\operatorname{spec}}
\providecommand{\diag}{}\renewcommand{\diag}{\operatorname{diag}}
\providecommand{\proj}{}\renewcommand{\proj}{\Pi}
\providecommand{\Lap}{}\renewcommand{\Lap}{L}
\providecommand{\eps}{}\renewcommand{\eps}{\epsilon}
\providecommand{\ChebT}{}\renewcommand{\ChebT}{T}
\providecommand{\ChebU}{}\renewcommand{\ChebU}{U}

\providecommand{\argmin}{}\renewcommand{\argmin}{\operatorname*{arg\,min}}
\newcommand{\bLap}{\boldsymbol{L}}
\newcommand{\bproj}{\boldsymbol{\Pi}}
\newcommand{\bprojperp}{\boldsymbol{\Pi}_{\!\perp}}

\graphicspath{{figs/}}

\usepackage{amsthm}
\newtheorem{theorem}{Theorem}[section]
\newtheorem{lemma}[theorem]{Lemma}
\newtheorem{proposition}[theorem]{Proposition}
\newtheorem{corollary}[theorem]{Corollary}
\theoremstyle{definition}
\newtheorem{definition}[theorem]{Definition}
\newtheorem{assumption}[theorem]{Assumption}

\theoremstyle{remark}
\newtheorem{remark}[theorem]{Remark}

\newcommand{\WAVE}{\textsc{WAVE}}
\newcommand{\ChebyshevWindow}{\textsc{CHEBYSHEV-WINDOW}}

\newcommand{\WAVESTM}{\textsc{WAVE-STM}}

\usepackage{array}
\usepackage{tabularx}

\newcolumntype{L}[1]{>{\raggedright\arraybackslash}p{#1}}
\newcolumntype{C}[1]{>{\centering\arraybackslash}p{#1}}
\newcolumntype{Y}{>{\raggedright\arraybackslash}X}

\usepackage[table]{xcolor}
\usepackage{array}
\usepackage{tabularx}
\usepackage{enumitem}
\usepackage{placeins}

\usepackage[colorlinks=true,linkcolor=blue!55!black,citecolor=blue!55!black,urlcolor=blue!55!black]{hyperref}

\newif\ifshowrevisioncolors
\showrevisioncolorstrue

  {\begingroup\ifshowrevisioncolors\color{magenta}\fi}%
  {\endgroup}
  {\begingroup\ifshowrevisioncolors\color{blue}\fi}%
  {\endgroup}

\definecolor{resultviolet}{RGB}{239,234,248}

\newcolumntype{Y}{>{\raggedright\arraybackslash}X}
\newcolumntype{C}[1]{>{\centering\arraybackslash}p{#1}}

\newcommand{\Lclass}{\mathbb{L}}
\newcommand{\Response}{\mathcal{R}}
\usepackage{natbib}

\title{Accelerated Consensus and Decentralized Optimization over
Slowly Varying Networks}

\author{
Darina~Dvinskikh\\
 HSE University \\
\texttt{dmdvinskikh@hse.ru}
\And
Anton Novitskii\\
Trusted AI Research Center RAS\\
MIRAI \\
\texttt{toxandrnovitskii17@gmail.com} 
\And
Demyan Yarmoshik \\
MIRAI\\
\texttt{yarmoshik.d@miriai.org}
\And
 Alexander Rogozin\\
MIRAI\\
\texttt{rogozin.a@miriai.org}
 \And
 Vladislav Matyukhin \\
Innopolis University \\
MIRAI \\
\texttt{vladmatyukh@gmail.com}
 \And
 Alexander Gasnikov\\
Innopolis University \\
MIRAI \\
 \texttt{gasnikov@yandex.ru}
}

\begin{document}

\maketitle

\begin{abstract}
We study average consensus over slowly time-varying networks and its application to decentralized optimization. We introduce WAVE, a windowed Chebyshev method that limits the accumulated effect of network variation by restarting the recurrence after finite windows. If $\chi$ bounds the network condition number and successive communication operators satisfy $\|L_{k+1}-L_k\|\leq\beta$, WAVE reaches $\epsilon$-consensus in $O((\sqrt{\chi}+\min\lbrace\beta\chi^2,\chi\rbrace)\ln(e/\epsilon))$ communication rounds. This rate matches the fixed-network $\sqrt{\chi}$ dependence when $\beta=O(\chi^{-3/2})$ and smoothly interpolates across the full range of network variation. For piecewise-constant networks where each change is detected when it occurs and consecutive changes are at least $\tau$ rounds apart, WAVE reaches $\epsilon$-consensus in $O((\sqrt{\chi}+\chi/\tau)\ln(e/\epsilon))$ communication rounds. Finally, we use WAVE as the consensus step in an accelerated decentralized optimization method for $\alpha$-smooth convex local objectives with a $\mu$-strongly convex average. For sufficiently slow network variation, every agent obtains an $\epsilon$-solution to the global optimization problem after $\widetilde O(\sqrt{\alpha/\mu}\,\sqrt{\chi})$ communication rounds, matching the fixed-network dependence.
\end{abstract}

\section{Introduction}\label{sec:intro}
The average consensus problem requires a network of agents to compute the
mean of their private initial vectors using only neighbor-to-neighbor
communication. It is a basic component of distributed computation,
multi-agent control, distributed estimation, and decentralized
optimization
\citep{DeGroot1974,Tsitsiklis1984,JadbabaieLinMorse2003,
OlfatiSaberMurray2004,Moreau2005,NedicOlshevskyRabbat2018,
GorbunovSurvey2022,RogozinEncyclopedia2023}. Related distributed
optimization problems arise in estimation, control, resource allocation,
and machine learning
\citep{RabbatNowak2004,ForeroCanoGiannakis2010,
BeckNedicOzdaglarTeboulle2014,GiselssonEtAl2013,
KonecnyEtAl2016,McMahanEtAl2017}. In both settings, communication is often
the dominant cost, making the number of communication rounds a central
complexity measure.

For a fixed normalized gossip operator with condition number $\chi$ on
the disagreement subspace, standard gossip requires
$O(\chi\ln(e/\eps))$ communication rounds. Classical Chebyshev
acceleration improves this bound to
$O(\sqrt{\chi}\ln(e/\eps))$
\citep{XiaoBoyd2004,GolubVarga1961,GutknechtRollin2002,
NguyenRJND2024}. For fixed networks, the same $\sqrt{\chi}$ network
dependence appears in the optimal communication complexity of smooth
strongly convex decentralized optimization
\citep{ScamanBBLM2017}. For arbitrarily time-varying networks, however,
the worst-case network dependence is linear in $\chi$
\citep{Kovalev2021ADOM,Kovalev2021LB}. This gap motivates a fundamental
question: how does communication complexity depend on the rate of network
variation, and when can the fixed-network square-root acceleration be
retained?

\paragraph{Network models and approach.}
We consider two complementary models of network variation. In the
\emph{metric-drift} model, the normalized gossip operator may change at
every communication round, while consecutive normalized gossip matrices satisfy
$\norm{\Lap_{k+1}-\Lap_k}\leq\beta$. In the
\emph{piecewise-constant} model, the operator may change abruptly but
remains fixed between change points, with consecutive changes at least
$\tau$ communication rounds apart.
For both models, we introduce \WAVE{} (\textbf{W}indowed
\textbf{A}cceleration for \textbf{V}arying n\textbf{E}tworks). It applies
the classical Chebyshev recurrence on the full spectral interval
$[\chi^{-1},1]$ and restarts after finite windows. The metric-drift
schedule chooses window lengths using $\beta$ and $\chi$, while the
piecewise-constant schedule restarts at changes reported by an ideal
external detector.

\paragraph{Contributions.}
Our main contribution is a consensus method (\WAVE{}) whose communication complexity
adapts to the magnitude and frequency of network changes. We consider two
complementary settings: small changes may occur at every communication round,
or the network may change arbitrarily at detected change points and remain
fixed between them. In both settings, our bounds interpolate between the
fixed-network $\sqrt\chi$ dependence and the worst-case $\chi$ dependence.
Below, $\widetilde O$ suppresses logarithmic factors.

\begin{itemize}

\item
\emph{Consensus.}
We prove that \WAVE{} reaches $\eps$-consensus in
\[
\widetilde O\!\left(\sqrt\chi+\min\{\beta\chi^2,\chi\}\right)
\quad\text{and}\quad
\widetilde O\!\left(\sqrt\chi+\frac{\chi}{\tau}\right)
\]
communication rounds under metric drift and detected piecewise-constant
changes, respectively. The bounds interpolate between the accelerated
$\sqrt\chi$ and robust $\chi$ dependences, recovering the former when
$\beta=O(\chi^{-3/2})$ or $\tau=\Omega(\sqrt\chi)$. Thus acceleration
persists under small changes at every round or arbitrary detected jumps
that are sufficiently separated, without requiring a persistent connected
subgraph. An exact variation-of-constants argument controls the
perturbations caused by noncommuting communication operators within a
Chebyshev window and determines its length. The piecewise-constant
schedule adapts without knowing $\tau$. When
$\tau\ge\sqrt\chi\ln(2n)+2$, a flooding-and-rollback protocol implements
change detection inside the network with the same communication order
plus $O(\sqrt\chi\ln n)$ rounds.

\item
\emph{Decentralized optimization.}
For $\alpha$-smooth, $\mu$-strongly convex local objectives, let
$\kappa:=\alpha/\mu$. Using \WAVE{} for mixing in accelerated gradient
tracking \citep{LiLin2024JMLR} gives
$O(\sqrt\kappa\ln(1/\eps))$ gradient evaluations and multiplies each
consensus network factor above by $O(\sqrt\kappa\ln(1/\eps))$.
In the accelerated regimes this yields
$O(\sqrt{\kappa\chi}\ln(1/\eps))$, with the optimal fixed-network
$\sqrt{\kappa\chi}$ dependence and a single accuracy logarithm
\citep{ScamanBBLM2017}.
When only the average objective is strongly convex, \WAVESTM{} uses
restarted Similar Triangles iterations \citep{Nesterov2004} and gives
the same network factors with $\widetilde O(\sqrt\kappa)$ gradient
evaluations and a further accuracy logarithm in communication.

\end{itemize}

\paragraph{Related work and comparison.}
For arbitrary time variation, ADOM and ADOM+ attain the optimal worst-case
network dependence \citep{Kovalev2021ADOM,Kovalev2021LB}. Subgradient-push,
gradient-tracking, and primal-dual methods provide convergence under broad
connectivity assumptions
\citep{NedicOlshevsky2015,NedicOlshevskyShi2017,MarosJalden2018,
LiLin2024JMLR,RogozinLukoshkin2021}. These guarantees are designed for
worst-case variation and do not improve continuously as the network changes
more slowly.
Existing acceleration results use different restrictions on network variation.
\citet{RogozinUribe2019} allow arbitrarily large changes and retain the
fixed-network rate when the number of
changing rounds among the first $N$ iterations is
$O(N/[\sqrt{\kappa\chi}\log(\kappa\chi)])$, so the admissible fraction of
changing rounds decreases with $\kappa\chi$. Our metric-drift model allows
changes at every round, while our detected piecewise-constant model recovers
the accelerated network factor when $\tau=\Omega(\sqrt\chi)$, independently
of the objective condition number $\kappa$.
\citet{MetelevICML2023}
require all communication graphs to share a connected spanning subgraph,
and their network factor is controlled by the spectral gap of this
persistent subgraph. \citet{MontijanoMontijanoSagues2013} give sufficient
conditions for convergence of a Chebyshev
iteration over time-varying networks, but do not derive a communication
bound that adapts to the rate of variation. \citet{MetelevCMS2024} consider
a stationary, uniformly geometrically mixing Markov chain of networks and
obtain bounds with additional mixing-dependent terms. They also show that,
for first-order decentralized
optimization, removing one edge and adding another at every round can
retain a linear dependence on $\chi$. \citet{LiLin2024JMLR} use repeated
gossip for their time-varying accelerated method; \WAVE{} replaces this
mixing step when variation is slow.
Our deterministic bounds
interpolate between the fixed-network and arbitrary-variation regimes
through the operator drift or the time between detected changes.
To the best of our knowledge, our metric-drift result is the first
deterministic average-consensus complexity bound that gives this
interpolation as an explicit function of the one-step operator-norm change.
For time-varying quadratics, \citet{PasechnyukVilenskyTakac2026} study the
first-order sensitivity of Chebyshev realizations to Hessian perturbations.
Here we bound the full perturbation accumulated in a window and use it to
obtain consensus communication guarantees.

\paragraph{Organization.}
Section~\ref{sec:main-results} formulates the consensus problem and
introduces \WAVE{}. Section~\ref{subsec:main-consensus-results} states the
main consensus guarantees. Sections~\ref{sec:metric-analysis} and
\ref{sec:dwell} analyze the metric-drift and piecewise-constant models,
respectively. Section~\ref{sec:optimization} presents the application to
decentralized optimization. Section~\ref{sec:experiments} reports the
numerical results, and Section~\ref{sec:conclusion} concludes the paper.
The supplementary material contains complete proofs, implementation
details, the extension to memoryless changes, and additional experiments.

\section{Problem statement and algorithm}
\label{sec:main-results}

This section formulates the consensus problem, specifies the two models of
network variation, and introduces \WAVE{}. We first recall the classical
Chebyshev recurrence for a fixed network, then explain how finite-window
restarts extend it to changing communication operators. 

\paragraph{Notation.}
The norm $\norm{\cdot}$ is Euclidean for vectors and spectral for
matrices, and $\spec(A)$ denotes the set of eigenvalues of a square matrix $A$. For
any positive integer $p$, $\ones_p$ is the vector of $p$ ones and $I_p$ is
the $p\times p$ identity matrix; $\otimes$ denotes the Kronecker product.
The complete notation is collected in Table~\ref{tab:notation}
(Appendix~\ref{app:notation}).

\subsection{Consensus and network model}
\label{subsec:consensus-model}

We consider a network of agents represented at each communication round
$k$ by an undirected connected graph $\mathcal G_k=(V,\mathcal E_k)$, where
$V=\{1,\ldots,n\}$ is the set of agents and $\mathcal E_k$ is the set of edges. Agent $i$ holds $u_k^{(i)}\in\R^d$ and may
exchange vectors only with its current neighbors. The goal is for every
agent to compute the average
$\bar u_0:=n^{-1}\sum_{i=1}^n u_0^{(i)}$ of the initial vectors. 

Standard decentralized optimization formulations model local
communication by a symmetric positive semidefinite gossip matrix whose
sparsity agrees with the communication graph
\citep{ScamanBBLM2017,NguyenRJND2024}. 
We measure network variation by the
distance between successive communication operators, which changes under
rescaling. We therefore place all communication operators on the same
scale.

\begin{definition}[Normalized gossip matrix]
\label{def:gossip-matrix}
A matrix $\Lap\in\R^{n\times n}$ associated with a connected graph
$G=(V,\mathcal E)$ is a normalized gossip matrix if
$\Lap=\Lap^\top\succeq0$,
$\ker\Lap=\operatorname{span}\{\ones_n\}$,
$\Lap_{ij}=0$ whenever $i\neq j$ and $\{i,j\}\notin\mathcal E$, and
$\norm{\Lap}\leq1$.
\end{definition}

Throughout the paper, $\Lap_k$ denotes a normalized gossip matrix
associated with $\mathcal G_k$, with the same normalization used at every round.
Concretely, the unnormalized operators are divided by one common upper
bound on their norms rather than rescaled independently at each round.
Consequently, $\norm{\Lap_{k+1}-\Lap_k}$ compares operators expressed on
the same scale. The kernel condition identifies the consensus subspace,
while the sparsity condition makes multiplication by $\Lap_k\otimes I_d$
implementable using one neighbor exchange. We therefore count one such
multiplication as one communication round.

The Chebyshev coefficients must be valid for every operator encountered by
the algorithm. We therefore impose one uniform spectral bound on the whole
operator sequence.

\begin{assumption}[Uniform spectral gap]
\label{ass:uniform-spectrum}
There is a known $\chi\geq1$ such that
\[
 \spec(\Lap_k)\setminus\{0\}
 =\spec\!\left(\Lap_k|_{\ones_n^\perp}\right)
 \subseteq[\chi^{-1},1]
 \qquad\text{for every }k\geq0.
\]
\end{assumption}
Thus $\chi$ is a uniform upper bound on the condition number of the
communication operators on $\ones_n^\perp$.

Since $\chi$ is a conservative upper bound, any valid bound
$\chi<4$ may be replaced by $4$. We therefore assume $\chi\geq4$
throughout. For $1\leq\chi<4$, standard gossip already reaches
$\eps$-consensus in $O(\ln(e/\eps))$ communication rounds under
arbitrary network variation.

\subsection{Network variation}
\label{subsec:variation-models}

We consider two complementary restrictions on network variation. Metric
drift bounds the size of each change but allows the network to change at
every round. The piecewise-constant model allows arbitrarily large changes
but requires consecutive changes to be separated in time.

\begin{assumption}[Metric drift]
\label{ass:metric-drift}
For a known $\beta\in[0,1]$,
$\norm{\Lap_{k+1}-\Lap_k}\leq\beta$ for every $k\geq0$.
\end{assumption}

The parameter $\beta$ is available to the metric-drift scheduler. No
commutativity is assumed between consecutive operators.

\begin{assumption}[Piecewise-constant network with minimum spacing]
\label{ass:dwell-time}
There are change points $0=t_0<t_1<t_2<\cdots$ such that
$\Lap_k=\Lap_{t_i}$ for $t_i\leq k<t_{i+1}$ and
$t_{i+1}-t_i\geq\tau$, where $\tau\geq1$. A new matrix may depend on
the previous network history, and its distance from the preceding matrix
is unrestricted.
\end{assumption}

For piecewise-constant networks, an ideal external detector reports a
change at the corresponding round boundary, before the new operator is
used. The algorithm does not know $\tau$. The basic analysis treats reports as
available at no communication cost; Proposition~\ref{prop:flooding}
implements them by messages inside the network under a longer spacing condition.
\subsection{Chebyshev acceleration and \WAVE{}}
\label{subsec:wave-at-glance}

For the analysis, stack the local states as
$\boldsymbol u_k:=((u_k^{(1)})^\top,\ldots,(u_k^{(n)})^\top)^\top
\in\R^{nd}$ and define
$\boldsymbol e_k:=\boldsymbol u_k-\ones_n\otimes\bar u_0$.
To describe communication between these stacked states, we use
$\bLap_k=\Lap_k\otimes I_d$. 

\begin{definition}[$\eps$-consensus]
For $\eps\in(0,1)$, a stacked state $\boldsymbol u\in\R^{nd}$ is an
$\eps$-consensus solution relative to $\boldsymbol u_0$ if
$\norm{\boldsymbol u-\ones_n\otimes\bar u_0}
\leq
\eps\norm{\boldsymbol u_0-\ones_n\otimes\bar u_0}.$
\end{definition}

\paragraph{Fixed network.}
If $\bLap_k\equiv\bLap$, every mean-preserving method in the linear
operator model (Appendix~\ref{app:model-details}) that uses $m$
communication rounds returns $\boldsymbol e_m=p(\bLap)\boldsymbol e_0$
for a polynomial $p$ with $\deg p\le m$ and $p(0)=1$. By
Assumption~\ref{ass:uniform-spectrum}, it contracts the disagreement by
at most $\max_{\lambda\in[\chi^{-1},1]}|p(\lambda)|$. The best such
polynomial is built from the Chebyshev polynomials
$\ChebT_0(x)=1$, $\ChebT_1(x)=x$,
$\ChebT_{t+1}(x)=2x\ChebT_t(x)-\ChebT_{t-1}(x)$, which are bounded by one
on $[-1,1]$ and grow rapidly outside it. Let
\begin{equation}
\label{eq:chebyshev-residual}
 z_\chi(\lambda):=
 \frac{1+\chi^{-1}-2\lambda}{1-\chi^{-1}},
 \qquad z_0:=z_\chi(0),
 \qquad
 P_t^\chi(\lambda)
 :=\frac{\ChebT_t(z_\chi(\lambda))}{\ChebT_t(z_0)}.
\end{equation}
The affine map $z_\chi$ sends the admissible disagreement spectrum
$[\chi^{-1},1]$ to $[-1,1]$, while $z_0>1$. Hence $P_t^\chi(0)=1$, and
the residual $P_t^\chi$ is uniformly small on $[\chi^{-1},1]$.

\begin{proposition}[Classical Chebyshev residual]
\label{prop:cheb-minimax}
For $m\ge1$ and $\chi>1$, $P_m^\chi$ is the unique minimizer of
\begin{equation}\label{eq:minimax-value}
 \rho_m^\chi:=
 \min_{\deg p\le m,\ p(0)=1}\
 \max_{\lambda\in[\chi^{-1},1]}|p(\lambda)|,
 \qquad
 \rho_m^\chi
 =\frac1{\ChebT_m(z_0)}
 =\operatorname{sech}(m\theta_\chi)
 \le2e^{-2m/\sqrt\chi},
\end{equation}
where $\theta_\chi=2\operatorname{artanh}(\chi^{-1/2})$.
\end{proposition}

Thus a degree-$\Theta(\sqrt\chi)$ Chebyshev polynomial contracts the
disagreement by a constant factor, and its repetition yields the
fixed-network $O(\sqrt\chi\ln(1/\eps))$ complexity
\citep{GolubVarga1961,GutknechtRollin2002,ScamanBBLM2017}. \WAVE{} also
uses windows with $m\ll\sqrt\chi$, for which $2e^{-2m/\sqrt\chi}$ may
exceed one; the exact $\operatorname{sech}$ value still gives progress.

\begin{lemma}[Short-window progress]
\label{lem:window}
If $\chi\ge4$, $m\ge1$, and $m^2\le2\chi$, then
\begin{equation}\label{eq:window-contr}
 \rho_m^\chi\le\exp\left(-\frac{4m^2}{5\chi}\right).
\end{equation}
\end{lemma}

A window of length $m\le\sqrt\chi$ thus earns order $m/\chi$ of
logarithmic contraction per round, between the gossip rate $1/\chi$
($m=1$) and the accelerated rate $1/\sqrt\chi$ ($m\asymp\sqrt\chi$).
The complexity of \WAVE{} is therefore governed by the window length that
network variation admits. Appendix~\ref{app:chebyshev} contains the proofs
and an illustration of the residuals.

\paragraph{Chebyshev window.}
Dividing the Chebyshev recurrence by $\ChebT_{t+1}(z_0)$ gives
$P_{t+1}^\chi(\lambda)=a_tz_\chi(\lambda)P_t^\chi(\lambda)
-c_tP_{t-1}^\chi(\lambda)$, where
$a_t:=2\ChebT_t(z_0)/\ChebT_{t+1}(z_0)$ and
$c_t:=\ChebT_{t-1}(z_0)/\ChebT_{t+1}(z_0)$ are computed without forming
the growing values $\ChebT_t(z_0)$:
\begin{equation}
\label{eq:cheb-coeffs-at-glance}
 a_0=\frac{2}{z_0},\qquad
 a_t=\frac{4}{4z_0-a_{t-1}},\qquad
 c_t=\frac{a_{t-1}a_t}{4}\quad(t\geq1).
\end{equation}
The coefficients do not depend on the operator, and each step uses one
multiplication by it, that is, one communication round. The recurrence
can therefore use the operator of the current round.

\begin{definition}[Chebyshev window]\label{def:chebyshev_window}
For $\bLap=\Lap\otimes I_d$, let
$z_\chi(\bLap):=((1+\chi^{-1})I_{nd}-2\bLap)/(1-\chi^{-1})$. For given window length $s\geq1$ and starting round $k$, initialize
$\boldsymbol v_0=\boldsymbol u\in\R^{nd}$ and compute
\begin{equation}
\label{eq:moving-window-recurrence}
 \boldsymbol v_1
 =z_0^{-1}z_\chi(\bLap_k)\boldsymbol v_0,
 \qquad
 \boldsymbol v_{t+1}
 =a_tz_\chi(\bLap_{k+t})\boldsymbol v_t
   -c_t\boldsymbol v_{t-1},
 \quad t=1,\ldots,s-1.
\end{equation}
We call the resulting map
$\mathcal C_{k,s}:\R^{nd}\to\R^{nd}$,
$\mathcal C_{k,s}(\boldsymbol u):=\boldsymbol v_s$, a Chebyshev window,
and set $\mathcal C_{k,0}(\boldsymbol u):=\boldsymbol u$.
\end{definition}

\begin{proposition}[Chebyshev representation on an unchanged window]
\label{prop:fixed-window-representation}
Suppose that $\bLap_{k+t}=\bLap$ for $t=0,\ldots,s-1$. Then
$\boldsymbol v_t=P_t^\chi(\bLap)\boldsymbol u$ for every
$t=0,\ldots,s$. In particular,
$\mathcal C_{k,s}(\boldsymbol u)=P_s^\chi(\bLap)\boldsymbol u$.
\end{proposition}

Hence a window of $h$ rounds on a single operator contracts the
disagreement by $\rho_h^\chi=1/\ChebT_h(z_0)$, whether or not it is
interrupted. When the operator changes inside a window, the recurrence is
no longer a polynomial of one matrix because consecutive operators need
not commute, and its deviation from $P_h^\chi$ grows with $h$. \WAVE{}
limits the window length and restarts: the next window starts at $t=0$
from the current state alone, discarding $\boldsymbol v_{s-1}$.

\paragraph{Window lengths.}
Under metric drift, every window has length
\begin{equation}
\label{eq:scale-of-beta}
 \widetilde m_\beta
 :=\min\!\left\{\sqrt\chi,(3\beta\chi)^{-1}\right\},
 \qquad
 m_\beta:=
 \max\!\left\{1,
 \left\lfloor\widetilde m_\beta\right\rfloor\right\},
\end{equation}
where $(3\beta\chi)^{-1}=+\infty$ for $\beta=0$. The second term keeps
the drift error of a window, $O(\beta m^3)$, below its progress, of order
$m^2/\chi$ (Section~\ref{sec:metric-analysis}). For a piecewise-constant
network, $\tau$ is unknown, and the length adapts by doubling: it starts
at one, doubles after every unchanged window up to
$m_{\max}:=\lfloor\sqrt\chi\rfloor$, and resets to one after a detected
change. Long unchanged intervals thus reach the accelerated length, while
frequent changes keep the windows short. The detector reports a change before round $r$ when
$\Lap_r\neq\Lap_{r-1}$, once and before $\Lap_r$ is used. The counter $q$
below records the certified logarithmic contraction: a drift window of
length $m_\beta$ earns $m_\beta^2/(5\chi)$
(Theorem~\ref{thm:metric-window}), and a piecewise window of $h$ executed
rounds earns $\ln\ChebT_h(z_0)$.

\begin{algorithm}[H]
\caption{\WAVE$(\boldsymbol u,k;\delta, \mathcal M)$}
\label{alg:wave}
\KwIn{$\boldsymbol u\in\R^{nd}$, next round $k$, target
$\delta\in(0,1)$, model
$\mathcal M\in\{\mathrm{drift},\mathrm{piecewise}\}$; known $\chi$ and,
for metric drift, $\beta$}
$q\leftarrow0$;
$m\leftarrow m_\beta$ for metric drift, and $m\leftarrow1$ for a
piecewise-constant network\;
\While{$q<\ln(1/\delta)$}{
\eIf{$\mathcal M=\mathrm{drift}$}{
 $(\boldsymbol u,k)\leftarrow(\mathcal C_{k,m}(\boldsymbol u),k+m)$;
 $q\leftarrow q+m^2/(5\chi)$\;
}{
 $k_0\leftarrow k$; execute the window until either $m$ rounds are
 completed or a change is reported before the next round\;
 $h\leftarrow$ number of executed rounds\;
 $(\boldsymbol u,k)\leftarrow
 (\mathcal C_{k_0,h}(\boldsymbol u),k_0+h)$;
 $q\leftarrow q+\ln\ChebT_h(z_0)$\;
 \eIf{a change was reported before round $k$}{
   $m\leftarrow1$\;
 }{
   $m\leftarrow\min\{2m,m_{\max}\}$\;
 }
}
}
\Return $(\boldsymbol u,k)$\;
\end{algorithm}

To state the output guarantee, let
$\proj:=n^{-1}\ones_n\ones_n^\top$,
$\bproj:=\proj\otimes I_d$, and
$\bprojperp:=(I_n-\proj)\otimes I_d$.
The matrices $\bproj$ and $\bprojperp$ are the orthogonal projections onto
the consensus and disagreement subspaces of $\R^{nd}$, respectively.
Every window preserves the consensus component, and the credit of every
window bounds the logarithm of its contraction on the disagreement
subspace. The analysis below shows that,
when the algorithm returns
$(\widehat{\boldsymbol u},k^+)=\WAVE(\boldsymbol u,k;\delta)$, it satisfies
\begin{equation}
\label{eq:wave-call-contract}
 \bproj\widehat{\boldsymbol u}=\bproj\boldsymbol u,
 \qquad
 \norm{\bprojperp\widehat{\boldsymbol u}}
 \leq\delta\norm{\bprojperp\boldsymbol u}.
\end{equation}
Here and below, once the scheduler is fixed, we suppress $\mathcal M$ in
the notation for a \WAVE{} call.
\section{Main results}
\label{subsec:main-consensus-results}
The first theorem covers the full range of metric drift with one
window-length rule and requires no commutativity between successive
operators.

\begin{theorem}[Consensus under metric drift]
\label{thm:main-drift}
Let Assumptions~\ref{ass:uniform-spectrum} and
\ref{ass:metric-drift} hold with $\chi\geq4$. For
$\eps\in(0,1)$, run Algorithm~\ref{alg:wave} on
$(\boldsymbol u_0,0;\eps)$ with $\mathcal M=\mathrm{drift}$ and
$m_\beta$ from \eqref{eq:scale-of-beta}. Then its output
$(\widehat{\boldsymbol u},K)$ is an $\eps$-consensus solution and
\begin{equation}
\label{eq:main-drift-bound}
 K=
 O\left(\left(\sqrt\chi+\min\{\beta\chi^2,\chi\}\right)
 \ln\frac{e}{\eps}\right).
\end{equation}
 Moreover, every completed window reduces
the disagreement norm by a factor at most
$\exp(-m_\beta^2/(5\chi))$.
\end{theorem}

For small drift, long windows retain Chebyshev acceleration. As $\beta$
increases, the window shortens and the additional term grows as
$\beta\chi^2$ until it reaches the robust $O(\chi)$ scale, where
$m_\beta=1$ and the recurrence becomes its degree-one endpoint.

\begin{corollary}[Fixed-network rate under small drift]
\label{cor:small-drift}
Under the assumptions of Theorem~\ref{thm:main-drift}, if
$\beta\leq(3\chi^{3/2})^{-1}$, then
$m_\beta=\lfloor\sqrt\chi\rfloor$ and
$ K=O\!\left(\sqrt\chi\ln\frac{e}{\eps}\right).$
This includes the fixed-network case $\beta=0$.
\end{corollary}

The corollary still allows the operator to change at every round; only
the norm of each change is restricted. The constant $3$ comes from the
explicit stability estimate of Lemma~\ref{lem:moving-window-stability},
and the proof gives $K\le10\sqrt\chi\ln(1/\eps)+\sqrt\chi$ in this regime.
The second theorem replaces control of the jump size by control of the
time between jumps. The detector restarts the recurrence at a change, and
the doubling rule increases the attempted window during each unchanged
interval without using $\tau$.

\begin{theorem}[Consensus under piecewise-constant changes]
\label{thm:dwell-det}
Let Assumptions~\ref{ass:uniform-spectrum} and
\ref{ass:dwell-time} hold with $\chi\geq4$, and assume access to the
ideal change detector. For $\eps\in(0,1)$, run
Algorithm~\ref{alg:wave} on $(\boldsymbol u_0,0;\eps)$ with
$\mathcal M=\mathrm{piecewise}$. Then its output
$(\widehat{\boldsymbol u},K)$ is an $\eps$-consensus solution and
\begin{equation}
\label{eq:dwell-det-bound}
 K=O\left(\left(\sqrt\chi+\frac{\chi}{\tau}\right)
 \ln\frac{e}{\eps}\right).
\end{equation}
\end{theorem}

The term $\chi/\tau$ is the amortized cost of restarts: for
$\tau\geq\sqrt\chi$ the bound is $O(\sqrt\chi\ln(e/\eps))$, and $\tau=1$
gives the robust $O(\chi\ln(e/\eps))$ dependence. The same bound holds in
expectation when the change times are memoryless with mean spacing
$\bar\tau$ (Theorem~\ref{thm:dwell-rand}). Together, the results give two
routes to accelerated consensus on changing networks: changes may be
frequent but small, or abrupt but separated. Neither requires a connected
subgraph that persists throughout the run.
The bounds also quantify partial acceleration between the fully accelerated
and worst-case endpoints. For example, $\beta=\Theta(\chi^{-5/4})$
in the metric-drift model or $\tau=\Theta(\chi^{1/4})$ in the
piecewise-constant model gives $O(\chi^{3/4}\ln(e/\eps))$ rounds.

\section{Metric-drift network}
\label{sec:metric-analysis}

We compare each moving \WAVE{} window with a fixed reference operator.
For a window beginning at round $k$, denote the sum drift of Laplacian matrix as
$\boldsymbol{E}_t:=\bLap_{k+t}-\bLap_{k}$, so that $\norm{\boldsymbol{E}_t}\le\beta t$ under metric drift (Assumption \ref{ass:metric-drift}). The recurrence uses the actual
operators:
\begin{equation}\label{eq:local-window-recurrence}
    \boldsymbol e_1=\frac{z_\chi(\bLap_{k})}{z_0} \boldsymbol e_0,
    \qquad
    \boldsymbol e_{t+1}=a_tz_\chi(\bLap_{k+t})\boldsymbol e_t-c_t\boldsymbol e_{t-1},
    \quad t=1,\ldots,m-1.
\end{equation}
If every $\boldsymbol E_t$ vanishes, then $\boldsymbol e_m=P_m^\chi(\bLap_{k}) \boldsymbol e_0$ by
Proposition~\ref{prop:fixed-window-representation}. To quantify the effect of changing
operators, let $\ChebU_r$ denote the Chebyshev polynomial of the second
kind, defined by $\ChebU_0(x)=1$, $\ChebU_1(x)=2x$, and
$\ChebU_{r+1}(x)=2x\ChebU_r(x)-\ChebU_{r-1}(x)$. For $1\leq t\leq m$,
the exact variation-of-constants identity is
\[
 \boldsymbol e_t=P_t^\chi(\bLap_{k})\boldsymbol e_0-
 \frac{4}{1-\chi^{-1}}\sum_{j=2}^t
 \frac{\ChebT_{j-1}(z_0)}{\ChebT_t(z_0)}
 \ChebU_{t-j}\!\left(z_\chi(\bLap_{k})\right)\boldsymbol E_{j-1}\boldsymbol e_{j-1}.
\]
Each perturbation acts on the actual iterate $e_{j-1}$ and is then
propagated by a polynomial of the reference operator. The matrix order
is preserved, so no commutativity is needed. On $(\ones_n\otimes \boldsymbol I_d)^\perp$,
$\norm{\ChebU_{t-j}(z_\chi(\bLap_{k}))}\leq t-j+1$ and
$\norm{\boldsymbol E_{j-1}}\leq\beta(j-1)$. Summing these responses produces the
cubic short-window term below. The normalization ratios give the second,
geometrically damped bound. Appendix~\ref{app:metric-proofs} proves the
identity and both estimates.
\label{subsec:moving-recurrence}\label{subsec:variation-of-constants}

\begin{lemma}[Stability of a moving full-spectrum window]
\label{lem:moving-window-stability}
Let Assumptions~\ref{ass:uniform-spectrum} and
\ref{ass:metric-drift} hold. Let a window satisfy
\eqref{eq:local-window-recurrence} and set
$    M_t:=\max_{0\le s\le t}\norm{\boldsymbol e_s},
    D_t:=\min\left\{\frac23\beta(t^3-t),2\beta t\chi\right\}.$
Then $\norm{\boldsymbol e_t-P_t^\chi(\bLap_{k})\boldsymbol e_0}\le D_tM_t$
for every $0\le t\le m$.
Consequently, if $D_m<1$, then
\begin{equation}\label{eq:scalar-window-contraction}
 \norm{\boldsymbol e_m}
 \le
 \left(\rho_m^\chi+\frac{D_m}{1-D_m}\right)\norm{\boldsymbol e_0}.
\end{equation}
\end{lemma}

The two terms in $D_t$ are the short-window and geometrically damped
response bounds, and $D_1=0$ because a single step uses a single operator.
For $m\leq\sqrt\chi$, the fixed-reference progress is of order $m^2/\chi$,
while metric drift contributes $O(\beta m^3)$. Requiring the perturbation
to remain below the reference progress gives $\beta m\chi=O(1)$ and
explains the choice $m_\beta\asymp\min\{\sqrt\chi,(\beta\chi)^{-1}\}$ in
\eqref{eq:scale-of-beta}.

\begin{theorem}[Uniform certificate for a full-spectrum window]
\label{thm:metric-window}
Let Assumptions~\ref{ass:uniform-spectrum} and
\ref{ass:metric-drift} hold with $\chi\ge4$. If $m=1$, no restriction
on $\beta$ is required.  If $m\ge2$, suppose
$m^2/\chi\le1$ and $\beta m\chi\le1/3$.
Then every complete Chebyshev window satisfies
\begin{equation}\label{eq:window-guarantee}
    \norm{\boldsymbol e_{m}}
    \le
    \exp\left(-\frac{m^2}{5\chi}\right)
    \norm{\boldsymbol e_{0}}.
\end{equation}
Hence, for every $\eps\in(0,1)$, repeated complete windows reach
$\eps$-consensus within
$K(\eps)\le(5\chi/m)\ln(1/\eps)+m$ communication rounds.
\end{theorem}
\label{subsec:short-window-analysis}

The main idea behind the proof of Theorem~\ref{thm:metric-window} is to compare the contraction of an unchanged Chebyshev
window with the perturbation accumulated while the operator moves. For
$m\leq\sqrt\chi$, these quantities scale respectively as $m^2/\chi$ and
$\beta m^3$, so the condition $\beta m\chi\leq1/3$ keeps the perturbation
below the useful progress and yields \eqref{eq:window-guarantee}. When
$m=1$, the window uses only one operator and therefore needs no drift
condition. The complete argument, including the constants, is given in
Appendix~\ref{app:metric-proofs}. Any shorter integer window, for example
one obtained by replacing the constant $3$ in \eqref{eq:scale-of-beta} by
a larger one, satisfies the same hypotheses.

\paragraph{Choosing the window length.}
\label{subsec:choose-window}
If $m_\beta\ge2$, then $m_\beta\le\sqrt\chi$ and $\beta m_\beta\chi\le1/3$,
so Theorem~\ref{thm:metric-window} applies; if $m_\beta=1$, its
degree-one part applies without a drift condition. When
$\widetilde m_\beta\ge2$, the rounding inequality
$m_\beta\ge\widetilde m_\beta/2$ gives
$\chi/m_\beta\le2\max\{\sqrt\chi,3\beta\chi^2\}$. When
$\widetilde m_\beta<2$, the degree-one cost is $O(\chi\ln(e/\eps))$ and
$\beta\chi^2>\chi/6$. This proves Theorem~\ref{thm:main-drift} with one
recurrence and one length formula; the transition from acceleration to
the robust endpoint is caused only by the integer value of $m$.

\section{Piecewise-constant networks}
\label{sec:dwell}

At a detected change, Algorithm~\ref{alg:wave} keeps the state reached
by the current window and starts a new length-one attempt before using the
new operator. The jump magnitude therefore does not enter the contraction
estimate: successive graphs may have entirely different edge sets, provided
each operator satisfies the common spectral bound. Every window runs on
a single operator and, by
Propositions~\ref{prop:fixed-window-representation} and~\ref{prop:cheb-minimax}, a window
of $h\le m_{\max}$ rounds, complete or interrupted, earns
$\ln\ChebT_h(z_0)\ge4h^2/(5\chi)$ by Lemma~\ref{lem:window}.
A change restarts the build-up of acceleration, but preserves the progress
already made. The term $\chi/\tau$ accounts for this restart cost over
unchanged intervals. In particular, attaining the fixed-network rate only
requires each interval to support a constant-factor reduction, rather than
an entire run to the final accuracy $\eps$.

\subsection{Minimum spacing between changes}
\label{subsec:deterministic-dwell}

Consider $r$ consecutive rounds with a fixed operator following a
length-one start. If the largest completed window has length $m<m_{\max}$,
then $r<4m$: completed windows use fewer than $2m$ rounds, and the
next attempt has length at most $2m$. If $r\ge4m_{\max}$, the shorter
windows use fewer than $2m_{\max}$ rounds, leaving time for at least
$r/(5m_{\max})$ completed windows of length $m_{\max}$. These two
cases give at least $r\min\{r,m_{\max}\}/(20\chi)$ of certified
logarithmic reduction in the disagreement norm.
On an interval between changes, $r\ge\tau$, so the certified reduction is
at least $1/(40(\sqrt\chi+\chi/\tau))$ per round; calls that begin or end
inside an interval are handled in Appendix~\ref{app:dwell-proofs}.

\subsection{Detecting changes inside the network}
\label{subsec:flooding}

The ideal detector can be replaced by local observation and flooding.
For this extension, assume synchronous neighbor exchanges, that every
agent knows $n\ge2$ and $\chi$, and that it observes its current row of
$\Lap_k$ before round $k$. A message may carry a round index alongside
the usual vector. Set
$D_n:=\lfloor(\sqrt\chi/2)\ln(2n)\rfloor+1$.
Agents flood each observed change index and retain $D_n+1$ local
checkpoints. Exactly $D_n$ rounds after a change, they synchronously
restore its checkpoint and restart the recurrence. A prospective output
is returned only after $D_n$ further rounds have ruled out an unprocessed
earlier change; numerical computation and flooding continue during this
validation. Appendix~\ref{app:flooding} specifies the protocol.

\begin{samepage}
\begin{proposition}[Flooding detector]
\label{prop:flooding}
Under Assumptions~\ref{ass:uniform-spectrum} and~\ref{ass:dwell-time},
suppose $\tau\ge2D_n$ and the preceding local-observation and message
assumptions hold. The flooding-and-rollback version of
Algorithm~\ref{alg:wave} reaches $\eps$-consensus, for $\eps\in(0,1)$, in
\[
 K=O\!\left(\left(\sqrt\chi+\frac\chi\tau\right)\ln\frac e\eps
       +\sqrt\chi\ln n\right)
\]
communication rounds, without knowing $\tau$. It uses $O(D_n d)$ local
vector storage and one change index per neighbor message.
\end{proposition}
\end{samepage}

The nonzero-support graph of each operator has diameter at most $D_n$.
Indeed, the spectral bound gives
$\norm{P_{D_n}^\chi(\Lap)-\proj}
\leq2e^{-2D_n/\sqrt\chi}<1/n$, so every entry of
$P_{D_n}^\chi(\Lap)$ is positive. If two agents were more than $D_n$
hops apart, the corresponding entry of any polynomial of degree at most $D_n$
in $\Lap$ would be zero, a contradiction. Thus the same spectral
information used for acceleration also bounds the reporting delay.
Consequently, every rollback occurs before the next change; after deleting
the discarded rounds, the retained computation is an ideal-detector run
with segment lengths at least $\tau/2$. The condition
$\tau\ge\sqrt\chi\ln(2n)+2$ is a sufficient, slightly stronger spacing
condition. 

\section{Application to decentralized optimization}
\label{sec:optimization}

We consider
\[
    \min_{x\in\R^d}
    f(x):=\frac1n\sum_{i=1}^n f_i(x),
\]
where each $f_i$ is convex and $\alpha$-smooth, while $f$ is
$\mu$-strongly convex. Let $x^\star:=\argmin_x f(x)$ and
$\kappa:=\alpha/\mu$. \WAVESTM{} combines restarted Similar Triangles
iterations \citep{Nesterov2004,GorbunovDanilova2020} with \WAVE{}
communication. The agents start from a common point $x_0$ and are given
$G_0>0$ such that
$n^{-1}\sum_i\norm{\nabla f_i(x_0)}^2\leq G_0^2$.

\begin{theorem}[\WAVESTM{}]
\label{thm:opt-main}
Let Assumption~\ref{ass:uniform-spectrum} and either
Assumption~\ref{ass:metric-drift} or~\ref{ass:dwell-time} hold with
$\chi\geq4$. For every $\eps>0$, \WAVESTM{} returns an $\eps$-solution
at every agent:
\begin{equation}\label{eq:decent_opt}
        \max_{1\leq i\leq n}
    \bigl\{f(y_{\rm out}^{(i)})-f(x^\star)\bigr\}
    \leq\eps.
\end{equation}
Each agent evaluates $\widetilde O(\sqrt\kappa)$ local gradients, and
the number of communication rounds satisfies
\begin{equation}
\label{eq:opt-main}
K=
\begin{cases}
\widetilde O\!\left(
\sqrt\kappa\,
[\sqrt\chi+\min\{\beta\chi^2,\chi\}]
\right),
& \text{metric drift},\\[1mm]
\widetilde O\!\left(
\sqrt\kappa\,
[\sqrt\chi+\chi/\tau]
\right),
& \text{piecewise-constant changes}.
\end{cases}
\end{equation}
Here $\widetilde O$ hides logarithmic factors in the target accuracy,
initialization, $\kappa$, and $n$. The piecewise-constant guarantee uses
the ideal external change detector.
\end{theorem}

Consequently, when $\beta=O(\chi^{-3/2})$ or
$\tau=\Omega(\sqrt\chi)$, the leading communication dependence is
$\widetilde O(\sqrt{\kappa\chi})$, recovering the fixed-network scaling
up to logarithmic factors.

\begin{remark}
Guarantee~\eqref{eq:decent_opt} is stronger than the usual objective bound
for the network average
$\bar y_{\rm out}:=n^{-1}\sum_{i=1}^n y_{\rm out}^{(i)}$.
Indeed, convexity of $f$ gives
$f(\bar y_{\rm out})-f(x^\star)
\leq
\frac1n\sum_{i=1}^n
\bigl(f(y_{\rm out}^{(i)})-f(x^\star)\bigr)
\leq\eps.$
Moreover, \eqref{eq:decent_opt} guarantees an $\eps$-solution directly
at every agent, whereas the averaged iterate need not be locally available.
\end{remark}

\section{Numerical experiments}
\label{sec:experiments}

We compare \WAVE{} with gossip, minimax Richardson iteration, and the
time-varying method of \citet{MontijanoMontijanoSagues2013}. We also run
the classical Chebyshev semi-iteration \citep{GolubVarga1961} continuously,
without restarts, to test whether restarts are needed when the network
changes. Every method processes the same operators and spectral bounds. We
report the exact worst-case gap
$r_k:=\max_{\boldsymbol e_0\ne0}\norm{\boldsymbol e_k}/
\norm{\boldsymbol e_0}$; the level $10^{-6}$ is used only to summarize
first-passage times.

\paragraph{Metric drift.}
Figure~\ref{fig:fair-consensus-main}\textup{(a)} uses the controlled family
of Appendix~\ref{app:high-drift-family}, with $\chi=1.6\cdot10^5$ and
$\beta\chi^{3/2}\approx30$. This drift is $90$ times the sufficient
threshold for the fixed-network rate, and the \WAVE{} rule selects $m=4$.
\WAVE{} reaches $r_k\leq10^{-6}$ in about $2.8\cdot10^5$ rounds, versus
$1.1\cdot10^6$ for Montijano and Richardson; gossip remains above the
target after $1.25\cdot10^6$ rounds. Chebyshev without restarts is unstable
and leaves the plotted range within $5\cdot10^4$ rounds.

\paragraph{Piecewise-constant changes.}
Figure~\ref{fig:fair-consensus-main}\textup{(b)} uses $32$ independent
pairs of connected sparse random geometric graphs with $n=100$ and
$\chi\in[339,2873]$ (median $905$). The operator changes every
$\operatorname{round}(\sqrt\chi)$ rounds, with each change reported to
\WAVE{}. Median first-passage counts are $180$ rounds for \WAVE{}, $344$
for Montijano and Richardson, and $422$ for gossip. \WAVE{} is first on
all $32$ pairs. Chebyshev reaches the target on only $12$ pairs within the
$1000$-round budget.

\begin{figure}[H]
\centering
\vspace{-1mm}
\includegraphics[width=0.98\textwidth]
{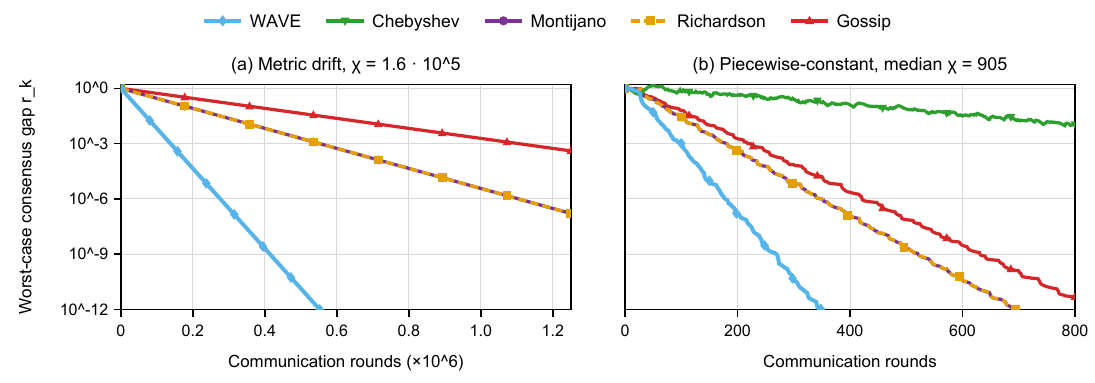}
\vspace{-1mm}
\caption{Worst-case consensus gap against communication rounds.
Classical Chebyshev is run without restarts. \textup{(a)} Controlled
metric drift with $\chi=1.6\cdot10^5$ and $m=4$. \textup{(b)} Pointwise
median over $32$ sparse random-geometric graph pairs, with changes every
$\operatorname{round}(\sqrt\chi)$ rounds.}
\label{fig:fair-consensus-main}
\vspace{-2mm}
\end{figure}

\section{Conclusion}
\label{sec:conclusion}


\WAVE{} shows that Chebyshev acceleration of consensus survives network variation if the recurrence is restarted before perturbations from changing, noncommuting operators outweigh its progress. To our knowledge, its bounds are the first consensus guarantees that interpolate between the fixed-network $\sqrt\chi$ and the arbitrary-variation $\chi$ rates as explicit functions of the one-step drift $\beta$, $O((\sqrt\chi+\min\{\beta\chi^2,\chi\})\ln(e/\eps))$, and of the spacing $\tau$ between changes, $O((\sqrt\chi+\chi/\tau)\ln(e/\eps))$; the fixed-network rate is retained whenever $\beta\le(3\chi^{3/2})^{-1}$ or $\tau\ge\sqrt\chi$. These guarantees require no persistent connected subgraph and no knowledge of $\tau$, hold in expectation for memoryless change times, and remain valid, up to an additive $O(\sqrt\chi\ln n)$ term, when sufficiently separated changes are detected by flooding inside the network. For convex local objectives with a strongly convex average, \WAVESTM{} multiplies these network factors by $\widetilde O(\sqrt\kappa)$ and thus retains the fixed-network $\widetilde O(\sqrt{\kappa\chi})$ dependence in the same two regimes. In the experiments, \WAVE{} outperforms all compared methods in both regimes.

\subsubsection*{Acknowledgments}
The research was supported by Russian Science Foundation (project No. 23-11-00229), \url{https://rscf.ru/en/project/23-11-00229/}.

\subsection*{AI use statement}
Generative AI tools were used for language editing and to improve the clarity of the manuscript. We reviewed all resulting changes and take full responsibility for the final content.

\bibliography{references}
\bibliographystyle{plainnat}

\newpage
\appendix

\section*{Supplementary material: contents}
\phantomsection
\addcontentsline{toc}{section}{Supplementary material: contents}
\begingroup
\small
\noindent
\hyperref[app:notation]{A. Notation and asymptotic conventions}
\dotfill\pageref*{app:notation}\\
\hyperref[app:model-details]{B. Model and implementation details}
\dotfill\pageref*{app:model-details}\\
\hyperref[app:chebyshev]{C. Classical Chebyshev facts}
\dotfill\pageref*{app:chebyshev}\\
\hyperref[app:metric-proofs]{D. Proofs for metric drift}
\dotfill\pageref*{app:metric-proofs}\\
\hyperref[app:dwell-proofs]{E. Piecewise-constant networks}
\dotfill\pageref*{app:dwell-proofs}\\
\hspace*{1em}\hyperref[app:flooding]{In-network change detection and rollback}
\dotfill\pageref*{app:flooding}\\
\hyperref[app:optimization]{F. Proofs for the optimization results}
\dotfill\pageref*{app:optimization}\\
\hspace*{1em}\hyperref[app:accgt]{Accelerated gradient tracking with \WAVE{}}
\dotfill\pageref*{app:accgt}\\
\hyperref[app:supplementary]{G. Examples and interpretation}
\dotfill\pageref*{app:supplementary}\\
\hyperref[app:supplementary-experiments]{H. Supplementary numerical experiments}
\dotfill\pageref*{app:supplementary-experiments}
\endgroup

\section{Notation and asymptotic conventions}
\label{app:notation}

Table~\ref{tab:notation} summarizes the notation and defines the
comparison symbol $\asymp$ used in the main scaling statements. Symbols
used only within individual proofs are defined locally.

\begingroup
\scriptsize
\setlength{\tabcolsep}{4pt}
\renewcommand{\arraystretch}{1.10}
\begin{longtable}{@{}
>{\raggedright\arraybackslash}p{0.16\textwidth}
>{\raggedright\arraybackslash}p{0.39\textwidth}
>{\raggedright\arraybackslash}p{0.37\textwidth}
@{}}
\caption{Consolidated notation and asymptotic conventions used throughout the paper.}
\label{tab:notation}\\
\toprule
\textbf{Symbol} & \textbf{Definition / convention} & \textbf{Meaning} \\
\midrule
\endfirsthead

\multicolumn{3}{@{}l}{\textit{Table~\ref{tab:notation} continued.}}\\[1mm]
\toprule
\textbf{Symbol} & \textbf{Definition / convention} & \textbf{Meaning} \\
\midrule
\endhead

\midrule
\multicolumn{3}{r@{}}{\textit{Continued on the next page.}}\\
\endfoot

\bottomrule
\endlastfoot

\multicolumn{3}{@{}l}{\textbf{Linear-algebra and consensus conventions}}\\

$n,d,i$ &
$n$ agents, local dimension $d$, agent index $i\in\{1,\ldots,n\}$ &
Network size, local vector dimension, and agent index. A parenthesized
superscript, as in $u_k^{(i)}$, is reserved for the agent index.\\

$\ones_n$, $\ones_n^\perp$ &
$\ones_n=(1,\ldots,1)^\top$,
$\ones_n^\perp=\{v\in\R^n:\ones_n^\top v=0\}$ &
All-ones vector and disagreement subspace.\\

$\proj$, $\proj_\perp$ &
$\proj=\ones_n\ones_n^\top/n$, $\proj_\perp=I_n-\proj$ &
Orthogonal projectors onto the consensus and disagreement subspaces.\\

$\bproj$, $\bprojperp$ &
$\bproj=\proj\otimes I_d$, $\bprojperp=\proj_\perp\otimes I_d$ &
Block projectors on stacked vectors in $\R^{nd}$.\\

$u_k^{(i)}$, $\boldsymbol u_k$ &
$\boldsymbol u_k=((u_k^{(1)})^\top,\ldots,(u_k^{(n)})^\top)^\top$ &
Local and stacked consensus states at communication round $k$.\\

$\bar u_0$ &
$\bar u_0=n^{-1}\sum_i u_0^{(i)}$ &
Initial average.\\

$\boldsymbol e_k$, $\boldsymbol e_{(w)}$ &
$\boldsymbol e_k=\boldsymbol u_k-\ones_n\otimes\bar u_0$,
$\boldsymbol e_{(w)}:=\boldsymbol e_{k_0+wm}$ for equal-length windows starting at $k_0$ &
Global consensus error and error at a window boundary.\\

$\lambda$, $\lambda_{\max}(A)$,
$\lambda_{\min}^{+}(A)$ &
$\lambda$ is reserved for eigenvalues or scalar spectral arguments;
$\lambda_{\max}(A)$ and $\lambda_{\min}^{+}(A)$ are the largest and
smallest positive eigenvalues &
Spectral variables and extremal positive eigenvalues.\\

$\spec(A)$, $A|_S$ &
$\spec(A)$ is the set of eigenvalues of $A$; $A|_S$ is the restriction of
$A$ to an invariant subspace $S$ &
Spectrum and restricted operator.\\

$\chi(A)$ &
$\chi(A)=\lambda_{\max}(A)/\lambda_{\min}^{+}(A)$ and
$\chi(cA)=\chi(A)$ for $c>0$ &
Condition number on $(\ker A)^\perp$; invariant under positive scalar
normalization.\\

$\norm{\cdot}$, $\inner{\cdot}{\cdot}$ &
Euclidean norm for vectors, spectral norm for matrices, and the
Euclidean inner product &
Norm and inner-product conventions.\\

$\lesssim$, $\gtrsim$, $\asymp$ &
For nonnegative $a,b$, $a\lesssim b$ means $a\le Cb$ for an absolute
constant $C>0$ independent of the problem parameters;
$a\gtrsim b$ is the reverse inequality and $a\asymp b$ means both &
Comparison up to absolute multiplicative constants.  Parameter-dependent
constants are stated explicitly when used.\\

\midrule
\multicolumn{3}{@{}l}{\textbf{Network and variation quantities}}\\

$V$ &
$V=\{1,\ldots,n\}$ &
Fixed set of agent indices.\\

$\mathcal G_k=(V,\mathcal E_k)$ &
Connected undirected communication graph at round $k$ &
Time-varying communication network.\\

$\chi$ &
$\spec(\Lap_k|_{\ones_n^\perp})\subseteq[\chi^{-1},1]$ for every $k$ &
Known conservative spectral-envelope condition bound for the normalized
communication operators.\\

$\Lap_k$, $\bLap_k$ &
$\bLap_k=\Lap_k\otimes I_d$, and
$\spec(\Lap_k|_{\ones_n^\perp})\subseteq[\chi^{-1},1]$ &
Normalized agent-level gossip matrix and its lifted operator on stacked
states.\\

$\Lclass(\chi)$ &
$\{\Lap:\Lap\text{ is a normalized gossip matrix and }
\spec(\Lap|_{\ones_n^\perp})\subseteq[\chi^{-1},1]\}$ &
Communication-operator class used in the model.\\

$\beta$ &
$\norm{\Lap_{k+1}-\Lap_k}\le\beta$ &
One-step operator drift in the metric-drift model.\\

$t_i$, $\tau_i$, $\tau$ &
$\tau_i=t_{i+1}-t_i$ and $\tau_i\ge\tau$ under the minimum-spacing assumption &
Change times, constant-segment lengths, and their minimum spacing.\\

$\bar\tau$ &
$\Prob(B_k=1)=1/\bar\tau$; equivalently, $\E T_i=\bar\tau$ for the
geometric segment lengths $T_i$ &
Mean segment length in the memoryless change-time model.\\

\midrule
\multicolumn{3}{@{}l}{\textbf{Communication and WAVE indices}}\\

$k,K$ &
$k=0,\ldots,K-1$ indexes communication updates; states are indexed by
$k=0,\ldots,K$ &
Global communication index and total number of communication rounds.\\

$w,J$ &
$w=0,\ldots,J-1$ &
Chebyshev-window index and number of complete windows.\\

$t$ &
$t=0,\ldots,m$ inside a single Chebyshev window &
Local step index within a window.\\

$\widetilde m_\beta,m_\beta$ &
Defined by \eqref{eq:scale-of-beta} &
Continuous drift target and selected integer window length.\\

$m,m_{\max},h$ &
$m$ is the current attempted length,
$m_{\max}=\lfloor\sqrt\chi\rfloor$, and $h$ is the number of rounds
executed in an attempted window &
Attempted length, dyadic cap, and executed length.\\

$q$ &
Increment $m_\beta^2/(5\chi)$ per drift window and $\ln\ChebT_h(z_0)$ per
piecewise window of $h$ rounds &
Certified logarithmic contraction accumulated by
Algorithm~\ref{alg:wave}.\\

$D_n$ &
$D_n=\lfloor\frac{\sqrt\chi}{2}\ln(2n)\rfloor+1$ &
Delay of the flooding detector (Proposition~\ref{prop:flooding}).\\

\midrule
\multicolumn{3}{@{}l}{\textbf{Chebyshev quantities}}\\

$\ChebT_m$, $\ChebU_m$ &
Chebyshev polynomials of the first and second kind &
Classical polynomial basis used in the accelerated recurrence.\\

$z_\chi(\lambda)$, $z_0$ &
$z_\chi(\lambda)=(1+\chi^{-1}-2\lambda)/(1-\chi^{-1})$,
$z_0=z_\chi(0)$ &
Affine map sending the full interval $[\chi^{-1},1]$ to $[-1,1]$ and
the image of the origin.\\

$P_m^\chi(\lambda)$ &
$P_m^\chi(\lambda)=\ChebT_m(z_\chi(\lambda))/\ChebT_m(z_0)$ &
Full-spectrum shifted and normalized Chebyshev residual executed by
\WAVE{}.\\

$\rho_m^\chi$ &
$\rho_m^\chi=\max_{[\chi^{-1},1]}|P_m^\chi|$ &
Full-spectrum minimax residual value.\\

$a_t,c_t$ &
$a_t=2\ChebT_t(z_0)/\ChebT_{t+1}(z_0)$ for $t\geq0$ and
$c_t=\ChebT_{t-1}(z_0)/\ChebT_{t+1}(z_0)$ for $t\geq1$ &
Coefficients of the Chebyshev three-term recurrence.\\

\midrule
\multicolumn{3}{@{}l}{\textbf{Optimization quantities and indices}}\\

$f_i$, $f$ &
$f(x)=n^{-1}\sum_{i=1}^n f_i(x)$ &
Local objectives and global average objective.\\

$x^\star$ &
$x^\star=\argmin f(x)$ &
Unique global optimizer.\\

$\alpha,\mu,\kappa$ &
Each $f_i$ is convex and $\alpha$-smooth; the average $f$ is
$\mu$-strongly convex;
$\kappa=\alpha/\mu$ &
Local smoothness, strong convexity of the average, and their ratio.\\

$G_0,R,H$ &
$n^{-1}\sum_i\norm{\nabla f_i(x_0)}^2\leq G_0^2$;
$R=G_0/\mu$, $H=(1+\kappa)G_0$ &
Single supplied initialization bound and its derived radius and
heterogeneity bounds.\\

$\WAVE(\boldsymbol u,k;\delta)$ &
Defined in Algorithm~\ref{alg:wave}; returns an updated state and the next
unused global communication round &
Direct \WAVE{} call with requested relative disagreement factor $\delta$.\\

$\ln_+(t)$ &
$\ln_+(t):=\max\{0,\ln t\}$ &
Nonnegative part of the natural logarithm.\\

$\ell$ &
$\ell=0,\ldots,N-1$ within an STM stage &
Inner optimization-iteration index.\\

$A_\ell$, $\omega_{\ell+1}$ &
$A_{\ell+1}=A_\ell+\omega_{\ell+1}$,
$\omega_{\ell+1}=(\ell+2)/(2\alpha)$ &
Weights of the Similar Triangles recurrence.\\

$s,S$ &
$s=0,\ldots,S-1$ &
Optimization-stage index and number of restart stages.\\

$N_0$, $r_s$, $r_\eps$, $\delta_\star$ &
$N_0=\lceil8\sqrt\kappa\rceil$,
$r_s=2^{-s}R$; $r_\eps$ and $\delta_\star$ are defined in
\eqref{eq:optimization-derived-scales} and \eqref{eq:delta-star} &
Stage length, deterministic radius schedule, and the single inner-call
\WAVE{} tolerance.\\

\end{longtable}
\endgroup

\section{Additional model and implementation details}
\label{app:model-details}

\subsection{Normalization and the linear operator model}

For a symmetric positive-semidefinite matrix $A$ with a positive
eigenvalue, its condition number on $(\ker A)^\perp$ is
\[
 \chi(A)=\frac{\lambda_{\max}(A)}{\lambda_{\min}^{+}(A)}.
\]
It is invariant under positive scalar normalization.  If the system
supplies an unnormalized weighted Laplacian $\widetilde\Lap_k$, a fixed
certified scale
$\Lambda\ge\sup_k\lambda_{\max}(\widetilde\Lap_k)$ may therefore be
absorbed into the communication weights by setting
$\Lap_k=\widetilde\Lap_k/\Lambda$.  The same scale must be used at all
rounds, since round-dependent rescaling changes metric drift.

In the linear-operator model used here, coefficients may depend on
the round number, the stated a priori parameters, and past external
change notifications, but not on a matrix-specific eigendecomposition,
the exact spectrum, or a topology-specific finite-time aggregation
scheme.  Generated vectors may be stored without restriction.  The
algorithms use no internal randomization: in the memoryless change-time model,
randomness belongs only to the network change process.

\subsection{Finite-window implementation and progress accounting}
\label{app:wave-implementation}

Algorithm~\ref{alg:chebyshev-window} evaluates the coefficient ratios in
\eqref{eq:cheb-coeffs-at-glance} without forming the potentially large
values $\ChebT_t(z_0)$.  In exact arithmetic it is the same
same recurrence and reproduces every fixed-operator Chebyshev iterate.

\begin{algorithm}[ht]
\caption{\ChebyshevWindow$(\boldsymbol u;k,m)$: finite run of the
classical Chebyshev semi-iteration}
\label{alg:chebyshev-window}
\KwIn{current state $\boldsymbol u\in\R^{nd}$; starting communication
round $k$; common spectral-envelope parameter $\chi$; window length $m\ge1$}
$z_0\leftarrow(1+\chi^{-1})/(1-\chi^{-1})$; 

$a_0\leftarrow2/z_0$\;

$\boldsymbol v_{\rm prev}\leftarrow\boldsymbol u$\;
$\boldsymbol v_{\rm curr}\leftarrow
   (a_0/2)z_\chi(\bLap_k)\boldsymbol u$\;
\For{$t=1,\dots,m-1$}{
  $a_t\leftarrow4/(4z_0-a_{t-1})$
  \tcp*[r]{$=2\ChebT_t(z_0)/\ChebT_{t+1}(z_0)$}
  $c_t\leftarrow a_{t-1}a_t/4$
  \tcp*[r]{$=\ChebT_{t-1}(z_0)/\ChebT_{t+1}(z_0)$}
  $\boldsymbol v_{\rm new}\leftarrow
     a_tz_\chi(\bLap_{k+t})\boldsymbol v_{\rm curr}
     -c_t\boldsymbol v_{\rm prev}$\;
  $\boldsymbol v_{\rm prev}\leftarrow\boldsymbol v_{\rm curr}$;
  
  $\boldsymbol v_{\rm curr}\leftarrow\boldsymbol v_{\rm new}$\;
}
\Return $\boldsymbol v_{\rm curr}$\;
\end{algorithm}

\begin{lemma}[Mean preservation and accumulation of certified progress]
\label{lem:wave-progress-accounting}
Every call of Algorithm~\ref{alg:wave} preserves the block average.
Moreover, if $\boldsymbol u^{(r)}$ and $q_r$ denote, respectively, the
state and the value of the progress counter after the $r$th update of
the outer loop, then
\[
    \bproj\boldsymbol u^{(r)}=\bproj\boldsymbol u^{(0)},
    \qquad
    \norm{\bprojperp\boldsymbol u^{(r)}}
    \le e^{-q_r}\norm{\bprojperp\boldsymbol u^{(0)}}.
\]
Consequently, termination with $q_r\ge\ln(1/\delta)$ implies
\[
    \norm{\bprojperp\boldsymbol u^{(r)}}
    \le\delta\norm{\bprojperp\boldsymbol u^{(0)}}.
\]
\end{lemma}

\begin{proof}
Since $\bLap_k\bproj=\bproj\bLap_k=0$,
$\bproj z_\chi(\bLap_k)=z_0\bproj$.  The first iterate of a Chebyshev
window therefore preserves the mean because $a_0/2=1/z_0$.  For every
subsequent local step, the Chebyshev recurrence gives
\[
 a_tz_0-c_t
 =
 \frac{2z_0\ChebT_t(z_0)-\ChebT_{t-1}(z_0)}
      {\ChebT_{t+1}(z_0)}
 =1.
\]
Induction on the local step thus shows that every complete or
interrupted Chebyshev window preserves the input mean.

Under metric drift, every window has length $m_\beta$ and norm at most
$e^{-m_\beta^2/(5\chi)}$ on the disagreement subspace by
Theorem~\ref{thm:metric-window}, which is its credit. Under the
piecewise-constant scheduler, every window, complete or interrupted by a
report, runs on one operator; by Propositions~\ref{prop:fixed-window-representation}
and~\ref{prop:cheb-minimax} a window of $h$ rounds has norm at most
$1/\ChebT_h(z_0)=e^{-\ln\ChebT_h(z_0)}$ on the disagreement subspace,
which is again its credit. Multiplication of these successive bounds
proves
$\norm{\bprojperp\boldsymbol u^{(r)}}\le
e^{-q_r}\norm{\bprojperp\boldsymbol u^{(0)}}$.  No commutativity
between different windows or operators is used.
\end{proof}

\section{Classical Chebyshev facts}
\label{app:chebyshev}

Here we collect the classical polynomial background, an illustrative
residual plot, and the proofs used in Sections
\ref{sec:main-results}, \ref{sec:metric-analysis},
and~\ref{sec:dwell}.  These details are
separated from the main argument because the polynomial construction is
standard; the moving-operator stability analysis is not.

\subsection{Residual-polynomial viewpoint and illustration}
\label{subsec:residual-minimax}

For a fixed operator, any mean-preserving method in this linear-operator
model gives
\[
 \boldsymbol e_m=p_m(\bLap)\boldsymbol e_0,\qquad
 \deg p_m\le m,\qquad p_m(0)=1 .
\]
The spectral theorem therefore yields
\[
 \norm{\boldsymbol e_m}
 \le\max_{\lambda\in[\chi^{-1},1]}|p_m(\lambda)|
      \norm{\boldsymbol e_0}.
\]
The affine map $z_\chi$ sends the admissible interval to $[-1,1]$, so
the classical alternation theorem gives the shifted, normalized
Chebyshev residual in Proposition~\ref{prop:cheb-minimax}.  In
particular,
\[
 \max_{\lambda\in[\chi^{-1},1]}|P_m^\chi(\lambda)|
 \le2e^{-2m/\sqrt\chi}.
\]

\begin{figure}[t]
 \centering
 \includegraphics[width=0.6\textwidth]{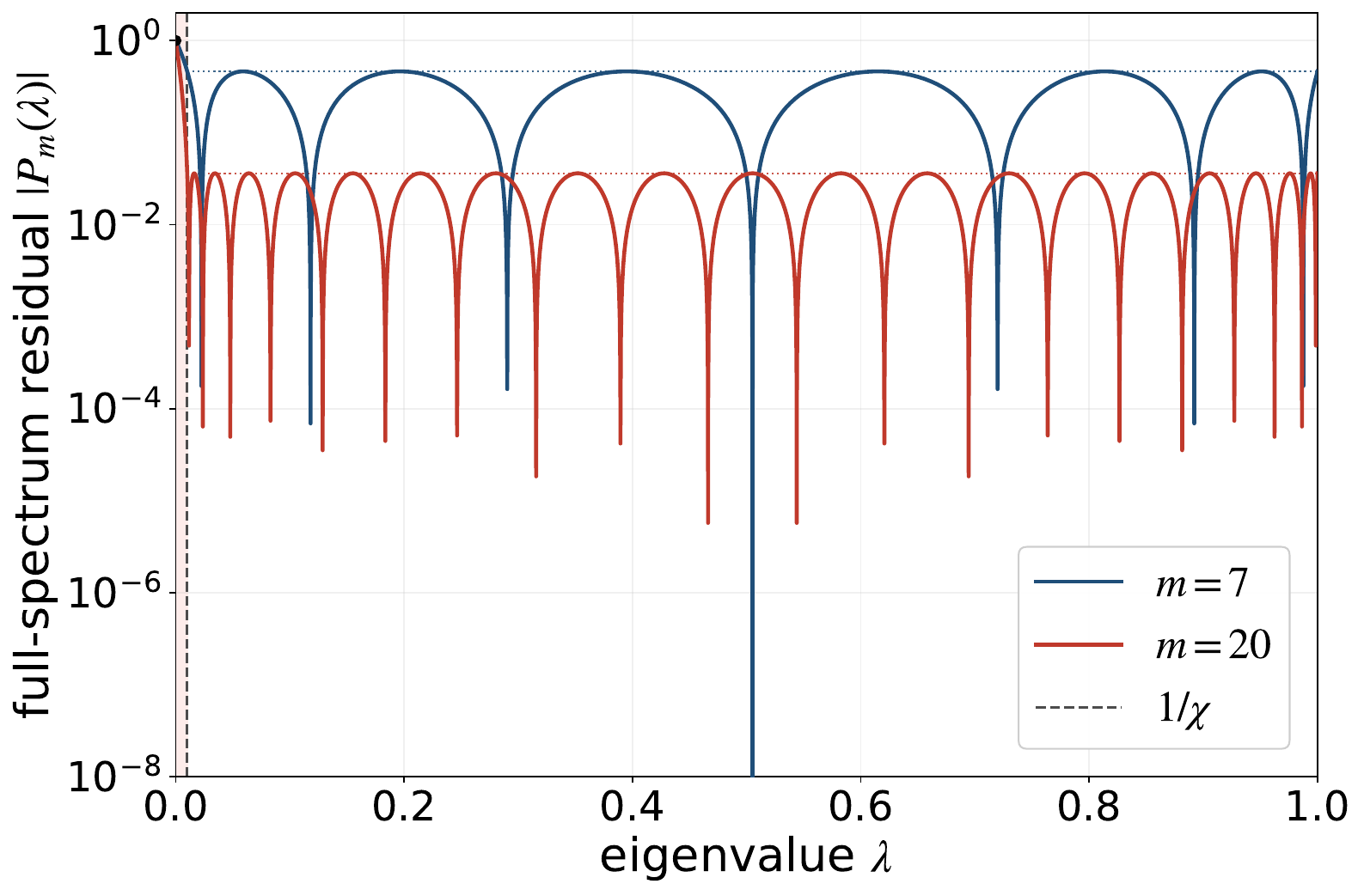}
 \caption{Full-spectrum residuals $|P_m^\chi(\lambda)|$ for $m=7$ and
 $m=20$ with $\chi=100$.  Both polynomials use the same interval
 $[\chi^{-1},1]$; only the finite window length changes.}
 \label{fig:cheb_residual}
\end{figure}

\subsection{Proof of Proposition~\ref{prop:cheb-minimax}}

The affine map $x=z_\chi(\lambda)$ sends $[\chi^{-1},1]$ onto $[-1,1]$ and
sends the normalization point $\lambda=0$ to $z_0>1$.  Hence
\eqref{eq:minimax-value} is equivalent to
\[
\min_{\substack{\deg q\le m\\q(z_0)=1}}
\max_{x\in[-1,1]}|q(x)|.
\]
The polynomial $q_\star(x)=\ChebT_m(x)/\ChebT_m(z_0)$ is feasible and
has uniform norm $1/\ChebT_m(z_0)$ on $[-1,1]$.

Let
\[
    x_j:=\cos(j\pi/m),\qquad j=0,\ldots,m,
\]
so that $x_0>x_1>\cdots>x_m$ and
$\ChebT_m(x_j)=(-1)^j$.  Let $\ell_j$ be the Lagrange cardinal
polynomial associated with these nodes.  Since $z_0>x_0$,
$\operatorname{sign}\ell_j(z_0)=(-1)^j$.  Interpolation of
$\ChebT_m$ at the $m+1$ nodes therefore gives
\[
 \ChebT_m(z_0)
 =\sum_{j=0}^m\ell_j(z_0)(-1)^j
 =\sum_{j=0}^m|\ell_j(z_0)|.
\]
For every polynomial $q$ of degree at most $m$ with $q(z_0)=1$,
Lagrange interpolation yields
\[
  1
  =\left|\sum_{j=0}^m\ell_j(z_0)q(x_j)\right|
  \le
  \norm{q}_{L^\infty[-1,1]}
  \sum_{j=0}^m|\ell_j(z_0)|
  =
  \ChebT_m(z_0)\norm{q}_{L^\infty[-1,1]}.
\]
Hence $\norm{q}_{L^\infty[-1,1]}\ge1/\ChebT_m(z_0)$, and
$q_\star=\ChebT_m/\ChebT_m(z_0)$ attains equality.

If equality holds, equality must hold both in the triangle inequality
and in all $m+1$ pointwise bounds above.  Consequently,
\[
    q(x_j)=\frac{(-1)^j}{\ChebT_m(z_0)}
    =q_\star(x_j),\qquad j=0,\ldots,m.
\]
Two polynomials of degree at most $m$ agreeing at $m+1$ distinct points
are identical.  Thus $q=q_\star$, proving uniqueness.

Finally,
\[
    \theta_\chi:=\operatorname{arccosh}z_0
    =2\operatorname{artanh}(\chi^{-1/2})
    \ge2\chi^{-1/2},
\]
so
\[
    \ChebT_m(z_0)=\cosh(m\theta_\chi)
    \ge\tfrac12e^{m\theta_\chi}
    \ge\tfrac12e^{2m/\sqrt\chi},
\]
which proves the final inequality in~\eqref{eq:minimax-value}.
\hfill$\square$

\subsection{Proof of Proposition~\ref{prop:fixed-window-representation}}

The identity holds for $t=0$.  Since
$P_1^\chi(\lambda)=z_\chi(\lambda)/z_0$, it also holds for $t=1$.
Assume it holds for $t$ and $t-1$.  The Chebyshev recurrence and the
closed forms of $a_t,c_t$ in Table~\ref{tab:notation} give
\begin{align*}
P_{t+1}^\chi(\lambda)
&=
\frac{2\ChebT_t(z_0)}{\ChebT_{t+1}(z_0)}
 z_\chi(\lambda)P_t^\chi(\lambda)
-
\frac{\ChebT_{t-1}(z_0)}{\ChebT_{t+1}(z_0)}
 P_{t-1}^\chi(\lambda)\\
&=a_tz_\chi(\lambda)P_t^\chi(\lambda)-c_tP_{t-1}^\chi(\lambda).
\end{align*}
Substitution of the induction hypothesis into
\eqref{eq:moving-window-recurrence} with $\bLap_{k+t}\equiv\bLap$ yields
$\boldsymbol v_{t+1}=P_{t+1}^\chi(\bLap)\boldsymbol u$.

\hfill$\square$

\subsection{Proof of Lemma~\ref{lem:window}}

Set
\[
    s:=\chi^{-1/2},
    \qquad
    x:=m\theta_\chi
      =2m\operatorname{artanh}s,
    \qquad
    h:=\frac{m^2}{\chi}.
\]
Proposition~\ref{prop:cheb-minimax} gives
\[
    \max_{\lambda\in[\chi^{-1},1]}|P_m^\chi(\lambda)|
    =\operatorname{sech}x.
\]
Since \(\operatorname{artanh}s\ge s\), one has \(x\ge2\sqrt h\).
For the upper range, \(s\le1/2\), \(ms\le\sqrt2\), and the function
\(\operatorname{artanh}(s)/s\) is increasing on \((0,1)\).  Therefore
\[
    x
    =2(ms)\frac{\operatorname{artanh}s}{s}
    \le2\sqrt2\ln3
    <\frac{25}{8}.
\]
The last strict inequality follows, for example, from
\(\sqrt2<99/70\) and \(\ln3<11/10\).

We next record a scalar estimate valid for
\(0\le x\le25/8\):
\begin{equation}\label{eq:logcosh-finite-range}
    \ln\cosh x\ge\frac{x^2}{5}.
\end{equation}
Indeed, let \(F(x)=\ln\cosh x-x^2/5\) and
\(G(x)=F'(x)=\tanh x-2x/5\).  Since
\(G'(x)=\operatorname{sech}^2x-2/5\) is strictly decreasing on
\((0,\infty)\), \(G\) can change sign from positive to negative at most
once.  Thus the minimum of \(F\) on \([0,25/8]\) occurs at an endpoint.
Now \(F(0)=0\), while
\[
 F(25/8)
 \ge\frac{25}{8}-\ln2-\frac{125}{64}
 >\frac{25}{8}-\frac7{10}-\frac{125}{64}>0.
\]
This proves \eqref{eq:logcosh-finite-range}.  The rational logarithm
bounds used above follow directly from the positive exponential series.

Combining \(x^2\ge4h\) with
\eqref{eq:logcosh-finite-range} yields
\[
    \operatorname{sech}x
    =\exp(-\ln\cosh x)
    \le\exp\left(-\frac{x^2}{5}\right)
    \le\exp\left(-\frac{4m^2}{5\chi}\right)
    \le\exp\left(-\frac{m^2}{2\chi}\right).
\]
This is \eqref{eq:window-contr}.
\hfill$\square$

\section{Proofs for the metric-drift results}
\label{app:metric-proofs}

We first establish the exact variation-of-constants decomposition and
the cumulative response bounds, then prove the direct metric-drift window
estimate. All operators below act on the disagreement subspace of
$\R^{nd}$.

Let $\ChebU_r$ denote the Chebyshev polynomial of the second kind.  For a
horizon $m$, define
\[
 \boldsymbol\xi_0:=-\frac{2}{(1-\chi^{-1})z_0}\boldsymbol E_0\boldsymbol e_0,
 \qquad
 \boldsymbol\xi_t:=-\frac{2a_t}{1-\chi^{-1}}\boldsymbol E_t\boldsymbol e_t
 \quad(t\ge1),
\]
and
\[
 \Response_{m,j}^\chi(\bLap_{k})
 :=\frac{\ChebT_j(z_0)}{\ChebT_m(z_0)}
 \ChebU_{m-j}(z_\chi(\bLap_{k})).
\]
Also set $\theta_\chi:=\operatorname{arccosh}z_0
=2\operatorname{artanh}(\chi^{-1/2})$.
For scalar spectral arguments write
\[
 R_{m,j}^\chi(\lambda):=
 \frac{\ChebT_j(z_0)}{\ChebT_m(z_0)}
 \ChebU_{m-j}(z_\chi(\lambda)).
\]

\begin{lemma}[Variation of constants]\label{lem:voc}
Under the hypotheses of Lemma~\ref{lem:moving-window-stability},
\begin{equation}\label{eq:voc}
 \boldsymbol e_m=P_m^\chi(\bLap_{k})\boldsymbol e_0+
 \sum_{j=1}^m\Response_{m,j}^\chi(\bLap_{k})\boldsymbol\xi_{j-1},
\end{equation}
and $\norm{\boldsymbol\xi_t}\le8\norm{\boldsymbol E_t}
\norm{\boldsymbol e_t}$.
\end{lemma}

\begin{lemma}[Response sums]\label{lem:response-bounds}
For $j=1,\ldots,m$ and $\lambda\in[\chi^{-1},1]$,
\begin{equation}\label{eq:response-pointwise}
 \left|\frac{\ChebT_j(z_0)}{\ChebT_m(z_0)}
 \ChebU_{m-j}(z_\chi(\lambda))\right|
 \le2(m-j+1)e^{-(m-j)\theta_\chi}.
\end{equation}
Moreover,
\begin{equation}\label{eq:response-sums}
 \sum_{j=1}^m\sup_{[\chi^{-1},1]}|R_{m,j}^\chi|
 \le2\min\left\{\frac{m(m+1)}2,
 \frac1{(1-e^{-\theta_\chi})^2}\right\},
\end{equation}
and
\begin{equation}\label{eq:response-weighted}
 \sum_{j=1}^m(j-1)\sup_{[\chi^{-1},1]}|R_{m,j}^\chi|
 \le2\min\left\{\frac{m^3}{6},
 \frac{m}{(1-e^{-\theta_\chi})^2}\right\}.
\end{equation}
\end{lemma}

\subsection{Proofs of Lemmas~\ref{lem:voc}, \ref{lem:response-bounds},
and~\ref{lem:moving-window-stability}}

\emph{Identity.} To remove the time-dependent normalization in the
recurrence, define the auxiliary sequence
$\boldsymbol b_t:=\ChebT_t(z_0)\,\boldsymbol e_t$ and set
$\widehat{\boldsymbol Z}:=z_\chi(\bLap_{k})$. From
\eqref{eq:local-window-recurrence} and \eqref{eq:cheb-coeffs-at-glance}, for $t\ge1$,
\begin{equation*}
\begin{aligned}
\boldsymbol b_{t+1}
&=\ChebT_{t+1}(z_0)
  \bigl[a_t z_\chi(\bLap_{k+t})\boldsymbol e_t
  -c_t\boldsymbol e_{t-1}\bigr]
  =2\ChebT_t(z_0)\,z_\chi(\bLap_{k+t})\boldsymbol e_t
  -\ChebT_{t-1}(z_0)\boldsymbol e_{t-1}\\
&=2\widehat{\boldsymbol Z}\boldsymbol b_t-\boldsymbol b_{t-1}
  +\ChebT_{t+1}(z_0)\,\boldsymbol\xi_t.
\end{aligned}
\end{equation*}
because
$z_\chi(\bLap_{k+t})-z_\chi(\bLap_{k})
=-\tfrac{2}{1-\chi^{-1}}\boldsymbol E_t$. Moreover,
\begin{equation*}
\ChebT_{t+1}(z_0)\boldsymbol\xi_t
=-\frac{2a_t\ChebT_{t+1}(z_0)}{1-\chi^{-1}}
  \boldsymbol E_t\boldsymbol e_t
=-\frac{4\ChebT_t(z_0)}{1-\chi^{-1}}\boldsymbol E_t\boldsymbol e_t
=2\ChebT_t(z_0)\bigl(z_\chi(\bLap_{k+t})-\widehat{\boldsymbol Z}\bigr)
  \boldsymbol e_t.
\end{equation*}
Similarly,
$\boldsymbol b_1=z_0\boldsymbol e_1=z_\chi(\bLap_{k})\boldsymbol e_0
=\widehat{\boldsymbol Z}\boldsymbol b_0
+\ChebT_1(z_0)\boldsymbol\xi_0$.
Thus $(\boldsymbol b_t)$ solves the constant-coefficient Chebyshev recurrence
$\boldsymbol b_{t+1}=2\widehat{\boldsymbol Z}\boldsymbol b_t
-\boldsymbol b_{t-1}$ driven by the impulses
$\ChebT_j(z_0)\boldsymbol\xi_{j-1}$ added at times $j=1,\dots,m$. The
homogeneous solution with data $\boldsymbol b_0=\boldsymbol e_0$,
$\boldsymbol b_1=\widehat{\boldsymbol Z}\boldsymbol b_0$ is
$\boldsymbol b_t^{\mathrm h}=\ChebT_t(\widehat{\boldsymbol Z})
\boldsymbol e_0$ (the initial conditions of
$\ChebT_t$). The response at time $m$ to a unit impulse $\boldsymbol w$ added at
time $j$ is $\ChebU_{m-j}(\widehat{\boldsymbol Z})\boldsymbol w$, since the sequence
$\ChebU_{t-j}(\widehat{\boldsymbol Z})\boldsymbol w$ satisfies the homogeneous recurrence for
$t>j$ and takes the values $0$ and $\boldsymbol w$ at times $j-1$ and $j$,
respectively ($\ChebU_0=1$, $\ChebU_1(z)=2z$). Superposition gives
$\boldsymbol b_m=\ChebT_m(\widehat{\boldsymbol Z})\boldsymbol e_0
+\sum_{j=1}^{m}\ChebU_{m-j}(\widehat{\boldsymbol Z})
\ChebT_j(z_0)\boldsymbol\xi_{j-1}$; dividing by $\ChebT_m(z_0)$ yields
\eqref{eq:voc}. The same derivation with horizon $t\le m$ in place of
$m$ gives the partial-horizon identity used in the bootstrap below.

\emph{Norm of the impulses.} For $t\ge1$,
$\norm{\boldsymbol\xi_t}\le\tfrac{2a_t}{1-\chi^{-1}}
\norm{\boldsymbol E_t}\norm{\boldsymbol e_t}
\le\tfrac{4}{1-\chi^{-1}}\norm{\boldsymbol E_t}\norm{\boldsymbol e_t}
\le8\norm{\boldsymbol E_t}\norm{\boldsymbol e_t}$
($a_t\le2$, $\chi^{-1}\le\tfrac14$); for $t=0$,
$\tfrac{2}{(1-\chi^{-1})z_0}\le\tfrac{2}{1-\chi^{-1}}\le4\le8$.

\emph{Bounds on the response operators.}
For every $\lambda\in[\chi^{-1},1]$,
$z_\chi(\lambda)\in[-1,1]$ and, for every integer $r\ge0$,
\[
 |\ChebU_r(\cos\vartheta)|
 =\left|\frac{\sin((r+1)\vartheta)}{\sin\vartheta}\right|
 \le r+1.
\]
At $\vartheta=0$ or $\pi$, the quotient is understood by continuity.
Taking $r=m-j$ and using
\[
 \frac{\ChebT_j(z_0)}{\ChebT_m(z_0)}
 =\frac{\cosh(j\theta_\chi)}{\cosh(m\theta_\chi)}
 \le\frac{e^{j\theta_\chi}}{\tfrac12e^{m\theta_\chi}}
 =2e^{-(m-j)\theta_\chi}
\]
gives \eqref{eq:response-pointwise}. Summing with $r=m-j$,
\begin{multline*}
\sum_{j=1}^{m}\sup_{[\chi^{-1},1]}|R_{m,j}^\chi|
\le2\sum_{r=0}^{m-1}(r+1)e^{-r\theta_\chi}\\
\le2\min\Bigl\{\frac{m(m+1)}{2},\ \sum_{r\ge0}(r+1)x^{r}\Big|_{x=
e^{-\theta_\chi}}\Bigr\}
=2\min\Bigl\{\frac{m(m+1)}{2},\frac1{(1-e^{-\theta_\chi})^{2}}\Bigr\},
\end{multline*}
which is \eqref{eq:response-sums}.  For \eqref{eq:response-weighted}, bounding
$e^{-(m-j)\theta_\chi}\le1$ and using
$\sum_{j=1}^{m}(j-1)(m-j+1)=\sum_{i=1}^{m-1}i(m-i)
=\tfrac{m(m-1)(m+1)}{6}\le\tfrac{m^{3}}{6}$ gives the first branch;
bounding $(j-1)\le m$ and reusing the geometric sum gives the second.

\hfill$\square$

\emph{Proof of Lemma~\ref{lem:moving-window-stability}.}
Apply the partial-horizon form of \eqref{eq:voc} with horizon $t$. The
term with $j=1$ vanishes because $\boldsymbol E_0=0$. For $j\ge2$, the identity
$\ChebT_j(z_0)a_{j-1}=2\ChebT_{j-1}(z_0)$ gives
\[
 \Response_{t,j}^\chi(\bLap_{k})\boldsymbol\xi_{j-1}
 =-\frac{4}{1-\chi^{-1}}\,
 \frac{\ChebT_{j-1}(z_0)}{\ChebT_t(z_0)}\,
 \ChebU_{t-j}\bigl(z_\chi(\bLap_{k})\bigr)
 \boldsymbol E_{j-1}\boldsymbol e_{j-1}.
\]
On the disagreement subspace one has
$\norm{\ChebU_{t-j}(z_\chi(\bLap_{k}))}\le t-j+1$, and
$\norm{\boldsymbol E_{j-1}}\le\beta(j-1)$,
$\norm{\boldsymbol e_{j-1}}\le M_t$. Hence
\begin{equation}\label{eq:stability-sum}
 \norm{\boldsymbol e_t-P_t^\chi(\bLap_{k})\boldsymbol e_0}
 \le\frac{4\beta M_t}{1-\chi^{-1}}
 \sum_{j=2}^{t}\frac{\ChebT_{j-1}(z_0)}{\ChebT_t(z_0)}(j-1)(t-j+1).
\end{equation}

\emph{Short-window branch.} Since
$\ChebT_{i+1}(z_0)=2z_0\ChebT_i(z_0)-\ChebT_{i-1}(z_0)\ge
z_0\ChebT_i(z_0)$ for $i\ge1$ and $\ChebT_1(z_0)=z_0\ChebT_0(z_0)$, the
ratio in \eqref{eq:stability-sum} is at most $z_0^{-1}$ for $j\le t$.
Moreover $(1-\chi^{-1})z_0=1+\chi^{-1}$ and
$\sum_{j=2}^{t}(j-1)(t-j+1)=(t^3-t)/6$, so the right-hand side of
\eqref{eq:stability-sum} is at most
\[
 \frac{4}{1+\chi^{-1}}\cdot\frac{t^3-t}{6}\,\beta M_t
 \le\frac23\beta(t^3-t)M_t.
\]

\emph{Damped branch.} Since
$\ChebT_{j-1}(z_0)/\ChebT_t(z_0)=\cosh((j-1)\theta_\chi)/\cosh(t\theta_\chi)
\le2e^{-(t-j+1)\theta_\chi}$ and $j-1\le t$, the sum in
\eqref{eq:stability-sum} is at most
\[
 2t\sum_{r\ge1}re^{-r\theta_\chi}
 =\frac{t}{2\sinh^2(\theta_\chi/2)}
 =\frac{t(\chi-1)}2,
\]
where $\sinh^2(\theta_\chi/2)=1/(\chi-1)$.
Multiplying by $4\beta M_t/(1-\chi^{-1})$ gives $2\beta t\chi M_t$.
Taking the better of the two estimates gives the asserted perturbation bound.

Because $\norm{P_t^\chi(\bLap_{k})}\le1$ and $D_t\le D_m$ for
$t\le m$, this bound implies
$M_m\le\norm{\boldsymbol e_0}+D_mM_m$.  If $D_m<1$, then
$M_m\le\norm{\boldsymbol e_0}/(1-D_m)$.  Applying
the perturbation bound once more at $t=m$ and using
$\norm{P_m^\chi(\bLap_{k})}\le\rho_m^\chi$ proves
\eqref{eq:scalar-window-contraction}.
\hfill$\square$

\subsection{Proof of Theorem~\ref{thm:metric-window}}

If $m=1$, then
\[
 P_1^\chi(\lambda)=1-\frac{2\lambda}{1+\chi^{-1}},
 \qquad
 \max_{[\chi^{-1},1]}|P_1^\chi|
 =\frac{\chi-1}{\chi+1}
 \le e^{-2/\chi}\le e^{-1/(5\chi)},
\]
since $\ln((\chi+1)/(\chi-1))=2\operatorname{artanh}(\chi^{-1})\ge2/\chi$.
This calculation uses only the current operator and is valid for arbitrary
drift.

Now let $m\ge2$ and let $h=m^2/\chi\in(0,1]$.  The conditions of the
theorem and Lemma~\ref{lem:moving-window-stability} give
\[
 D_m\le\frac23\beta m^3
 =\frac23(\beta m\chi)h\le\frac{2h}{9}\le\frac29.
\]
Therefore \eqref{eq:scalar-window-contraction} and
Lemma~\ref{lem:window} imply
\[
 \norm{\boldsymbol e_m}
 \le\left(e^{-4h/5}+\frac{2h/9}{1-2h/9}\right)\norm{\boldsymbol e_0}
 \le\left(e^{-4h/5}+\frac{2h}{7}\right)\norm{\boldsymbol e_0}.
\]
For $0<h\le1$, write $A=e^{-h/5}$ and $B=e^{-4h/5}$.
The inequality $e^x\ge1+x$ gives $A\ge4/5$ and
$A\ge B(1+3h/5)$. Consequently,
\[
 A-B\ge\frac{3hA}{5+3h}\ge\frac{3h}{10}>\frac{2h}{7}.
\]
Hence \eqref{eq:window-guarantee} holds in both cases.

\emph{Complexity.}
Each complete window earns at least \(m^2/(5\chi)\) nats.  Hence
\[
 J
 =
 \left\lceil
 \frac{5\chi}{m^2}\ln\frac1\eps
 \right\rceil
\]
windows suffice.  Since \(K=Jm\) and
\(\lceil a\rceil m\le am+m\), this gives
the communication bound stated in Theorem~\ref{thm:metric-window}.
\hfill$\square$

\subsection{Proof of Theorem~\ref{thm:main-drift} and
Corollary~\ref{cor:small-drift}}

Let $\widetilde m_\beta,m_\beta$ be given by
\eqref{eq:scale-of-beta}. If $m_\beta\ge2$, then
$m_\beta\le\widetilde m_\beta\le\sqrt\chi$ and
$\beta m_\beta\chi\le1/3$, so Theorem~\ref{thm:metric-window} applies.
If $m_\beta=1$, its degree-one case applies. Thus every complete window
earns $m_\beta^2/(5\chi)$ nats, and the communication bound in
Theorem~\ref{thm:metric-window} holds with
$m=m_\beta$.

If $\widetilde m_\beta\ge2$, then
$m_\beta=\lfloor\widetilde m_\beta\rfloor
\ge\widetilde m_\beta/2$ and
\[
 \frac{\chi}{m_\beta}
 \le2\frac{\chi}{\widetilde m_\beta}
 =2\max\{\sqrt\chi,3\beta\chi^2\}.
\]
Moreover, $\widetilde m_\beta\ge2$ implies
\(\beta\chi^2\le\chi/6<\chi\), so the second term is exactly the
unsaturated quantity in \(\min\{\beta\chi^2,\chi\}\).  Together with
$m_\beta\le\sqrt\chi$, this bounds that communication estimate by an absolute multiple
of the right-hand side of
\eqref{eq:main-drift-bound}.

If $\widetilde m_\beta<2$, then $m_\beta=1$. Because $\chi\ge4$, this can occur only when
$(3\beta\chi)^{-1}<2$, hence
$\beta\chi^2>\chi/6$.  If $\beta\le1/\chi$, then
$\min\{\beta\chi^2,\chi\}=\beta\chi^2>\chi/6$; if
$\beta>1/\chi$, the minimum equals $\chi$.  In either case the
degree-one cost $5\chi\ln(1/\eps)+1$ is bounded by the claimed right-hand
side after enlarging one absolute constant.

Finally, if $\beta\le(3\chi^{3/2})^{-1}$, then the first argument in
the minimum defining $\widetilde m_\beta$ is active, so
$m_\beta=\lfloor\sqrt\chi\rfloor\ge\sqrt\chi/2$. The communication
estimate in Theorem~\ref{thm:metric-window} then gives
$K\le(5\chi/m_\beta)\ln(1/\eps)+m_\beta\le10\sqrt\chi\ln(1/\eps)+\sqrt\chi$,
proving Corollary~\ref{cor:small-drift}.
\hfill$\square$

\section{Piecewise-constant networks}
\label{app:dwell-proofs}

\subsection{Detector timing and scheduler convention}

Before round $k\ge1$, the external detector reports whether
$\Lap_k\ne\Lap_{k-1}$. If it reports a change, the current window
attempt ends before round $k$, and round $k$ is the first round of
a new length-one attempt. A change notification before round $k$ is processed only once: after it resets the attempted length, the new attempt executes round $k$ using $\Lap_k$ and checks for further changes only before subsequent rounds. Thus no communication round is skipped or counted twice. 

An attempt of length $m$ starting at round $k$ completes if no change
is reported before rounds $k+1,\ldots,k+m-1$. A change reported
before round $k+m$ does not erase a window that has already
completed; it resets the next attempted length to one. If no change
is reported, the next attempted length is
$\min\{2m,m_{\max}\}$, where
\[
m_{\max}=\lfloor\sqrt\chi\rfloor,
\qquad
\frac{\sqrt\chi}{2}\le m_{\max}\le\sqrt\chi.
\]

Every window, complete or interrupted, earns the credit
$\ln\ChebT_h(z_0)$, where $h$ is its number of executed rounds. By
Lemma~\ref{lem:window}, every window of length $1\le h\le m_{\max}$ earns at
least $4h^2/(5\chi)$; the bounds below use only completed windows and
discard the nonnegative credit of interrupted ones.

The first $t$ iterations of a window, $t\ge1$, use a single
operator $\Lap_k$. By Proposition~\ref{prop:fixed-window-representation}, they apply
$P_t^\chi(\bLap_k)$ to the stacked state. The map $z_\chi$ in
\eqref{eq:chebyshev-residual} sends $[\chi^{-1},1]$ onto $[-1,1]$, while $z_0>1$.
For $s=\cos\theta\in[-1,1]$ one has
$\ChebT_t(s)=\cos(t\theta)$, whereas
$\ChebT_t(z_0)=\cosh(t\operatorname{arccosh}z_0)\ge1$.
Thus the definition~\eqref{eq:chebyshev-residual} gives
\[
\begin{aligned}
\max_{\lambda\in[\chi^{-1},1]}|P_t^\chi(\lambda)|
&=\frac{\max_{s\in[-1,1]}|\ChebT_t(s)|}{\ChebT_t(z_0)}\\
&=\frac1{\cosh(t\operatorname{arccosh}z_0)}\le1.
\end{aligned}
\]
By the spectral theorem and the uniform spectral bound,
$\norm{\bprojperp P_t^\chi(\bLap_k)\boldsymbol u}
\le\norm{\bprojperp\boldsymbol u}$ for every stacked state
$\boldsymbol u$. The same argument applies to a window interrupted by a
change and to the state observed after any prescribed number of completed
rounds; zero completed rounds leave the state unchanged. Since the norm
after $t$ iterations is at most
$1/\ChebT_t(z_0)=e^{-\ln\ChebT_t(z_0)}$, if $q$ is the progress counter in
Algorithm~\ref{alg:wave}, the disagreement norm at every window boundary
is at most $e^{-q}$ times its value at the beginning of the call.

\subsection{Accumulated progress in a stable segment}

Consider $r\ge1$ consecutive rounds with a fixed operator, beginning
with a length-one attempt. Completed window lengths double until
they reach $m_{\max}$; subsequent completed windows have length
$m_{\max}$. We bound the progress credited within these $r$ rounds.

Let $m$ be the largest window length completed by round $r$.
If $m<m_{\max}$, the completed lengths through $m$ sum to less
than $2m$, and the next attempted length is at most $2m$.
Since that attempt has not completed by round $r$, we have
$r<4m$, and therefore $m\ge(r+1)/4$. If $m=m_{\max}$,
the bound is immediate. Thus
$m\ge\min\{(r+1)/4,m_{\max}\}$. This single completed window
earns at least $4m^2/(5\chi)$, so the total credited progress is at least
\begin{equation}\label{eq:dyadic-segment-progress}
\frac{4m^2}{5\chi}
\ge
\frac{4}{5\chi}
\min\left\{\frac{r+1}{4},m_{\max}\right\}^{2}.
\end{equation}

For $r\ge4m_{\max}$, we also count completed windows of length
$m_{\max}$. The shorter windows preceding the first such window
use
$2^{\lceil\log_2m_{\max}\rceil}-1<2m_{\max}$ rounds.
Consequently, at least
$\lfloor(r-2m_{\max})/m_{\max}\rfloor
\ge r/(5m_{\max})$
windows of length $m_{\max}$ complete by round $r$.
Each earns at least $4m_{\max}^2/(5\chi)$, giving total credited
progress of at least
\begin{equation}\label{eq:long-segment-rate}
\frac{r}{5m_{\max}}
\cdot\frac{4m_{\max}^2}{5\chi}
=
\frac{4rm_{\max}}{25\chi}.
\end{equation}

For $r<4m_{\max}$, \eqref{eq:dyadic-segment-progress} is at least
$r\min\{r,m_{\max}\}/(20\chi)$: if $(r+1)/4\le m_{\max}$, it is at least
$r^2/(20\chi)$; otherwise it is at least $4m_{\max}^2/(5\chi)$, which is at
least $rm_{\max}/(5\chi)$ because $r<4m_{\max}$.
For $r\ge4m_{\max}$,
\eqref{eq:long-segment-rate} gives a stronger bound.
Hence every such segment of $r$ rounds earns at least
\begin{equation}\label{eq:dwell-main-segment-credit}
\frac{r\min\{r,m_{\max}\}}{20\chi}.
\end{equation}

\subsection{Minimum spacing and arbitrary call starts}

\begin{lemma}[Deterministic doubling from an arbitrary start]
\label{lem:dwell-shifted-start}
Under the hypotheses of Theorem~\ref{thm:dwell-det}, let a fresh
\WAVE{} call start at an arbitrary global communication round, reset
its attempted length to one, and thereafter reset at every detected
change. There is an absolute constant $C_{\rm call}$ such that, for
every $\delta\in(0,1)$, its finishing round $k^+$ satisfies
\[
k^+-k\le C_{\rm call}
\left(\sqrt\chi+\frac{\chi}{\tau}\right)
\ln\frac e\delta.
\]
The scheduler does not require prior knowledge of $\tau$.
\end{lemma}

\begin{proof}
Continue the call hypothetically without its stopping test. Let
$q_K$ be the value of the certified progress counter $q$ after its
first $K$ rounds, counting only windows completed by then. Set
\[
a_\tau:=\frac1{80}
\min\left\{\frac{\tau}{\chi},\frac1{m_{\max}}\right\}.
\]
On any complete interval between changes of length $r\ge\tau$,
\eqref{eq:dwell-main-segment-credit} gives an average credited
progress per round of at least
\[
\frac{\min\{r,m_{\max}\}}{20\chi}
\ge
\frac{\min\{\tau,m_{\max}\}}{20\chi}
\ge a_\tau.
\]
The last inequality uses $m_{\max}^2\ge\chi/4$.

A piece of length $r$ at either boundary of the call may be shorter
than $\tau$. If $r<4m_{\max}$, its credited progress is at least
$0\ge a_\tau(r-4m_{\max})$. If $r\ge4m_{\max}$,
\eqref{eq:long-segment-rate} gives credited progress of at least
$4rm_{\max}/(25\chi)\ge a_\tau r$. Thus every boundary piece
contributes at least $a_\tau(r-4m_{\max})$.

Partition the first $K$ call rounds at their change boundaries.
All interior pieces are complete intervals between changes, and at most the
first and last pieces are boundary pieces. If these
coincide, count their overhead only once. Summing the preceding
bounds gives
\[
q_K\ge a_\tau\max\{K-8m_{\max},0\}.
\]
Consequently, the stopping threshold $L=\ln(1/\delta)$ has been
reached by
$K_0=\lceil8m_{\max}+L/a_\tau\rceil$ rounds. Since
$a_\tau^{-1}=80\max\{\chi/\tau,m_{\max}\}$ and
$m_{\max}\le\sqrt\chi$, the claimed estimate follows from
$L+1=\ln(e/\delta)$.
\end{proof}

\begin{proof}[Proof of Theorem~\ref{thm:dwell-det}]
Apply Lemma~\ref{lem:dwell-shifted-start} at $k=0$ with
$\delta=\eps$. This gives~\eqref{eq:dwell-det-bound} after
renaming an absolute constant.
\end{proof}

\subsection{Memoryless change times}

This is a stochastic variant of the piecewise-constant model. The
operator is fixed between reported change points, but the segment lengths
are geometric rather than bounded below by a deterministic $\tau$.

\begin{assumption}[Memoryless change times]
\label{ass:memoryless-changes}
Let $(B_k)_{k\geq0}$ be independent Bernoulli variables, independent of
the past, with $\Prob(B_k=1)=1/\bar\tau$, where $\bar\tau\geq1$. If
$B_k=0$, then $\Lap_{k+1}=\Lap_k$. If $B_k=1$, the detector reports a
new segment boundary and a different normalized gossip matrix satisfying
Assumption~\ref{ass:uniform-spectrum} may be selected as a function of
the past. Thus only the change times are memoryless.
\end{assumption}

\begin{theorem}[Consensus under memoryless change times]
\label{thm:dwell-rand}
Let Assumptions~\ref{ass:uniform-spectrum} and
\ref{ass:memoryless-changes} hold with $\chi\geq4$, and assume access
to the ideal change detector. For $\eps\in(0,1)$, run
Algorithm~\ref{alg:wave} on $(\boldsymbol u_0,0;\eps)$ with
$\mathcal M=\mathrm{piecewise}$, and let
$(\widehat{\boldsymbol u},K)$ be its output. Then $K<\infty$ almost
surely, $\widehat{\boldsymbol u}$ is an $\eps$-consensus solution almost
surely, and
\begin{equation}
\label{eq:dwell-rand-bound}
 \E K
 =O\!\left(
 \left(\sqrt\chi+\frac{\chi}{\bar\tau}\right)
 \ln\frac{e}{\eps}
 \right).
\end{equation}
\end{theorem}

In particular, $\bar\tau\geq\sqrt\chi$ recovers the fixed-network
$O(\sqrt\chi\ln(e/\eps))$ dependence in expectation.

\begin{proof}[Proof of Theorem~\ref{thm:dwell-rand}]
We bound the duration $D_\delta$ of a fresh certified call from any
starting round. Its disagreement hitting time is no larger than
$D_\delta$. For the analysis, continue the scheduler hypothetically after the call terminates, and let $q_K$ be the value of its certified progress counter after $K$ rounds of this continuation. All probabilities and expectations below may be
conditioned on the history at an almost surely finite stopping-time
start. Future change indicators remain independent Bernoulli
variables with probability $1/\bar\tau$; the new call starts with
a length-one attempt.

First, a deterministic bound holds for every switching sequence.
Partition any $K$ call rounds into constant-operator pieces of
lengths $r_i$. Each piece starts with a length-one attempt, so
\eqref{eq:dwell-main-segment-credit} and
$\min\{r_i,m_{\max}\}\ge1$ give
\begin{equation}\label{eq:dyadic-baseline}
q_K
\ge
\sum_i\frac{r_i\min\{r_i,m_{\max}\}}{20\chi}
\ge\frac{K}{20\chi}.
\end{equation}
Hence
$D_\delta\le\lceil20\chi\ln(1/\delta)\rceil
\le C_0\chi\ln(e/\delta)$,
which gives the required order when $1\le\bar\tau<2$.

Suppose $2\le\bar\tau\le4m_{\max}$. Let $T_i$ be successive
constant-operator segment lengths from the call start. By
memorylessness, the $T_i$ are i.i.d. geometric random variables with
mean $\bar\tau$, even conditional on the call-start history. Set
\[
Y_i:=
\frac{4}{5\chi}
\min\left\{\frac{T_i+1}{4},m_{\max}\right\}^{2},
\qquad
\mu_Y:=\mathbb E Y_i,
\qquad
L:=\ln\frac1\delta.
\]
By \eqref{eq:dyadic-segment-progress}, $Y_i$ is a lower bound on
the progress earned in segment $i$, and
$0\le Y_i\le4m_{\max}^2/(5\chi)\le4/5$. Moreover,
\[
\mathbb P(T_i\ge\lceil\bar\tau\rceil)
=(1-1/\bar\tau)^{\lceil\bar\tau\rceil-1}
\ge(1-1/\bar\tau)^{\bar\tau}
\ge\tfrac14.
\]
On this event, $\bar\tau\le4m_{\max}$ implies
$Y_i\ge\bar\tau^2/(20\chi)$. Therefore
\begin{equation}\label{eq:memoryless-short-mean-credit}
\mu_Y\ge\frac{\bar\tau^2}{80\chi}.
\end{equation}

Let $N:=\inf\{n\ge1:\sum_{i=1}^nY_i\ge L\}$.
For each integer $a\ge1$, the event $\{N\ge i\}$ depends only on
$Y_1,\ldots,Y_{i-1}$. Independence and the overshoot bound give
\[
\mu_Y\mathbb E(N\wedge a)
=
\mathbb E\sum_{i=1}^{N\wedge a}Y_i
\le L+\tfrac45.
\]
Monotone convergence and
\eqref{eq:memoryless-short-mean-credit} imply
$\mathbb E N\le80\chi(L+1)/\bar\tau^2<\infty$.
The certified call ends no later than the end of segment $N$.
Because $\{N\ge i\}$ is independent of $T_i$, Tonelli's theorem
gives
\[
\mathbb E D_\delta
\le\mathbb E\sum_{i=1}^N T_i
=\bar\tau\mathbb E N
\le\frac{80\chi}{\bar\tau}(L+1).
\]

Finally, suppose $\bar\tau>4m_{\max}$. Partition future call time
into disjoint blocks of $4m_{\max}$ rounds. Call a block clean if
no change is reported before its second through last rounds.
A change before its first round merely starts a new length-one
attempt there. The clean-block events involve disjoint change
indicators and are i.i.d., with probability
\[
p_{\rm cl}
=(1-1/\bar\tau)^{4m_{\max}-1}
\ge(1-1/\bar\tau)^{\bar\tau}
\ge\tfrac14.
\]
Within a clean block, the attempt in progress at its start finishes
after at most its attempted length. From that attempt through the
next completed window of length $m_{\max}$, the total number of
rounds is at most
$A+m_{\max}<3m_{\max}<4m_{\max}$, where
$A=2^{\lceil\log_2m_{\max}\rceil}-1$.
Thus a window of length $m_{\max}$ completes within the block,
regardless of the attempt state at its start. It earns at least
$b:=4m_{\max}^2/(5\chi)\ge1/5$.

For block $j$, put
$Z_j:=b\mathbf1\{\text{block $j$ is clean}\}$.
The variables $Z_j$ are i.i.d.,
$0\le Z_j\le4/5$, and
$\mathbb E Z_j=bp_{\rm cl}\ge1/20$.
The progress earned by the end of block $n$ is at least
$\sum_{j=1}^nZ_j$. If
$M:=\inf\{n\ge1:\sum_{j=1}^nZ_j\ge L\}$,
the same bounded stopping-time calculation gives
$\mathbb E M\le20(L+1)$. Consequently,
\[
\mathbb E D_\delta
\le4m_{\max}\mathbb E M
\le80m_{\max}\left(\ln\frac1\delta+1\right).
\]

Together with \eqref{eq:dyadic-baseline} and
$m_{\max}\le\sqrt\chi$, these cases give the conditional
$O((\sqrt\chi+\chi/\bar\tau)\ln(e/\delta))$ bound.
In particular, $D_\delta<\infty$ almost surely. The matrices
selected at change points may depend on the past: the argument
uses only the change indicators and progress bounds uniform over
admissible operators. Taking the call start at zero and
$\delta=\eps$ proves~\eqref{eq:dwell-rand-bound}.
\end{proof}

\subsection{Flooding, rollback, and committed outputs}
\label{app:flooding}

Throughout this subsection $n\ge2$, $\chi\ge4$, and
$D:=D_n=\lfloor(\sqrt\chi/2)\ln(2n)\rfloor+1$.
The additional local-observation and message assumptions are those stated
before Proposition~\ref{prop:flooding}. Global round indices and window
schedules are synchronized. Changes to rows, including changes in edge
weights or the disappearance of an edge, are observed before the new row
is used. No agent needs to know which other rows have changed.

\begin{lemma}[Diameter of the operator support]
\label{lem:diameter}
Let $\Lap$ be a normalized gossip matrix satisfying
$\spec(\Lap|_{\ones_n^\perp})\subseteq[\chi^{-1},1]$. Its undirected
nonzero-support graph, with edges $\{i,j\}$ whenever $i\ne j$ and
$\Lap_{ij}\ne0$, is connected and has diameter at most
\[
 \left\lfloor\frac{\operatorname{arccosh}(n)}{\theta_\chi}\right\rfloor+1
 \le D,\qquad
 \theta_\chi:=2\operatorname{artanh}(\chi^{-1/2}).
\]
\end{lemma}

\begin{proof}
We use the polynomial diameter argument of \citet{ChungFaberManteuffel1994}.
If the support graph were disconnected, its component indicators would
give linearly independent kernel vectors, contradicting
$\ker\Lap=\operatorname{span}\{\ones_n\}$.
Put $r:=\lfloor\operatorname{arccosh}(n)/\theta_\chi\rfloor+1$ and
$M:=P_r^\chi(\Lap)$. Since $P_r^\chi(0)=1$, the spectral theorem and
the Chebyshev residual bound give
\[
 M=\proj+M(I_n-\proj),\qquad
 \|M(I_n-\proj)\|\le\frac1{\cosh(r\theta_\chi)}<\frac1n.
\]
Every entry of $M$ is therefore strictly positive. On the other hand,
$(\Lap^j)_{ab}=0$ whenever the support-graph distance from $a$ to $b$
exceeds $j$. As $M$ is a polynomial of degree $r$ in $\Lap$, positivity
forces every such distance to be at most $r$. Finally,
$\theta_\chi\ge2/\sqrt\chi$ and
$\operatorname{arccosh}(n)\le\ln(2n)$ imply $r\le D$.
\end{proof}

\paragraph{Change propagation.}
At a change time $t$, every agent whose row changes attaches $t$ to its
messages starting in round $t$. Agents forward the largest change index
they have learned. They do not alter the numerical schedule merely upon
learning an index. All agents process that index only at the common
boundary $t+D$ and only once. Since $\tau\ge2D$, the operator is fixed
through rounds $t,\ldots,t+D-1$. Its nonzero-support graph is consequently
fixed and is a subgraph of every available communication graph during
these rounds, even if additional unused links vary. Information traverses
one support edge per round, so Lemma~\ref{lem:diameter} ensures that every
agent knows $t$ before boundary $t+D$. No later change can overtake this
propagation. A single timestamp per message therefore suffices.

\paragraph{Checkpoints and synchronized rollback.}
For the numerical computation, the ordinary stopping test of
Algorithm~\ref{alg:wave} is temporarily disabled: window attempts continue
even when their progress counter exceeds $L:=\ln(1/\eps)$.
At each boundary $s$, after crediting any just-completed window and
processing any rollback at that boundary, each agent records its local
state $u_s^{(i)}$ and the common scalar data $(q_s,h_s)$. Here $q_s$ is the
credit of completed windows and $h_s$ is the number of already executed
rounds of the current, unfinished attempt; $h_s=0$ if a window has just
ended or restarted. The most recent $D+1$ checkpoints are retained.
For a change at $t$, the checkpoint at $t$ is taken before any use of
$\Lap_t$.

At boundary $t+D$, every agent restores $u_t^{(i)}$, replaces its counter
by
\[
 q\leftarrow q_t+\ln\ChebT_{h_t}(z_0),
\]
and resets the attempted length to one. The physical round index is not
rolled back. If $h_t=0$, the added credit is zero; in particular, the
credit of a window completed exactly at $t$ is not counted twice.
The restored state is the output after $h_t$ iterations using the old operator
only. The new length-one attempt begins in physical round $t+D$, using
the new operator. Local notifications never cause an earlier or
asynchronous numerical restart. Apart from these common rollbacks, the
ordinary doubling schedule is used.

\paragraph{Output validation.}
All candidate handling occurs at common boundaries, after any scheduled
rollback and its counter correction. At a window-completion or rollback
boundary $b$ with $q\ge L$, if no candidate is pending, the agents save
the local output candidate $v^{(i)}:=u_b^{(i)}$ and its index $b$.
They continue numerical computation, row observation, and timestamp
forwarding during the next $D$ rounds. If a rollback subsequently occurs
for a change time $t<b$, they cancel this candidate; after correcting the
counter, a new candidate may be created at that same boundary if its
counter meets the threshold. A rollback for $t\ge b$ does not cancel the
saved candidate, because no operator from round $t$ or later contributed
to it. At boundary $b+D$, after processing any rollback due there, an
uncancelled candidate is returned by every agent. The saved candidate is
retained separately until cancellation or return. Thus neither an
unconfirmed threshold crossing nor a locally received flag can halt the
communication needed to validate or cancel an output.

\begin{proof}[Proof of Proposition~\ref{prop:flooding}]
First continue the numerical protocol hypothetically forever, ignoring
the return instruction. Delete the intervals of physical rounds
\[
 [t_i,t_i+D),\qquad i\ge1,
\]
from the operator sequence. They are disjoint because $\tau\ge2D$.
The first remaining constant segment has length at least $\tau$, and
each subsequent complete segment has length at least
$\tau-D\ge\tau/2$. At boundary $t_i+D$, the synchronized rollback restores
the state reached with the old operator at $t_i$, credits the corresponding
completed iterations, and restarts
on the new operator. Therefore the states, counters, and schedules on the
retained rounds agree exactly with the ideal-detector scheduler on the
compressed sequence. The temporary states and credits inside a deleted
interval play no role in this retained computation. The argument is
pathwise and does not require the new operators to be independent of the
past.

By Lemma~\ref{lem:dwell-shifted-start}, the ideal-detector run reaches its
stopping threshold after at most
\[
 K'\le C_{\rm call}
       \left(\sqrt\chi+\frac{2\chi}{\tau}\right)\ln\frac e\eps
\]
retained rounds. Let $B$ be the corresponding physical boundary, taking
the boundary after a rollback when a deleted interval is compressed to
the terminal boundary. If $N_B$ intervals have been deleted before $B$,
then $B=K'+DN_B$ (or the same inequality with the displayed upper bound
for $K'$). Minimum spacing gives $N_B\le B/\tau+1$, hence
\[
 B\le K'+\frac D\tau B+D\le K'+\frac B2+D,
 \qquad B\le2K'+2D.
\]

Every false candidate is cancelled: if its state or credit was affected
by an as-yet unprocessed change at $t<b$, the common rollback at $t+D$
occurs after its creation and no later than $b+D$. Conversely, a
candidate not cancelled by $b+D$ contains no such uncorrected change,
and coincides with a threshold-reaching state of the retained
ideal-detector computation. It thus satisfies the claimed consensus
bound. A change at or after $b$ cannot invalidate that saved state.
At the retained boundary $B$, all earlier deleted intervals have already
been corrected. Any candidate still pending there is valid; if none is
pending, the threshold-reaching state at $B$ creates one. Consequently
the actual protocol returns no later than $B+D$, so
\[
 K\le2K'+3D.
\]
Finally $D\le\sqrt\chi\ln n+1$ for $n\ge2$. This proves the stated
round complexity. Checkpoints use $O(Dd)$ real coordinates per agent
plus $O(D)$ scalar scheduling data, with one extra local candidate
vector. Each neighbor message carries one timestamp, requiring
$O(\log(K+1))$ bits up to termination; this is bandwidth overhead, not
an additional exchange round. The $D$ validation exchanges are included
in the bound on $K$.
\end{proof}

\section{Proofs for the optimization results}
\label{app:optimization}

For completeness, we make the logarithmic factors hidden in
Theorem~\ref{thm:opt-main} explicit. Define
\[
\Lambda_\eps
:=
1+\ln_+\frac{G_0^2}{\mu\eps},
\qquad
\Gamma_{\eps,n}
:=
\Lambda_\eps+\ln\kappa
+\ln_+\frac{\sqrt n}{\kappa}.
\]
The proof below gives
\[
N_{\rm grad}
=
O(\sqrt\kappa\,\Lambda_\eps)
\]
local-gradient evaluations per agent and replaces
$\widetilde O(\sqrt\kappa)$ in~\eqref{eq:opt-main} by
$O(\sqrt\kappa\,\Lambda_\eps\Gamma_{\eps,n})$.

\subsection{Complete schedule}
\label{app:optimization-implementation}

Algorithm~\ref{alg:wave-stm} combines restarted Similar Triangles
iterations~\cite{Nesterov2004,GorbunovDanilova2020} with \WAVE{}
communication. For $G_0>0$, set
\begin{equation}\label{eq:optimization-parameters}
\begin{gathered}
 N_0:=\lceil8\sqrt\kappa\rceil,\qquad
 S:=1+\left\lceil\max\left\{0,
       \log_2\frac{G_0}{2\sqrt{\mu\eps}}\right\}\right\rceil,\\
 \delta_\star:=\left[60000\,\kappa
 \left(\kappa+(1+\kappa)
       \max\left\{1,\frac{G_0}{\sqrt{\mu\eps}}\right\}\right)
 \right]^{-1}.
\end{gathered}
\end{equation}
Each \WAVE{} call is Algorithm~\ref{alg:wave}: the global round counter
$k$ continues across calls, while the piecewise-constant doubling schedule
starts anew at length one.

The optimization proof uses the mean-preservation and contraction contract
in~\eqref{eq:wave-call-contract}, together with the following call-cost
bounds. Let $\mathcal F_k$ denote all information available before round
$k$, including the observed change indicators and operators and the current
algorithm state. Under Assumption~\ref{ass:memoryless-changes}, future
change indicators are independent of $\mathcal F_k$.

\begin{corollary}[Certified \WAVE{} call]
\label{cor:wave-call-cost}
Under Assumption~\ref{ass:uniform-spectrum} and the corresponding
variation assumption (Assumption~\ref{ass:metric-drift},
\ref{ass:dwell-time}, or \ref{ass:memoryless-changes}), a fresh call
$(\widehat{\boldsymbol u},k^+)=\WAVE(\boldsymbol u,k;\delta)$,
$\delta\in(0,1)$, run with the corresponding scheduler satisfies the
respective bound below for absolute constants
$C_{\rm m},C_{\rm d},C_{\rm r}>0$:
\begin{align}
 k^+-k
 &\le C_{\rm m}\bigl(\sqrt\chi+\min\{\beta\chi^2,\chi\}\bigr)
       \ln\frac e\delta, &&\text{metric drift},\\
 k^+-k
 &\le C_{\rm d}\bigl(\sqrt\chi+\chi/\tau\bigr)
       \ln\frac e\delta, &&\text{piecewise-constant, minimum spacing},\\
 \mathbb E[k^+-k\mid\mathcal F_k]\label{eq:wave-call-memoryless}
 &\le C_{\rm r}\left(\sqrt\chi+\frac\chi{\bar\tau}\right)
       \ln\frac e\delta, &&\text{memoryless change times}.
\end{align}
The last line holds almost surely for every stopping-time start $k$;
both detector-assisted bounds use the ideal external change detector,
and neither $\tau$ nor $\bar\tau$ is supplied to the scheduler.
\end{corollary}

\begin{algorithm}[ht]
\caption{\WAVESTM{}: restarted STM with \WAVE{} communication}
\label{alg:wave-stm}\label{alg:stm-stage}
\KwIn{common point $x_0$; $G_0$ satisfying
$n^{-1}\sum_i\norm{\nabla f_i(x_0)}^2\leq G_0^2$; $\alpha,\mu$;
target $\eps>0$; a fixed \WAVE{} scheduler}
$\boldsymbol v\leftarrow\ones_n\otimes x_0$; $k\leftarrow0$\;
\lIf{$G_0=0$}{\Return $(\boldsymbol v,k)$}
Set $N_0,S,\delta_\star$ by~\eqref{eq:optimization-parameters}\;
\For{$s=0,\ldots,S-1$}{
  \lIf{$s\ge1$}{$(\boldsymbol v,k)\leftarrow\WAVE(\boldsymbol v,k;1/4)$}
  $\boldsymbol x_0=\boldsymbol y_0=\boldsymbol z_0\leftarrow\boldsymbol v$;
  $A_0\leftarrow0$\;
  \For{$\ell=0,\ldots,N_0-1$}{
    $\omega_{\ell+1}\leftarrow(\ell+2)/(2\alpha)$;
    $A_{\ell+1}\leftarrow A_\ell+\omega_{\ell+1}$\;
    $\boldsymbol x_{\ell+1}\leftarrow
      (A_\ell\boldsymbol y_\ell+\omega_{\ell+1}\boldsymbol z_\ell)
      /A_{\ell+1}$\;
    $g_{\ell+1}^{(i)}\leftarrow\nabla f_i(x_{\ell+1}^{(i)})$
    for all $i$; stack into $\boldsymbol g_{\ell+1}$\;
    $(\boldsymbol z_{\ell+1},k)\leftarrow
      \WAVE(\boldsymbol z_\ell-\omega_{\ell+1}\boldsymbol g_{\ell+1},
            k;\delta_\star)$\;
    $\boldsymbol y_{\ell+1}\leftarrow
      (A_\ell\boldsymbol y_\ell+\omega_{\ell+1}\boldsymbol z_{\ell+1})
      /A_{\ell+1}$\;
  }
  $\boldsymbol v\leftarrow\boldsymbol y_{N_0}$\;
}
\lIf{$\kappa<\sqrt n$}{$(\boldsymbol v,k)\leftarrow
  \WAVE(\boldsymbol v,k;\kappa/\sqrt n)$}
$\boldsymbol y_{\rm out}\leftarrow\boldsymbol v$;
\Return $(\boldsymbol y_{\rm out},k)$\;
\end{algorithm}

\begin{corollary}[\WAVESTM{} under memoryless change times]
\label{cor:opt-memoryless}
Let Assumptions~\ref{ass:uniform-spectrum} and
\ref{ass:memoryless-changes} hold with $\chi\geq4$, and assume access to
the ideal change detector. With $\Lambda_\eps$ and $\Gamma_{\eps,n}$ defined above,
run Algorithm~\ref{alg:wave-stm} with the piecewise-constant \WAVE{}
scheduler. It returns an
$\eps$-solution at every agent almost surely, uses the same gradient
bound, and satisfies
\[
 \E K=O\!\left(\sqrt\kappa\,\Lambda_\eps\Gamma_{\eps,n}
 [\sqrt\chi+\chi/\bar\tau]\right).
\]
\end{corollary}

\subsection{Proof scales and the \WAVE{} interface}
\label{app:optimization-proof-scales}

A stage means one execution of the inner STM loop in
Algorithm~\ref{alg:wave-stm}. Its parameters $N_0,S,\delta_\star$ are
specified in~\eqref{eq:optimization-parameters}. Throughout the proof,
$\kappa=\alpha/\mu\ge1$.
For $G_0>0$, define
\begin{equation}\label{eq:optimization-derived-scales}
 R:=\frac{G_0}{\mu},\qquad H:=(1+\kappa)G_0,\qquad
 r_\eps:=\min\left\{R,\sqrt{\frac{\eps}{\mu}}\right\}.
\end{equation}
The initialization argument below proves
$\norm{x_0-x^\star}\le R$ and
$(n^{-1}\sum_i\norm{\nabla f_i(x^\star)}^2)^{1/2}\le H$.
Set $r_s:=2^{-s}R$ for $s=0,1,\ldots$.
The main-text parameter formulas give the exact identities
\begin{equation}\label{eq:delta-star}
\begin{aligned}
 S&=1+\left\lceil\max\left\{0,
       \log_2\left(R\sqrt{\frac{\mu}{4\eps}}\right)\right\}\right\rceil,\\
 \delta_\star&=\left[60000\,\kappa
       \left(\kappa+\frac{H}{\mu r_\eps}\right)\right]^{-1},
\end{aligned}
\end{equation}
because
$H/(\mu r_\eps)=(1+\kappa)\max\{1,G_0/\sqrt{\mu\eps}\}$.
The quantities $R,H,r_\eps,r_s$ are used only in the analysis.

\paragraph{Communication interface.}
The optimization proof uses exact mean preservation and the contraction
in~\eqref{eq:wave-call-contract}, followed by the call-cost bounds of
Corollary~\ref{cor:wave-call-cost}. Each call is
Algorithm~\ref{alg:wave}: the global communication counter continues
across calls, and each piecewise-constant call starts a fresh doubling
schedule at length one.
Use a common spectral envelope $\chi\ge4$ for the scheduler.
A smaller valid envelope may be enlarged to $4$, changing the stated
complexity bounds only by absolute factors.

\begin{proof}[Proof of Corollary~\ref{cor:wave-call-cost}]
The metric estimate follows from Theorem~\ref{thm:main-drift} after
shifting the starting round; the spectral and drift assumptions are
preserved by this shift. The minimum-spacing estimate for a
call starting anywhere within an interval between changes is
Lemma~\ref{lem:dwell-shifted-start}.

For memoryless change times, use the filtration $\mathcal F_t$ defined
above. For a fixed $\delta\in(0,1)$, start a call at an almost
surely finite stopping time $K$, with an $\mathcal F_K$-measurable input.
Conditional on $\mathcal F_K$, future boundary indicators retain their
original i.i.d.\ law. The certified-progress argument in the proof of
Theorem~\ref{thm:dwell-rand} bounds the duration until the accumulated
certificate reaches $\ln(1/\delta)$, uniformly over the initial
operator and admissible history-dependent choices of subsequent
operators. Applied conditionally, it gives
\[
 \mathbb E[K^+-K\mid\mathcal F_K]
 \le C_{\rm r}\left(\sqrt\chi+\frac{\chi}{\bar\tau}\right)
          \ln\frac e\delta
 \quad\text{almost surely}.
\]
This is a bound on the actual duration of Algorithm~\ref{alg:wave},
and proves~\eqref{eq:wave-call-memoryless}.
Both detector-assisted estimates use the ideal external change detector.
\end{proof}

\subsection{STM with additive gradient errors}

We use the similar-triangles form of the fast gradient method (STM;
see \cite{Nesterov2004,GorbunovDanilova2020} and, for the
inexact-oracle context, \cite{DGN2014,Devolder2013,KornilovEtAl2025JOTA}):
given a convex $\alpha$-smooth $f$ on $\R^{d}$, use estimates
$\tilde g_{\ell+1}=\nabla f(x_{\ell+1})+\theta_{\ell+1}$ with
$\norm{\theta_{\ell+1}}\le\eta$.
The errors may depend on the entire history; no unbiasedness or
independence is assumed. The recursion is
given below; within this subsection, $(z_\ell)$ denotes the STM iterate
sequence and is unrelated to the Chebyshev scalar $z_0=z_\chi(0)$.
\begin{equation}\label{eq:stm}
\begin{aligned}
&\omega_{\ell+1}=\frac{\ell+2}{2\alpha},\qquad
A_{\ell+1}=A_{\ell}+\omega_{\ell+1},\qquad A_0=0,\qquad
x_0=y_0=z_0,\\
&x_{\ell+1}=\frac{A_{\ell}y_{\ell}+\omega_{\ell+1}z_{\ell}}{A_{\ell+1}},\qquad
z_{\ell+1}=z_{\ell}-\omega_{\ell+1}\tilde g_{\ell+1},\qquad
y_{\ell+1}=\frac{A_{\ell}y_{\ell}+\omega_{\ell+1}z_{\ell+1}}{A_{\ell+1}}.
\end{aligned}
\end{equation}
Note $A_N=\sum_{\ell=0}^{N-1}\frac{\ell+2}{2\alpha}=\frac{N(N+3)}{4\alpha}
\ge\frac{N^{2}}{4\alpha}$, $\;\alpha\omega_{\ell+1}^{2}\le A_{\ell+1}$ (equivalent to
$(\ell+2)^{2}\le(\ell+1)(\ell+4)$), and
$W_N:=\sum_{\ell=0}^{N-1}\omega_{\ell+1}^{2}\le\frac{(N+2)^{3}}{12\alpha^{2}}$.

\begin{lemma}[STM with additive noise]\label{lem:stm-noise}
Let $f$ be convex and $\alpha$-smooth with a minimizer $x^\star$,
write $f^\star:=f(x^\star)$ and $R_0:=\norm{x_0-x^\star}$, and let
\eqref{eq:stm} be run with
$\norm{\theta_{\ell+1}}\le\eta$ for all $\ell<N$, $N\ge4$. Then, with
$C_N:=A_N+\alpha W_N\le\frac{(N+2)^{3}}{6\alpha}$,
\begin{itemize}
\item[\textup{(i)}] \textup(trajectory\textup)\ \
$\displaystyle\max_{\ell\le N}\ \norm{z_{\ell}-x^\star}\ \le\
\bar\rho:=R_0+2\eta C_N+\eta\sqrt{2W_N}$, and all points
$x_{\ell},y_{\ell}$ lie in the ball of radius $\bar\rho$ around $x^\star$.
Moreover, this estimate is causal: for every $T\le N$, if the error
bounds are known only for $\ell=0,\ldots,T-1$, then
\begin{equation}\label{eq:stm-causal-trajectory}
 \max_{\ell\le T}\norm{z_\ell-x^\star}
 \le R_0+2\eta C_T+\eta\sqrt{2W_T}
 \le R_0+2\eta C_N+\eta\sqrt{2W_N},
\end{equation}
where $C_T:=A_T+\alpha W_T$.
\item[\textup{(ii)}] \textup(accuracy\textup)\ \
$\displaystyle f(y_N)-f^\star\ \le\
\frac{2\alpha R_0^{2}}{N^{2}}+3\,\eta R_0 N+\frac{6\,\eta^{2}N^{4}}{\alpha}$.
\end{itemize}
\end{lemma}

\begin{proof}
\emph{Master inequality.} Write
$\omega=\omega_{\ell+1}$, $A=A_{\ell+1}$. By $\alpha$-smoothness and
$y_{\ell+1}-x_{\ell+1}=-\frac{\omega^{2}}{A}\tilde g_{\ell+1}$,
\begin{equation*}
\begin{aligned}
f(y_{\ell+1})
&\le f(x_{\ell+1})-\frac{\omega^{2}}{A}
  \inner{\nabla f(x_{\ell+1})}{\tilde g_{\ell+1}}
  +\frac{\alpha\omega^{4}}{2A^{2}}\norm{\tilde g_{\ell+1}}^{2}\\
&\le f(x_{\ell+1})-\frac{\omega^{2}}{2A}
  \norm{\tilde g_{\ell+1}}^{2}
  +\frac{\omega^{2}}{A}
  \inner{\theta_{\ell+1}}{\tilde g_{\ell+1}}.
\end{aligned}
\end{equation*}
using $\inner{\nabla f}{\tilde g}=\norm{\tilde g}^{2}
-\inner{\theta}{\tilde g}$ and $\alpha\omega^{2}\le A$. Multiplying by
$A=A_{\ell}+\omega$ and applying convexity twice,
$A_{\ell}f(x_{\ell+1})\le A_{\ell}f(y_{\ell})+A_{\ell}\inner{\nabla f(x_{\ell+1})}
{x_{\ell+1}-y_{\ell}}$ and $\omega f(x_{\ell+1})\le\omega f(u)
+\omega\inner{\nabla f(x_{\ell+1})}{x_{\ell+1}-u}$ for any $u$, and noting
$A_{\ell}(x_{\ell+1}-y_{\ell})+\omega(x_{\ell+1}-u)
=\omega(z_{\ell}-u)$ by the definition of $x_{\ell+1}$, we get
\begin{equation*}
\begin{aligned}
A_{\ell+1}f(y_{\ell+1})
&\le A_{\ell}f(y_{\ell})+\omega f(u)
  +\omega\inner{\tilde g_{\ell+1}}{z_{\ell}-u}\\
&\quad-\omega\inner{\theta_{\ell+1}}{z_{\ell}-u}
  -\frac{\omega^{2}}{2}\norm{\tilde g_{\ell+1}}^{2}
  +\omega^{2}\inner{\theta_{\ell+1}}{\tilde g_{\ell+1}}.
\end{aligned}
\end{equation*}
Finally, from $z_{\ell+1}=z_{\ell}-\omega\tilde g_{\ell+1}$,
$\omega\inner{\tilde g_{\ell+1}}{z_{\ell}-u}
=\tfrac12\norm{z_{\ell}-u}^{2}-\tfrac12\norm{z_{\ell+1}-u}^{2}
+\tfrac{\omega^{2}}{2}\norm{\tilde g_{\ell+1}}^{2}$; the quadratic terms
cancel, and telescoping over $\ell=0,\dots,N-1$ (with $A_0=0$) gives, for
every $u$,
\begin{equation}\label{eq:master}
\begin{aligned}
A_N\bigl(f(y_N)-f(u)\bigr)
&\le \tfrac12\norm{z_0-u}^{2}
  -\tfrac12\norm{z_N-u}^{2}\\
&\quad-\sum_{\ell=0}^{N-1}\omega_{\ell+1}
  \inner{\theta_{\ell+1}}{z_{\ell}-u}\\
&\quad+\sum_{\ell=0}^{N-1}\omega_{\ell+1}^{2}
  \inner{\theta_{\ell+1}}{\tilde g_{\ell+1}}.
\end{aligned}
\end{equation}

\emph{Trajectory bound.} An easy induction shows
\[
y_{\ell}\in\operatorname{conv}\{z_0,\dots,z_{\ell}\},
\qquad
x_{\ell+1}\in\operatorname{conv}\{z_0,\dots,z_{\ell}\}.
\]
Hence, with $\rho_T:=\max_{\ell\le T}\norm{z_{\ell}-x^\star}$, all queried
points lie within $\rho_{T}$ of $x^\star$. Moreover,
\[
\norm{\nabla f(x_{\ell+1})}
=\norm{\nabla f(x_{\ell+1})-\nabla f(x^\star)}
\le \alpha\rho_{\ell}.
\]
Take $u=x^\star$ in \eqref{eq:master}, drop
$f(y_N)-f^\star\ge0$, and bound
$\inner{\theta}{\tilde g}\le\eta(\alpha\rho+\eta)$ and
$|\inner{\theta}{z_{\ell}-x^\star}|\le\eta\rho$: for every $T\le N$,
\begin{equation*}
\tfrac12\norm{z_T-x^\star}^{2}\ \le\ \tfrac12R_0^{2}
+\eta\,(A_T+\alpha W_T)\,\rho_{T}+\eta^{2}W_T
\ \le\ \tfrac12R_0^{2}+\eta\,C_N\,\rho_{N}+\eta^{2}W_N.
\end{equation*}
Taking $T$ to realize $\rho_N$ (note $\rho_0=R_0$ is included), we get
$\rho_N^{2}\le R_0^{2}+2\eta C_N\rho_N+2\eta^{2}W_N$, whence
$\rho_N\le\eta C_N+\sqrt{\eta^{2}C_N^{2}+R_0^{2}+2\eta^{2}W_N}
\le R_0+2\eta C_N+\eta\sqrt{2W_N}=\bar\rho$, which is (i).
The same argument stopped at any $T\le N$ uses only the errors
$\theta_1,\ldots,\theta_T$ and gives
\eqref{eq:stm-causal-trajectory}; no bound on future oracle errors is
needed.

\emph{Accuracy.} Using \eqref{eq:master} with $u=x^\star$ once more,
now keeping the left side and inserting $\rho_N\le\bar\rho$,
\begin{equation*}
f(y_N)-f^\star\ \le\ \frac{\tfrac12R_0^{2}
+\eta C_N\bar\rho+\eta^{2}W_N}{A_N}
\ \le\ \frac{\tfrac12R_0^{2}+\eta C_NR_0+2\eta^{2}C_N^{2}
+\eta^{2}C_N\sqrt{2W_N}+\eta^{2}W_N}{A_N}.
\end{equation*}
With $A_N\ge N^{2}/(4\alpha)$, $C_N\le(N+2)^{3}/(6\alpha)$,
$W_N\le(N+2)^{3}/(12\alpha^{2})$ and $N\ge4$ (so $N+2\le\tfrac32N$):
$\tfrac12R_0^{2}/A_N\le2\alpha R_0^{2}/N^{2}$;
$\eta C_NR_0/A_N\le\tfrac23\eta R_0(N+2)^{3}/N^{2}
\le\tfrac23(\tfrac32)^{3}\eta R_0N\le3\eta R_0N$ (indeed
$2.25\le3$); and the three $\eta^{2}$ terms are bounded by
$\bigl(\tfrac29(\tfrac32)^{6}
+\tfrac{4}{6\sqrt6}(\tfrac32)^{9/2}
+\tfrac13(\tfrac32)^{3}\bigr)\eta^{2}N^{4}/\alpha
\le(2.6+1.7+1.2)\,\eta^{2}N^{4}/\alpha\le6\eta^{2}N^{4}/\alpha$,
using $N^{2.5},N\le N^{4}$.

\end{proof}

\subsection{Initialization from a single gradient bound}

Because \(f\) is \(\mu\)-strongly convex and differentiable,
\[
 \mu\norm{x_0-x^\star}
 \le\norm{\nabla f(x_0)}
 =\left\|\frac1n\sum_i\nabla f_i(x_0)\right\|
 \le
 \left(\frac1n\sum_i\norm{\nabla f_i(x_0)}^2\right)^{1/2}
 \le G_0.
\]
Hence \(\norm{x_0-x^\star}\le R=G_0/\mu\).  Smoothness of the
individual objectives further gives
\begin{align*}
 \left(\frac1n\sum_i\norm{\nabla f_i(x^\star)}^2\right)^{1/2}
 &\le
 \left(\frac1n\sum_i\norm{\nabla f_i(x_0)}^2\right)^{1/2}
 +\left(\frac1n\sum_i
 \norm{\nabla f_i(x^\star)-\nabla f_i(x_0)}^2\right)^{1/2}\\
 &\le G_0+\alpha\norm{x^\star-x_0}
 \le(1+\kappa)G_0=H.
\end{align*}
This proves the two derived bounds asserted after
\eqref{eq:optimization-derived-scales}.  If \(G_0=0\), then
\(\nabla f(x_0)=0\), and strong convexity implies \(x_0=x^\star\).

\subsection{Mean and disagreement decomposition}

For this proof only, define
\[
 F(\boldsymbol x):=\sum_{i=1}^nf_i(x^{(i)}),
 \qquad
 \boldsymbol x^\star:=\ones_n\otimes x^\star.
\]
For a stacked vector \(\boldsymbol v\), write
\[
 \bar v:=\frac1n(\ones_n^\top\otimes I_d)\boldsymbol v.
\]
For the inner loop of Algorithm~\ref{alg:wave-stm}, exact mean
preservation in~\eqref{eq:wave-call-contract} gives
\[
 \bar z_{\ell+1}=\bar z_\ell-\omega_{\ell+1}
                    \frac1n\sum_i\nabla f_i(x_{\ell+1}^{(i)}).
\]
The averaged $\boldsymbol x$- and $\boldsymbol y$-updates are the same
convex combinations as in~\eqref{eq:stm}. Thus the averaged iterates
follow STM for \(f\), with gradient estimate
\begin{equation}\label{eq:mean-noise}
\begin{aligned}
 \widetilde g_{\ell+1}
 &:=\frac1n\sum_i\nabla f_i(x_{\ell+1}^{(i)}),\\
 \norm{\widetilde g_{\ell+1}-\nabla f(\bar x_{\ell+1})}
 &\le
 \frac{\alpha}{\sqrt n}
 \norm{\bprojperp\boldsymbol x_{\ell+1}}.
\end{aligned}
\end{equation}

The inequality follows by local smoothness, the triangle inequality,
and Cauchy--Schwarz applied to
$n^{-1}\sum_i\alpha\norm{x_{\ell+1}^{(i)}-\bar x_{\ell+1}}$.

\begin{lemma}[Disagreement within one stage]\label{lem:disagreement}
Fix a stage radius \(r>0\) and an error level $\eta\geq0$, let
\(\hat\rho:=r+2\eta C_{N_0}+\eta\sqrt{2W_{N_0}}\), and set
\[
 M_\ell:=
 \max\left\{
 \norm{\bprojperp\boldsymbol x_\ell},
 \norm{\bprojperp\boldsymbol y_\ell},
 \norm{\bprojperp\boldsymbol z_\ell}
 \right\},
 \qquad
 \Delta:=
 \delta_\star\omega_{\max}\sqrt n(\alpha\hat\rho+H),
\]
where
\(\omega_{\max}:=\omega_{N_0}\le5\sqrt\kappa/\alpha\).
Fix $T\le N_0$. Suppose
$\norm{\bar x_{\ell+1}-x^\star}\le\hat\rho$ for $0\le\ell<T$,
$M_0\le\Delta$, and $\delta_\star\le1/(12\sqrt\kappa)$.
Then $M_\ell\le2\Delta$ for $0\le\ell\le T$. In particular, the
one-step bound uses only the current averaged query point.
\end{lemma}

\begin{proof}
The \(\boldsymbol x\)- and \(\boldsymbol y\)-updates are convex
combinations, while \eqref{eq:wave-call-contract} gives
\[
 \norm{\bprojperp\boldsymbol z_{\ell+1}}
 \le\delta_\star\left(
 M_\ell+\omega_{\ell+1}
 \norm{\bprojperp\boldsymbol g_{\ell+1}}\right).
\]
The bound derived above at \(x^\star\) and \(\alpha\)-smoothness imply
\begin{align*}
 \norm{\bprojperp\boldsymbol g_{\ell+1}}
 &\le
 \norm{\nabla F(\boldsymbol x_{\ell+1})
       -\nabla F(\boldsymbol x^\star)}
 +\norm{\bprojperp\nabla F(\boldsymbol x^\star)}\\
 &\le
 \alpha\left(
 M_\ell+\sqrt n\,\norm{\bar x_{\ell+1}-x^\star}\right)
 +\sqrt n H.
\end{align*}
Consequently,
\[
 \norm{\bprojperp\boldsymbol z_{\ell+1}}
 \le
 \delta_\star(1+\omega_{\max}\alpha)M_\ell+\Delta
 \le\frac12M_\ell+\Delta.
\]
Together with the convex-combination updates this yields
\(M_{\ell+1}\le\max\{M_\ell,\frac12M_\ell+\Delta\}\).
Induction from \(M_0\le\Delta\) proves \(M_\ell\le2\Delta\).
\end{proof}

\subsection{Stage induction with one inner accuracy}

\begin{proof}[Proof of Theorem~\ref{thm:opt-main} and
Corollary~\ref{cor:opt-memoryless}]
Assume \(G_0>0\); the zero case was handled above.
First consider any realization on which all calls terminate and satisfy
\eqref{eq:wave-call-contract}. The argument through the local-output
step is deterministic; almost-sure termination in the memoryless model
is established in the cost analysis below.

At the start of stage $s$, we inductively require that the mean of its
input lie within $r_s$ of $x^\star$. This holds at $s=0$ by the
initialization bound; the stage reset preserves the mean.
For stage \(s\), put
\(r=r_s\) and
\[
 \eta_s:=\frac{\mu r}{1200\sqrt\kappa},
 \qquad
 N_0=\lceil8\sqrt\kappa\rceil\le9\sqrt\kappa.
\]
The trajectory radius in Lemma~\ref{lem:stm-noise} satisfies
\begin{equation}\label{eq:rho-stage}
 \hat\rho_s:=
 r+2\eta_sC_{N_0}+\eta_s\sqrt{2W_{N_0}}
 \le1.4r.
\end{equation}
Indeed, using \(N_0+2\le11\sqrt\kappa\) and
\(\alpha=\mu\kappa\), the two additional terms are bounded by
\[
 2\eta_s C_{N_0}\le\frac{1331}{3600}r<0.37r,\qquad
 \eta_s\sqrt{2W_{N_0}}
 \le\frac{11^{3/2}}{1200\sqrt6}\kappa^{-3/4}r<0.02r.
\]

Define
\[
 \omega_{\max}:=\omega_{N_0}\le\frac{5\sqrt\kappa}{\alpha},
 \qquad
 \Delta_s:=
 \delta_\star\omega_{\max}\sqrt n(\alpha\hat\rho_s+H).
\]
At stage zero the input is consensual, so \(M_0=0\).  If \(s\ge1\),
let $M_{\rm in}:=\norm{\bprojperp\boldsymbol v}$ before the
reinitialization call. The preceding stage gives
\(M_{\rm in}\le2\Delta_{s-1}\).  Here
\(r_{s-1}=2r_s\), \(\eta_{s-1}=2\eta_s\), and \(N_0\) is fixed;
hence the definition in \eqref{eq:rho-stage} gives the exact identity
\(\hat\rho_{s-1}=2\hat\rho_s\).  Consequently,
\(\alpha\hat\rho_{s-1}+H
\le2(\alpha\hat\rho_s+H)\), and therefore
\[
 \Delta_{s-1}\le2\Delta_s.
\]
The fixed reinitialization call therefore gives
\(M_0\le\frac14M_{\rm in}\le\Delta_s\).  No stage-zero reset is needed.

We next close the mean and disagreement estimates simultaneously.  The
definition \eqref{eq:delta-star} gives
\[
 \delta_\star\le\frac1{60000\kappa^2}
 \le\frac1{12\sqrt\kappa}.
\]
Suppose causally that the mean iterates through inner iteration \(t\)
lie within \(\hat\rho_s\) and that \(M_j\le2\Delta_s\) for \(j\le t\).
The formula for \(\boldsymbol x_{t+1}\) is a convex combination of
\(\boldsymbol y_t\) and \(\boldsymbol z_t\).  It therefore first gives
\(\norm{\bar x_{t+1}-x^\star}\le\hat\rho_s\) and
\(\norm{\bprojperp\boldsymbol x_{t+1}}\le M_t\), using only already
established iterates.  Equation~\eqref{eq:mean-noise} then gives
\begin{align*}
 \norm{\widetilde g_{t+1}-\nabla f(\bar x_{t+1})}
 &\le2\delta_\star\omega_{\max}\alpha
       (\alpha\hat\rho_s+H)\\
 &\le10\delta_\star\sqrt\kappa(1.4\alpha r+H)\\
 &\le14\delta_\star\sqrt\kappa\,\mu r
       \left(\kappa+\frac{H}{\mu r}\right).
\end{align*}
Every executed stage satisfies $r_s\ge r_\eps$.
Indeed, the final radius is larger than
\(\sqrt{\eps/\mu}\) when more than one stage is used, and the assertion
is immediate in the one-stage case.  Hence the last display and
\eqref{eq:delta-star} imply
\[
 \norm{\widetilde g_{t+1}-\nabla f(\bar x_{t+1})}
 \le\frac{14\mu r}{60000\sqrt\kappa}
 <\eta_s.
\]
The one-step estimate in Lemma~\ref{lem:disagreement} now bounds
\(\boldsymbol z_{t+1}\), and the convex-combination update bounds
\(\boldsymbol y_{t+1}\); hence it advances the disagreement induction to
time \(t+1\).  Only after that step do we apply the causal trajectory
estimate \eqref{eq:stm-causal-trajectory} through iteration
\(t+1\), whose noise bounds have just been established.  This advances
the mean induction in a causal order.  Thus both estimates hold throughout
the stage without assuming future error bounds.

Lemma~\ref{lem:stm-noise}(ii), \(N_0\ge8\sqrt\kappa\), and
\(N_0\le9\sqrt\kappa\) now give
\[
 f(\bar y_{N_0})-f(x^\star)
 \le
 \left(\frac1{32}+\frac{27}{1200}
 +\frac{6\cdot6561}{1200^2}\right)\mu r^2
 \le\frac{\mu r^2}{8}.
\]
Strong convexity implies
\(\norm{\bar y_{N_0}-x^\star}\le r/2\), so the next stage starts
within \(r_{s+1}=r_s/2\).

Let $\boldsymbol v_{\rm fin}$ be the value of $\boldsymbol v$ after
the last stage and before the optional final \WAVE{} call.
Let \(t_\eps:=\sqrt{\eps/\mu}\). The definition of \(S\) ensures
\(r_{S-1}\le2t_\eps\), and the preceding stage estimate therefore gives
\begin{equation}\label{eq:mean-final-gap}
 f(\bar v_{\rm fin})-f(x^\star)
 \le\frac{\mu r_{S-1}^2}{8}\le\frac\eps2.
\end{equation}
It also gives \(r_{S-1}\le2r_\eps\), including the one-stage case.

\paragraph{Final local-output step.}
\label{subsec:local-output-proof}

Before the final call,
\[
 \norm{\bprojperp\boldsymbol v_{\rm fin}}
 \le
 B_{\rm out}:=
 2\delta_\star\omega_{\max}\sqrt n
 (1.4\alpha r_{S-1}+H).
\]
Since \(r_{S-1}\le2r_\eps\),
\[
 B_{\rm out}
 \le
 \frac{28\sqrt n}{60000\,\kappa^{3/2}}\,r_\eps.
\]
Write $\delta_{\rm loc}:=\min\{1,\kappa/\sqrt n\}$ for the effective
contraction of the final step, including the case where it is skipped.
Then the returned vector $\boldsymbol y_{\rm out}$ satisfies
\[
 \norm{\bprojperp\boldsymbol y_{\rm out}}
 \le
 \frac{28}{60000\sqrt\kappa}r_\eps
 <
 \rho_{\rm loc}:=\frac13\sqrt{\frac{\eps}{\alpha}}.
\]
When \(\delta_{\rm loc}=1\), the inequality already holds and the call
is skipped.  Otherwise the final \WAVE{} call establishes it while
preserving the mean. Thus $\bar y_{\rm out}=\bar v_{\rm fin}$ and
\(\norm{y_{\rm out}^{(i)}-\bar y_{\rm out}}\le\rho_{\rm loc}\)
for every agent.  From \eqref{eq:mean-final-gap} and smooth convexity,
\[
 \norm{\nabla f(\bar y_{\rm out})}
 \le\sqrt{2\alpha(f(\bar y_{\rm out})-f(x^\star))}
 \le\sqrt{\alpha\eps}.
\]
Therefore
\begin{align*}
 f(y_{\rm out}^{(i)})-f(x^\star)
 &\le\frac\eps2+
 \sqrt{\alpha\eps}\,\rho_{\rm loc}
 +\frac\alpha2\rho_{\rm loc}^2\\
 &\le\left(\frac12+\frac13+\frac1{18}\right)\eps
 <\eps.
\end{align*}
This proves the local-output guarantees in
Theorem~\ref{thm:opt-main} and Corollary~\ref{cor:opt-memoryless}.

\paragraph{Gradient and communication counts.}

With $\Lambda_\eps,\Gamma_{\eps,n}$ as in Theorem~\ref{thm:opt-main},
the stage schedule gives $S=O(\Lambda_\eps)$.
The exact gradient count per agent is $N_0S$, proving the stated gradient
bound. From~\eqref{eq:optimization-derived-scales},
\[
 \frac{H}{\mu r_\eps}
 =(1+\kappa)\max\left\{1,\frac{G_0}{\sqrt{\mu\eps}}\right\}.
\]
Writing $b=\max\{1,G_0/\sqrt{\mu\eps}\}$, we have
$\kappa+(1+\kappa)b\le3\kappa b$. Hence
\[
 \ln\frac e{\delta_\star}
 \le1+\ln(180000)+2\ln\kappa
       +\frac12\ln_+\frac{G_0^2}{\mu\eps}
 =O(\Lambda_\eps+\ln\kappa).
\]
There are $N_0S$ calls with factor $\delta_\star$, $S-1$ resets with
factor $1/4$, and one final call only if $\kappa<\sqrt n$.
Thus the sum of the call-cost logarithms is bounded by
\[
 O\!\left(
 \sqrt\kappa\,\Lambda_\eps(\Lambda_\eps+\ln\kappa)
 +1+\ln_+\frac{\sqrt n}{\kappa}\right)
 =O(\sqrt\kappa\,\Lambda_\eps\Gamma_{\eps,n}).
\]
Corollary~\ref{cor:wave-call-cost} gives the two deterministic
communication bounds in~\eqref{eq:opt-main}.

For the memoryless model, fix the objectives and the supplied inputs.
The number of scheduled calls is the deterministic integer
$N_{\rm call}=N_0S+(S-1)+\mathbf1_{\{\kappa<\sqrt n\}}$; their requested factors
$\delta_j$ are also deterministic.
Let $K_0=0$ and let $K_j$ be the next unused round at the start of
call $j$, so $K_{N_{\rm call}}$ is the total communication count.
The scheduler makes its decisions from the observed history.
Consequently, each call endpoint is a stopping time.
Inductively, the conditional finite-duration bound
\eqref{eq:wave-call-memoryless} ensures that every $K_j$ is almost
surely finite. The input vector of call $j$ is
$\mathcal F_{K_j}$-measurable, though it need not be deterministic.
For $j=0,\ldots,N_{\rm call}-1$, the tower property gives
\[
 \mathbb E(K_{j+1}-K_j)
 =\mathbb E\!\left[
    \mathbb E(K_{j+1}-K_j\mid\mathcal F_{K_j})\right]
 \le C_{\rm r}\left(\sqrt\chi+\frac{\chi}{\bar\tau}\right)
       \ln\frac e{\delta_j}.
\]
Summing proves the expected communication bound in
Corollary~\ref{cor:opt-memoryless}. All $N_{\rm call}$ calls terminate
and satisfy their
contraction guarantees simultaneously almost surely. The preceding
pathwise proof therefore establishes the local-output guarantee almost
surely as well.
\end{proof}

\subsection{Accelerated gradient tracking with \WAVE{}}
\label{app:accgt}

When every local objective is strongly convex, \WAVE{} can also replace
the mixing operation in accelerated gradient tracking
\citep{LiLin2024JMLR}. We first verify the properties required of the
resulting mixing maps.

\begin{lemma}[A \WAVE{} call is an average-preserving contraction]
\label{lem:wave-linear}
Fix the operator sequence, scheduler, starting round, and, in the
piecewise-constant model, change reports.  A \WAVE{} call is a linear map
$A\otimes I_d$ with
\[
 A\ones_n=\ones_n,\qquad
 \ones_n^\top A=\ones_n^\top,\qquad
 A=\proj+\proj_\perp A\proj_\perp,
 \qquad \norm{A-\proj}\leq\delta,
 \qquad \norm A\leq1.
\]
\end{lemma}

\begin{proof}
The scheduler and the stopping rule depend on the network sequence and
change reports, but not on the input vector.  Every recurrence step is
linear and acts identically on all vector coordinates, so the full call
has the form $A\otimes I_d$.  Every window fixes consensus vectors and
preserves their average; hence $A\proj=\proj A=\proj$ and therefore
$A=\proj+\proj_\perp A\proj_\perp$.  The call contract
\eqref{eq:wave-call-contract} gives
$\norm{A-\proj}=\norm{\proj_\perp A\proj_\perp}\leq\delta$.
The two displayed blocks act on orthogonal subspaces, and thus
$\norm A=\max\{1,\norm{A-\proj}\}=1$.
\end{proof}

\paragraph{Exact packed implementation.}
Let $A_k$ denote the map of the $k$th \WAVE{} call, with target
$\delta=e^{-1}$.  Use the initialization of
\citet[Algorithm~1]{LiLin2024JMLR}, replacing $W^0$ by $A_0$; the
common initialization makes this first mixing operation immaterial.
For $k\geq1$, given
$\boldsymbol x^k,\boldsymbol z^k,\boldsymbol s^{k-1}$ and
$\boldsymbol y^{k-1}$, compute
\[
 \boldsymbol y^k
   =\theta\boldsymbol z^k+(1-\theta)\boldsymbol x^k,
 \qquad
 \boldsymbol g^k
   =\bigl(\nabla f_i(y_i^k)\bigr)_{i=1}^n.
\]
Each agent packs its three local $d$-vectors
\[
 \left(s_i^{k-1},\;
       \frac{\mu\eta}{\theta}y_i^k+z_i^k,\;
       x_i^k\right)
\]
into one $3d$-vector and runs one \WAVE{} call.  After unpacking, write
\[
 \widetilde{\boldsymbol s}^{k}=A_k\boldsymbol s^{k-1},\qquad
 \widetilde{\boldsymbol z}^{k}
   =A_k\!\left(\frac{\mu\eta}{\theta}\boldsymbol y^k
                +\boldsymbol z^k\right),\qquad
 \widetilde{\boldsymbol x}^{k}=A_k\boldsymbol x^k,
\]
where $A_k$ abbreviates $A_k\otimes I_d$ on stacked $d$-vectors.  Then set
\begin{align*}
 \boldsymbol s^k
   &=\widetilde{\boldsymbol s}^{k}
     +\boldsymbol g^k-\boldsymbol g^{k-1},\\
 \boldsymbol z^{k+1}
   &=\frac{\widetilde{\boldsymbol z}^{k}
       -(\eta/\theta)\boldsymbol s^k}
      {1+\mu\eta/\theta},\\
 \boldsymbol x^{k+1}
   &=\theta\boldsymbol z^{k+1}
     +(1-\theta)\widetilde{\boldsymbol x}^{k}.
\end{align*}
These are exactly equations (13a)--(13d) of \citet{LiLin2024JMLR}
with the same virtual matrix $A_k$ in all three mixing locations.  The
three vectors are communicated in parallel, so this is one \WAVE{} call
per outer iteration, with a constant-factor increase in message size.

\begin{corollary}[Accelerated gradient tracking with \WAVE{}]
\label{cor:accgt}
Let Assumption~\ref{ass:uniform-spectrum} and either
Assumption~\ref{ass:metric-drift} or~\ref{ass:dwell-time} hold with
$\chi\geq4$, and let every $f_i$ be $\mu$-strongly convex and
$\alpha$-smooth. Run accelerated gradient tracking
\citep[eqs.~(13a)--(13d), Theorem~3]{LiLin2024JMLR} from a common point,
with $\eta=(1-e^{-1})^3/(4244\alpha)$ and $\theta=\sqrt{\mu\eta}/2$.
Replace its mixing matrix by one \WAVE{} call with target $e^{-1}$,
applied jointly to the three vectors specified above.
There is a constant $C_0>0$, depending on the initialization and first
virtual iterate, such that, writing
$H_\eps:=1+\ln_+(nC_0/\eps)$, every agent obtains an $\eps$-solution
using $O(\sqrt\kappa H_\eps)$ gradient evaluations.
With the ideal detector in the piecewise-constant model, the number of
communication rounds satisfies
\[
 K=
 \begin{cases}
 O\!\left(\sqrt\kappa[\sqrt\chi+\min\{\beta\chi^2,\chi\}]H_\eps\right),
 &\text{metric drift},\\
 O\!\left(\sqrt\kappa[\sqrt\chi+\chi/\tau]H_\eps\right),
 &\text{piecewise-constant changes},
 \end{cases}
\]
\end{corollary}

Fixed networks belong to both models. Thus the fixed-network lower bound
of \citet{ScamanBBLM2017} applies to their gossip oracle model, and
Corollary~\ref{cor:accgt} matches its $\sqrt{\kappa\chi}$ dependence and
single accuracy logarithm whenever $\beta\le(3\chi^{3/2})^{-1}$ or
$\tau\ge\sqrt\chi$.

\begin{proof}[Proof of Corollary~\ref{cor:accgt}]
The strong-convexity result used here is Theorem~3 of
\citet{LiLin2024JMLR}.  Although a virtual \WAVE{} map may be dense and
have signed entries, their proof remains valid for the sequence
$\{A_k\}$.  Indeed, column stochasticity is used to obtain their averaged
recursion (15), while double stochasticity and their estimate (9) are
used to control disagreement.  Lemma~\ref{lem:wave-linear} gives the
corresponding identities and the stronger one-step estimate
\[
 \norm{((A_k-\proj)\otimes I_d)\boldsymbol v}
 \leq e^{-1}\norm{\bprojperp\boldsymbol v}.
\]
The only further mixing estimate in their general $\gamma$-step proof is
their (10) for products of fewer than $\gamma$ matrices.  Here
$\gamma=1$, so it reduces to the identity map.  Entrywise positivity and
one-hop sparsity are therefore not used in the convergence inequalities;
they only describe how one multiplication by $W^k$ is implemented.  In
our construction, $A_k$ is instead implemented by the legal
neighbor-to-neighbor rounds of one \WAVE{} call.  Thus their Theorem~3
applies algebraically with $\gamma=1$, $\sigma_1=e^{-1}$, the displayed
$\eta$, and $\theta=\sqrt{\mu\eta}/2$.

Let $C_0$ be the constant $C$ in that theorem; for $\gamma=1$ it depends
only on the initialization and the first virtual iterate.  For every
integer $T\geq1$, the theorem gives
\begin{align}
 f(\bar x^{T+1})-f(x^\star)
 &\leq (1-\theta)^{T+1}C_0,
 \label{eq:accgt-average-audited}\\
 \frac1n\norm{\bprojperp\boldsymbol x^T}^2
 &\leq (1-\theta)^{T+1}\frac{4C_0}{\alpha}.
 \label{eq:accgt-disagreement-audited}
\end{align}
The two estimates refer to different time indices.  Apply
\eqref{eq:accgt-average-audited} with $T$ and
\eqref{eq:accgt-disagreement-audited} with $T+1$, run one additional
outer iteration, and return the stored local states $x_i^{T+1}$.  Since
$f$ is convex and $\alpha$-smooth,
\[
 f(x_i)-f(x^\star)
 \leq2\bigl(f(\bar x)-f(x^\star)\bigr)
      +\alpha\norm{x_i-\bar x}^2.
\]
Together with
$\norm{x_i^{T+1}-\bar x^{T+1}}^2
 \leq\norm{\bprojperp\boldsymbol x^{T+1}}^2$, this yields
\[
 f(x_i^{T+1})-f(x^\star)
 \leq(2+4n)C_0(1-\theta)^{T+1}
 \qquad (i=1,\ldots,n).
\]
Hence
$T=O(\theta^{-1}[1+\ln_+(nC_0/\eps)])
 =O(\sqrt\kappa[1+\ln_+(nC_0/\eps)])$ outer iterations suffice.
Each iteration evaluates one local gradient and makes one \WAVE{} call.
For $\delta=e^{-1}$, Corollary~\ref{cor:wave-call-cost} bounds the call
cost by
$2C_{\rm m}(\sqrt\chi+\min\{\beta\chi^2,\chi\})$ or
$2C_{\rm d}(\sqrt\chi+\chi/\tau)$, respectively.  Multiplication by the
number of outer iterations proves the two communication bounds.
\end{proof}

\section{Supplementary examples and interpretation}
\label{app:supplementary}

We give a stationary-momentum instability example and relate its
perturbation to relative inexactness. These observations are independent
of the positive complexity proofs.

\subsection{Stationary-momentum resonance example}
A periodic modulation of a quadratic coefficient can amplify the
auxiliary state in a momentum recurrence, much as in a parametrically
excited oscillator. The example below uses the stationary heavy-ball
recurrence obtained from a generic family of limiting Chebyshev
coefficients. This is not the recurrence executed by \WAVE{}, whose
coefficients are fixed by $[\chi^{-1},1]$ and whose momentum is restarted
after every finite window. For this subsection, define
\[
    z_a(\lambda):=\frac{1+a-2\lambda}{1-a},
    \qquad
    q_a:=\frac{1-\sqrt{a}}{1+\sqrt{a}}.
\]
Then $a_t\to2q_a$, $c_t\to q_a^2$, and
\[
2q_a z_a(\Lap)=(1+q_a^2)I-
\frac{4}{(1+\sqrt{a})^2}\Lap,
\]
so the stationary limit is
\begin{equation}\label{eq:hb}
  e_{t+1}=\bigl((1+\gamma)I-h\Lap_t\bigr)e_t-
  \gamma e_{t-1},
  \qquad
  h=\frac{4}{(1+\sqrt{a})^2},\quad
  \gamma=q_a^2.
\end{equation}
For a static matrix with spectrum in $[a,1]$, this recurrence has the
usual asymptotic factor $q_a$.

\begin{proposition}[Resonant divergence of the stationary tuning]
\label{prop:resonance}
Let $n\ge3$, $a\in(0,1/16]$, and $\chi^{-1}\le a/8$. Put
\[
  \lambda_c:=\frac{1+\gamma}{h}=\frac{1+a}{2},
  \qquad
  A:=3\sqrt{a}.
\]
Let $Q\in\mathbb{R}^{n\times(n-1)}$ have orthonormal columns
spanning $\mathbf{1}^{\perp}$, and denote its second column by
$v_{\mathrm{osc}}$. For each phase $s\in\{-1,+1\}$, define
the commuting family
\[
\Lap_t
=
Q\,\diag\bigl(
\chi^{-1},\lambda_t,\lambda_3,\ldots,\lambda_{n-1}
\bigr)Q^{\!\top},
\qquad
\lambda_t
=
\lambda_c\bigl(1+s(-1)^tA\bigr),
\]
where the remaining eigenvalues
$\lambda_j\in[\chi^{-1},1]$ are fixed.

For either phase, the family is admissible and satisfies
\[
  \norm{\Lap_{t+1}-\Lap_t}
  =
  2A\lambda_c
  =
  3(1+a)\sqrt{a}
  \le
  4\sqrt{a}.
\]

Consider the initialization $e_{-1}=0$ and let
\[
  e_0^{\mathrm{osc}}:=\langle v_{\mathrm{osc}},e_0\rangle.
\]
If $e_0^{\mathrm{osc}}\ne0$, then the choice $s=+1$ makes
\eqref{eq:hb} divergent. More precisely, there exists a constant
$C_a>0$, depending only on $a$, such that
\[
  \norm{e_{2t}}
  \ge
  C_a |e_0^{\mathrm{osc}}|(1+\sqrt{a})^t
  \qquad
  \text{for all sufficiently large }t.
\]

If $A>0$ is regarded as the amplitude parameter of the corresponding
scalar period-two recurrence, its exact instability threshold is
\[
  A>\frac{1-\gamma}{1+\gamma}
   =\frac{2\sqrt{a}}{1+a}.
\]
Equivalently, in terms of the operator-drift amplitude
$\beta:=2A\lambda_c$, the threshold is
\[
  \beta>2\sqrt{a}.
\]
\end{proposition}

The parameter $a$ belongs only to this stationary comparator; the
condition $\chi^{-1}\le a/8$ separates its tuning from the full-spectrum
coefficients on $[\chi^{-1},1]$ used by \WAVE{}. The proposition
motivates finite restarts. Its proof follows the comparison below.

\paragraph{Connection with inexact first-order methods.}
On a disagreement coordinate, replacing the reference operator
$\bLap_{k}$ by $\bLap_{k}+E_t$ changes the gradient from $\bLap_{k} e$ to
$(\bLap_{k}+E_t)e$. The resulting structured error is $E_t e$, with
\[
\norm{E_t e}\le\norm{E_t}\norm e.
\]
Equation~\eqref{eq:voc} keeps this operator structure explicit. A
black-box inexact-oracle reduction instead treats the same term as
arbitrary gradient noise and can therefore give a weaker robustness
certificate. This relative-error interpretation connects the analysis
to \cite{DGN2014,Devolder2013,VGDS2023,KornilovEtAl2025JOTA}.

\subsection{Proof of Proposition~\ref{prop:resonance}}

The columns of $Q$ form a fixed orthonormal basis of
$\mathbf{1}^{\perp}$, so \eqref{eq:hb} decouples into independent
scalar recurrences. Let
\[
  e_t^{\mathrm{osc}}:=\langle v_{\mathrm{osc}},e_t\rangle
\]
denote the coordinate in the oscillating eigendirection. With
\[
\lambda_t=\lambda_c\bigl(1+s(-1)^tA\bigr),
\qquad
h\lambda_c=1+\gamma,
\]
the corresponding scalar recurrence is
\begin{equation}\label{eq:scalar-hb}
  e_{t+1}^{\mathrm{osc}}
  =
  -s(-1)^t(1+\gamma)A\,e_t^{\mathrm{osc}}
  -\gamma e_{t-1}^{\mathrm{osc}}.
\end{equation}

For admissibility,
\[
  \lambda_c=\frac{1+a}{2}\in\left[\frac12,\frac{17}{32}\right],
  \qquad
  A=3\sqrt{a}\le\frac34.
\]
Hence
\[
  \lambda_c(1+A)
  \le
  \frac{17}{32}\cdot\frac74
  =
  \frac{119}{128}
  <1,
\]
whereas
\[
  \lambda_c(1-A)
  \ge
  \frac12\cdot\frac14
  =
  \frac18
  \ge\chi^{-1},
\]
because $\chi^{-1}\le a/8\le1/128$. Thus
$\lambda_t\in[\chi^{-1},1]$ for either phase. Moreover, since the eigenbasis
is fixed and only $\lambda_t$ varies,
\[
  \norm{\Lap_{t+1}-\Lap_t}
  =
  |\lambda_{t+1}-\lambda_t|
  =
  2A\lambda_c
  =
  3(1+a)\sqrt{a}
  \le
  4\sqrt{a}.
\]

\emph{Transfer matrix.}
Set
\[
  r_{\rm HB}:=(1+\gamma)A.
\]
Fix the phase $s=+1$ and set
\[
  M_-:=\begin{pmatrix}-r_{\rm HB}&-\gamma\\1&0\end{pmatrix},
  \qquad
  M_+:=\begin{pmatrix}r_{\rm HB}&-\gamma\\1&0\end{pmatrix}.
\]
These are the even- and odd-step transfer matrices of
\eqref{eq:scalar-hb}, respectively. Consequently,
\[
  \begin{pmatrix}
    e_{2k+2}^{\mathrm{osc}}\\
    e_{2k+1}^{\mathrm{osc}}
  \end{pmatrix}
  =T_2
  \begin{pmatrix}
    e_{2k}^{\mathrm{osc}}\\
    e_{2k-1}^{\mathrm{osc}}
  \end{pmatrix},
  \qquad
  T_2:=M_+M_-.
\]
Explicitly,
\begin{equation*}
\begin{aligned}
T_2
&=
\begin{pmatrix}
  -r_{\rm HB}^2-\gamma&-r_{\rm HB}\gamma\\
  -r_{\rm HB}&-\gamma
\end{pmatrix},\\
\det T_2&=\gamma^2,
\qquad
\operatorname{tr}T_2=-(r_{\rm HB}^2+2\gamma).
\end{aligned}
\end{equation*}

The eigenvalues of $T_2$ are
\[
  \lambda_{\mathrm{st}}
  :=-\frac12\left(r_{\rm HB}^2+2\gamma-r_{\rm HB}\sqrt{r_{\rm HB}^2+4\gamma}\right),
  \qquad
  \lambda_{\mathrm{un}}
  :=-\frac12\left(r_{\rm HB}^2+2\gamma+r_{\rm HB}\sqrt{r_{\rm HB}^2+4\gamma}\right).
\]
They are real and negative because $r_{\rm HB},\gamma>0$ and
$(r_{\rm HB}^2+2\gamma)^2-r_{\rm HB}^2(r_{\rm HB}^2+4\gamma)=4\gamma^2>0$.
The eigenvalue of larger modulus is $\lambda_{\mathrm{un}}$, with
\begin{equation}\label{eq:unstable-modulus}
  |\lambda_{\mathrm{un}}|
  =\frac12
  \left(
    r_{\rm HB}^2+2\gamma+r_{\rm HB}\sqrt{r_{\rm HB}^2+4\gamma}
  \right).
\end{equation}

\emph{Exact instability threshold.}
By \eqref{eq:unstable-modulus},
\[
  |\lambda_{\mathrm{un}}|>1
\]
if and only if
\[
  r_{\rm HB}\sqrt{r_{\rm HB}^2+4\gamma}
  >
  2(1-\gamma)-r_{\rm HB}^2.
\]
If the right-hand side is positive, squaring both sides gives
\[
  r_{\rm HB}^2(r_{\rm HB}^2+4\gamma)
  >
  \bigl(2(1-\gamma)-r_{\rm HB}^2\bigr)^2,
\]
which simplifies to
\[
  r_{\rm HB}^2>(1-\gamma)^2.
\]
Since $r_{\rm HB}>0$, this is equivalent to $r_{\rm HB}>1-\gamma$. If the right-hand
side is nonpositive, the displayed inequality is automatic and again
implies $r_{\rm HB}>1-\gamma$. Therefore
\[
  |\lambda_{\mathrm{un}}|>1
  \quad\Longleftrightarrow\quad
  r_{\rm HB}>1-\gamma.
\]
At $r_{\rm HB}=1-\gamma$, one has $|\lambda_{\mathrm{un}}|=1$ exactly.

Since $r_{\rm HB}=(1+\gamma)A$,
\[
  |\lambda_{\mathrm{un}}|>1
  \quad\Longleftrightarrow\quad
  A>\frac{1-\gamma}{1+\gamma}.
\]
Using
\[
  1-\gamma
  =
  \frac{4\sqrt{a}}{(1+\sqrt{a})^2},
  \qquad
  \frac{1-\gamma}{1+\gamma}
  =
  \frac{2\sqrt{a}}{1+a},
\]
we obtain the exact threshold
\[
  A>\frac{2\sqrt{a}}{1+a}.
\]
In terms of the drift amplitude
$\beta=2A\lambda_c$, this is equivalent to
\[
  \beta>2\sqrt{a}.
\]
Our choice $A=3\sqrt{a}$ lies strictly above this threshold.

\emph{Quantitative rate.}
We claim that
\[
  |\lambda_{\mathrm{un}}|\ge1+\sqrt{a}
  \qquad
  \text{for all }a\in\left(0,\frac1{16}\right].
\]
Since
\[
  \sqrt{r_{\rm HB}^2+4\gamma}
  \ge
  2\sqrt\gamma
  =
  2q_a,
\]
equation \eqref{eq:unstable-modulus} gives
\[
  |\lambda_{\mathrm{un}}|
  \ge
  \gamma+r_{\rm HB} q_a+\frac{r_{\rm HB}^2}{2}.
\]
Using
\[
  r_{\rm HB}=3\sqrt{a}(1+\gamma)
   =3\sqrt{a}(1+q_a^2)
\]
and
\[
  1-\gamma
  =
  \frac{4\sqrt{a}}{(1+\sqrt{a})^2},
\]
the desired inequality reduces, after dividing by $\sqrt{a}$, to
\[
  3q_a(1+q_a^2)
  +
  \frac92\sqrt{a}\,(1+q_a^2)^2
  \ge
  1+\frac{4}{(1+\sqrt{a})^2}.
\]
Substitute
\[
  t:=\sqrt{a}\in\left(0,\frac14\right],
  \qquad
  q_a=\frac{1-t}{1+t},
  \qquad
  1+q_a^2=\frac{2(1+t^2)}{(1+t)^2}.
\]
Multiplying by $(1+t)^4>0$, the required inequality becomes
\begin{equation*}
\begin{aligned}
P(t)
&:=
6(1-t^4)
+18t(1+t^2)^2
-(1+t)^4
-4(1+t)^2\\
&=
1+6t-10t^2+32t^3-7t^4+18t^5
\ge0.
\end{aligned}
\end{equation*}
For $t\in[0,1/4]$,
\[
  10t^2\le\frac{10}{4}t,
  \qquad
  7t^4\le\frac74t^3\le32t^3.
\]
Consequently,
\[
  P(t)\ge1+3.5t>0,
\]
which proves $|\lambda_{\mathrm{un}}|\ge1+\sqrt{a}$.

\emph{Divergence of the iterate.}
Since
\[
  \lambda_{\mathrm{st}}\lambda_{\mathrm{un}}
  =
  \det T_2
  =
  \gamma^2,
\]
we have
\[
  |\lambda_{\mathrm{st}}|
  =
  \frac{\gamma^2}{|\lambda_{\mathrm{un}}|}
  <1
  <
  |\lambda_{\mathrm{un}}|.
\]
Thus $T_2$ has one-dimensional stable and unstable eigenlines.

Under the initialization $e_{-1}=0$,
\[
  \begin{pmatrix}
    e_{2t}^{\mathrm{osc}}\\
    e_{2t-1}^{\mathrm{osc}}
  \end{pmatrix}
  =
  e_0^{\mathrm{osc}}\,T_2^t
  \begin{pmatrix}1\\0\end{pmatrix}.
\]
Moreover,
\[
  T_2
  \begin{pmatrix}1\\0\end{pmatrix}
  =
  \begin{pmatrix}
    -r_{\rm HB}^2-\gamma\\
    -r_{\rm HB}
  \end{pmatrix},
\]
which is not collinear with $(1,0)^\top$ because $r_{\rm HB}>0$. Hence
$(1,0)^\top$ does not belong to the stable eigenline of $T_2$.

Let $v_{\mathrm{un}}$ and $v_{\mathrm{st}}$ be eigenvectors
corresponding to $\lambda_{\mathrm{un}}$ and
$\lambda_{\mathrm{st}}$, respectively. We may therefore write
\[
  \begin{pmatrix}1\\0\end{pmatrix}
  =
  a_{\mathrm{un}}v_{\mathrm{un}}
  +
  a_{\mathrm{st}}v_{\mathrm{st}},
  \qquad
  a_{\mathrm{un}}\ne0.
\]
The first coordinate $(v_{\mathrm{un}})_1$ is also nonzero. Indeed, if
$(v_{\mathrm{un}})_1=0$, then the first row of the eigenvalue equation
would give
\[
  -r_{\rm HB}\gamma(v_{\mathrm{un}})_2=0,
\]
forcing $v_{\mathrm{un}}=0$, a contradiction.

It follows that
\[
  e_{2t}^{\mathrm{osc}}
  =
  e_0^{\mathrm{osc}}\left(
    a_{\mathrm{un}}(v_{\mathrm{un}})_1
      \lambda_{\mathrm{un}}^t
    +
    a_{\mathrm{st}}(v_{\mathrm{st}})_1
      \lambda_{\mathrm{st}}^t
  \right).
\]
Since
\[
  \frac{|\lambda_{\mathrm{st}}|}{|\lambda_{\mathrm{un}}|}<1,
\]
the unstable term dominates for all sufficiently large $t$. Therefore
there exists a constant $C_a>0$, depending only on $a$, such that
\[
  |e_{2t}^{\mathrm{osc}}|
  \ge
  C_a |e_0^{\mathrm{osc}}||\lambda_{\mathrm{un}}|^t
  \ge
  C_a |e_0^{\mathrm{osc}}|(1+\sqrt{a})^t
\]
for all sufficiently large $t$.

Finally, since $e_{2t}^{\mathrm{osc}}$ is one orthonormal coordinate of
$e_{2t}$,
\[
  \norm{e_{2t}}
  \ge
  |e_{2t}^{\mathrm{osc}}|.
\]
Thus
\[
  \norm{e_{2t}}
  \ge
  C_a |e_0^{\mathrm{osc}}|(1+\sqrt{a})^t.
\]
\hfill$\square$

\section{Supplementary numerical experiments}
\label{app:supplementary-experiments}

The supplementary material contains the code, fixed seeds, and raw data
needed to reproduce every reported result.

\subsection{Protocol and full statistics for
Figure~\ref{fig:fair-consensus-main}}
\label{app:fair-consensus-protocol}

\paragraph{Methods and performance measure.}
We compare \WAVE{} with the classical Chebyshev semi-iteration
\citep{GolubVarga1961}, the time-varying recurrence of
\citet{MontijanoMontijanoSagues2013}, minimax Richardson iteration, and
gossip. Classical Chebyshev is run continuously without restarts and is
included to isolate the effect of the finite windows and restarts in
\WAVE{}. All methods process the same normalized operators and spectral
bounds. We evaluate the exact operator norm $r_k$, rather than sampling an
initial state. The threshold $10^{-6}$ is used only to report
first-passage rounds and is not supplied to any method.

\paragraph{Metric-drift stress test.}
Panel~\textup{(a)} uses the six-dimensional invariant block of the
weighted-graph family in Appendix~\ref{app:high-drift-family}. We set
$s=400$, so $\chi=s^2=160000$, and
$\beta=120/[s(4s^2+9)]$. Hence $\beta\chi^{3/2}\approx30$, and the window
rule $m=\lfloor\min\{\sqrt\chi,(3\beta\chi)^{-1}\}\rfloor$ gives $m=4$.
The methods run for $1.25\cdot10^6$ rounds, and $r_k$ is recorded every
$2500$ rounds. \WAVE{} reaches $10^{-6}$ after $277500$ rounds, whereas
Montijano and Richardson require $1107500$ rounds and gossip does not
reach the target. Classical Chebyshev leaves the plotted range within
$50000$ rounds.

\paragraph{Piecewise-constant graph pairs.}
For each of $32$ fixed test seeds, we independently sample two connected
random geometric graphs on $n=100$ vertices in the unit square, with
connection radius $\sqrt{1.25\ln(n)/(\pi n)}$. Both Laplacians are divided
by the larger of their two largest eigenvalues, and $\chi$ is computed from
the smallest positive eigenvalue of the normalized pair. Across the test
set, $\chi\in[339,2873]$, with median $905$. The operator
alternates between the two graphs every
$\operatorname{round}(\sqrt\chi)$ rounds. Each change is reported to
\WAVE{} immediately, but the interval between changes is not supplied to
the algorithm. Curves show the pointwise median of $r_k$ over the $32$
pairs; the reported round counts are medians of the per-instance
first-passage times. The Montijano width parameter $q_{\rm M}=0.99$ is
selected on validation instances disjoint from this test set and then held
fixed.

The intersection graph of every sampled pair is disconnected: it has only
$5$ to $24$ edges on $100$ vertices and cannot contain a connected
spanning subgraph. The assumptions of the non-recoverable-links method
(NRL) of \citet{MetelevICML2023} are therefore not satisfied on this
benchmark.

\begin{lemma}[Exact evaluation of the consensus gap]
\label{lem:exact-gap}
Let $\Phi_k$ be the linear map generated after $k$ rounds by any method
above. Since these methods preserve averages and consensus states,
$\bproj\Phi_k=\Phi_k\bproj=\bproj$. Therefore, a run initialized with
$X_0=\bprojperp$ satisfies
\[
 X_k=\bprojperp\Phi_k\bprojperp,
 \qquad
 \norm{X_k}=r_k .
\]
\end{lemma}

\begin{proof}
Average preservation and invariance of consensus give
$\ones_n^\top\Phi_k=\ones_n^\top$ and
$\Phi_k\ones_n=\ones_n$, which are equivalent to
$\bproj\Phi_k=\Phi_k\bproj=\bproj$. Hence
$X_k=\Phi_k\bprojperp=\bprojperp\Phi_k\bprojperp$. For every
$\boldsymbol e_0\in\ones_n^\perp$, we have
$\boldsymbol e_k=X_k\boldsymbol e_0$, and $X_k$ vanishes on
$\operatorname{span}\{\ones_n\}$. Its spectral norm is therefore the
worst-case ratio defining $r_k$.
\end{proof}

\begin{table}[htbp]
\centering
\small
\caption{Paired results for the piecewise-constant experiment in
Figure~\ref{fig:fair-consensus-main}\textup{(b)}. The Chebyshev median is
conditional on its $12$ successful runs. Ratios are baseline rounds divided
by \WAVE{} rounds on successful pairs; brackets give the interquartile
range. A missing first-passage time counts as later than \WAVE{} in the
last column.}
\label{tab:main-piecewise-paired}
\begin{tabular}{@{}lcccc@{}}
\toprule
Method & Successes & Median rounds & Median ratio & \WAVE{} earlier \\
\midrule
\WAVE{} & $32/32$ & $180$ & -- & -- \\
Chebyshev, no restarts & $12/32$ & $516$ & $3.42$ $[2.52,3.56]$ & $32/32$ \\
Montijano & $32/32$ & $344$ & $1.98$ $[1.75,2.18]$ & $32/32$ \\
Richardson & $32/32$ & $344$ & $1.98$ $[1.77,2.18]$ & $32/32$ \\
Gossip & $32/32$ & $422$ & $2.48$ $[2.35,2.66]$ & $32/32$ \\
\bottomrule
\end{tabular}
\end{table}

\subsection{Sensitivity to the interval between changes}
\label{app:piecewise-interval-sensitivity}

We vary the graph density, the number of operators in the sequence, and
the interval $p$ between changes. A random geometric graph on $n=100$
nodes uses radius $\sqrt{c_{\rm r}\ln(n)/(\pi n)}$, with $c_{\rm r}=2$
for RGG and $c_{\rm r}=1.25$ for sparse RGG. Each sequence contains two
or four independently sampled connected graphs, and the operator changes
every $p=\max\{1,\operatorname{round}(c\sqrt\chi)\}$ rounds for
$c\in\{1/8,1/4,1/2,3/4,1,3/2,2,3\}$.

Each configuration uses $32$ fixed test seeds. Montijano's parameter is
selected from $\{0.5,0.7,0.9,0.99\}$ on six disjoint validation
sequences. Classical Chebyshev is never restarted. The round budget is
$\min\{20\chi,60000\}$, and the target is $r_k\leq10^{-8}$. \WAVE{} has
the smallest median in $30$ of the $32$ settings, ties once, and is within
$1\%$ of the smallest median in the remaining setting. All four
switching-stable methods reach the target in every run, whereas Chebyshev
without restarts frequently fails when changes are rapid; the success
counts are given in Table~\ref{tab:chebyshev-success}.

Because independently sampled graphs can have complementary edges, more
frequent changes can improve the observed mixing rate. This sweep is
therefore a robustness comparison across change frequencies, rather than
an empirical equality for the worst-case term $\chi/\tau$ in
Theorem~\ref{thm:dwell-det}.

\begin{figure}[htbp]
\centering
\includegraphics[width=\textwidth]
{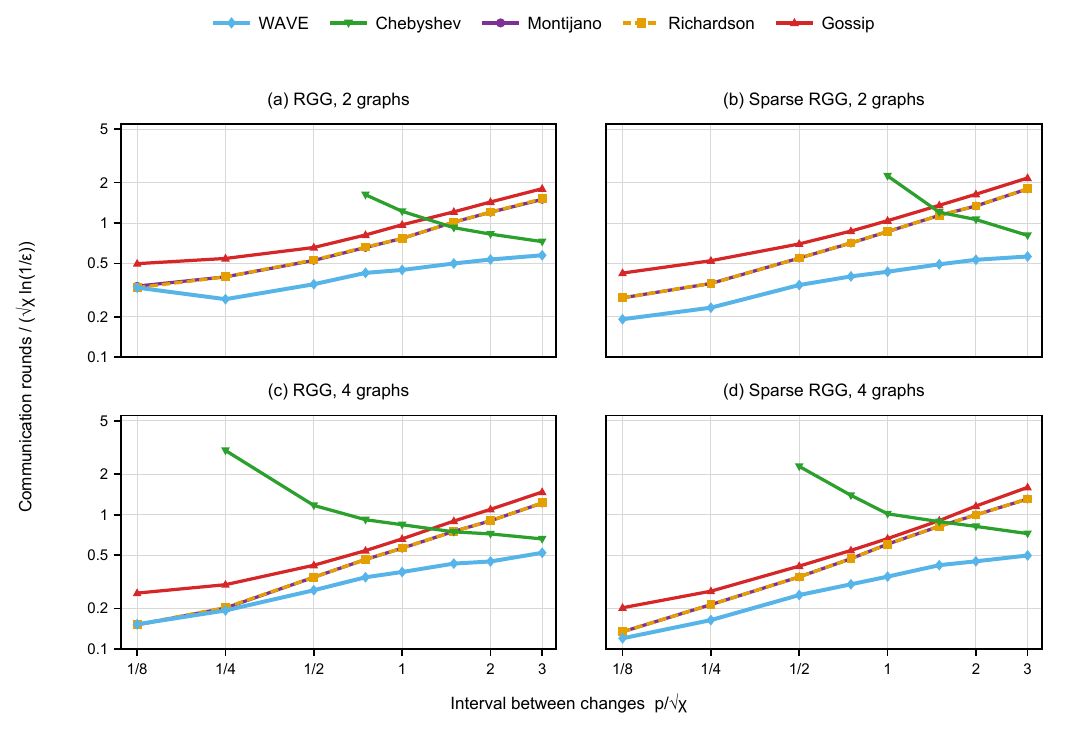}
\caption{Median rounds to reach $r_k\leq10^{-8}$, normalized by
$\sqrt\chi\ln(1/\eps)$, over $32$ independently generated sequences.
Classical Chebyshev is run without restarts; its curve is omitted when
fewer than $17$ runs reach the target.}
\label{fig:piecewise-interval-sensitivity}
\end{figure}

\begin{table}[htbp]
\centering
\small
\setlength{\tabcolsep}{4pt}
\caption{Number of the $32$ runs in which classical Chebyshev without
restarts reaches $r_k\leq10^{-8}$ within the round budget. \WAVE{},
Montijano, Richardson, and gossip succeed in all runs.}
\label{tab:chebyshev-success}
\begin{tabular}{@{}lrrrrrrrr@{}}
\toprule
& \multicolumn{8}{c}{$p/\sqrt\chi$}\\
Graphs & $1/8$ & $1/4$ & $1/2$ & $3/4$ & $1$ & $3/2$ & $2$ & $3$\\
\midrule
RGG, 2 graphs        & 0 & 0 & 16 & 29 & 30 & 32 & 32 & 32\\
RGG, 4 graphs        & 0 & 25 & 31 & 32 & 32 & 32 & 32 & 32\\
Sparse RGG, 2 graphs & 0 & 0 & 6  & 13 & 22 & 31 & 32 & 32\\
Sparse RGG, 4 graphs & 0 & 0 & 25 & 32 & 32 & 32 & 32 & 32\\
\bottomrule
\end{tabular}
\end{table}

\FloatBarrier

\subsection{Square-root scaling in \texorpdfstring{$\chi$}{chi}}
\label{app:sqrt-chi-scaling}

The comparison above covers a limited range of condition numbers.
Corollary~\ref{cor:small-drift} makes a sharper prediction: when
$\beta\leq(3\chi^{3/2})^{-1}$, the number of rounds grows as
$\sqrt\chi$, as for a fixed network. We test this prediction over a
$64$-fold range of $\chi$.

\textbf{Design.}
We draw one support graph from $G(40,0.2)$ and accept it only when the
graph and both halves of a fixed balanced cut are connected; the accepted
support has $166$ edges. Edges inside the halves have base weight $1$, and
all cut edges share a base weight $w$. For each
$\chi\in\{25,50,100,200,400,800,1600\}$ we form $16$ endpoint pairs by
multiplying every base weight by $1+0.18s_e$ in $\Lap_A$ and by
$1-0.18s_e$ in $\Lap_B$, with independent random signs
$s_e\in\{\pm1\}$, and tune $w$ by bisection so that the pair has condition
bound exactly $\chi$ under the normalization of
Appendix~\ref{app:fair-consensus-protocol}. The two endpoints do not
commute. The operator moves between them as in the metric-drift regime,
with $\beta=(400\chi^{3/2})^{-1}$, and \WAVE{} uses windows of length
$\lfloor\sqrt\chi\rfloor$. As a control, the same method is run on the
fixed operator at the starting point of the path.

\textbf{Result.}
The median number of rounds to reach $r_k\leq10^{-6}$ grows from $51$ at
$\chi=25$ to $421$ at $\chi=1600$
(Figure~\ref{fig:wave-sqrt-chi-scaling}). A least-squares fit of
$K=C\chi^p$ to the seven medians gives $p=0.505$ with $R^2=0.9998$. The
medians under metric drift and on the fixed network coincide at every
$\chi$. Thus, at the drift scale of Corollary~\ref{cor:small-drift}, the
observed communication counts retain the fixed-network $\sqrt\chi$
dependence.

\begin{figure}[H]
\centering
\includegraphics[width=\textwidth]
{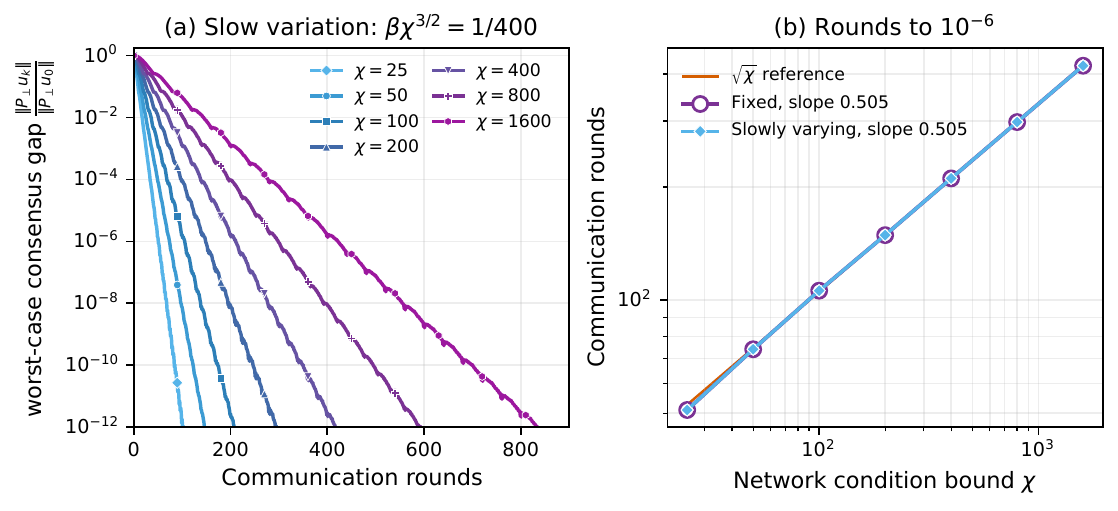}
\caption{Condition-number scaling of \WAVE{} at
$\beta=(400\chi^{3/2})^{-1}$. \textup{(a)} Median worst-case consensus gap
$r_k$ over $16$ noncommuting endpoint pairs for each $\chi$.
\textup{(b)} Median number of rounds to reach $r_k\leq10^{-6}$ for the
metric-drift runs and the fixed-network control; the two coincide, and
both fitted slopes equal $0.505$.}
\label{fig:wave-sqrt-chi-scaling}
\end{figure}

\subsection{Controlled high-drift family}
\label{app:high-drift-family}

We describe the weighted-graph family used in
Figure~\ref{fig:fair-consensus-main}\textup{(a)}. This is a controlled
stress test of the metric-drift bound rather than a random graph model.
Fix $s\geq400$, let $a=s^{-2}$ and $\chi=s^2$, and take $b=3$. Define
$J_b:=\ones_b\ones_b^\top/b$, $P_b:=I_b-J_b$, and, for
$j=0,\ldots,b-1$, set
$u_j:=\sqrt{2/b}\cos(2\pi j/b)$ and
$v_j:=\sqrt{2/b}\sin(2\pi j/b)$. Let
\[
 \theta:=2\arctan\frac{3}{2s},\qquad
 z_k:=\sin(k\theta)u+\cos(k\theta)v.
\]
For $P_4:=I_4-\ones_4\ones_4^\top/4$, set
$h_1=(1,1,-1,-1)^\top/2$,
$h_2=(1,-1,1,-1)^\top/2$, and
$w_k:=\cos(\pi k/(2s))h_1+\sin(\pi k/(2s))h_2$. The normalized
Laplacian on four groups of $b$ nodes is
\begin{equation}
\label{eq:high-drift-family}
 \Lap_k=
 I_4\otimes(P_b-10a z_kz_k^\top)
 +a(P_4+w_kw_k^\top)\otimes J_b.
\end{equation}
An entrywise calculation shows that \eqref{eq:high-drift-family} is the
Laplacian of a connected weighted graph. Its spectrum consists of $0$,
$a$ twice, $2a$, $1-10a$ four times, and $1$ four times. Hence the
condition bound is $\chi=s^2$. Moreover,
\[
 \norm{\Lap_{k+1}-\Lap_k}
 =\frac{120}{s(4s^2+9)}=:\beta,
 \qquad
 \beta\chi^{3/2}=\frac{120s^2}{4s^2+9}\approx30.
\]

The disagreement dynamics splits into repeated invariant blocks. One
representative block, together with the group-mean block, is the
six-dimensional matrix used by the code:
\[
 B_k=\diag(I_2,aI_2,a,1)
 -10a\,(z_kz_k^\top\oplus0_4)
 +a\,(0_2\oplus \widetilde w_k\widetilde w_k^\top\oplus0_2),
\]
where
$\widetilde w_k=(\cos(\pi k/(2s)),\sin(\pi k/(2s)))^\top$.
The repeated blocks do not change the spectral norm, so propagating $B_k$
gives the exact worst-case consensus gap of the full graph family. The
current window rule selects
\[
 m_s=\left\lfloor(3\beta\chi)^{-1}\right\rfloor
 =\left\lfloor\frac{4s^2+9}{360s}\right\rfloor.
\]
For Figure~\ref{fig:fair-consensus-main}\textup{(a)}, $s=400$ and
$m_s=4$. All five methods are applied to the same block sequence. The
reported values are exact up to the sampling interval of $2500$ rounds.
Classical Chebyshev is run without restarts and exceeds a gap of $10^6$
within $50000$ rounds, whereas \WAVE{} reaches $r_k\leq10^{-6}$ after
$277500$ rounds.

\end{document}